\documentclass[12pt,a4paper,twoside]{article}
\usepackage[hmarginratio=1:1,bmargin=1.3in]{geometry}
\usepackage[frenchb]{babel}
\usepackage[utf8]{inputenc}
\usepackage[OT2,T1]{fontenc}
\usepackage{amsmath, amsthm}
\usepackage{amsfonts}
\usepackage{amssymb}
\usepackage[all]{xy}
\usepackage{graphicx}
\usepackage[nottoc, notlof, notlot]{tocbibind}
\usepackage{array}
\usepackage{url}
\usepackage{rotating}
\usepackage{fancyhdr}
\usepackage{fancyhdr}
\usepackage{titlesec}
\titleformat{\section}[hang]{\center\Large\bf}{\thesection.}{0.5cm}{}

\DeclareSymbolFont{cyrletters}{OT2}{wncyr}{m}{n}
\DeclareMathSymbol{\Sha}{\mathalpha}{cyrletters}{"58}
\DeclareMathSymbol{\Brusse}{\mathalpha}{cyrletters}{"42}
\addto\captionsfrenchb{}

\theoremstyle{plain}
\newtheorem{theorem}{Th\'eor\`eme}[section]
\newtheorem{lemma}[theorem]{Lemme}
\newtheorem{proposition}[theorem]{Proposition}
\newtheorem{corollary}[theorem]{Corollaire}
\newtheorem{definition}[theorem]{D\'efinition}

\theoremstyle{definition}
\newtheorem{remarque}[theorem]{Remarque}

\newtheorem{example}[theorem]{Exemple}
\newtheorem{notation}[theorem]{Notation}

\newtheoremstyle{hypo}  
  {\topsep}   
  {\topsep}   
  {\itshape}  
  {1.5ex}       
  {\bfseries} 
  {)}         
  {8pt plus 1pt minus 1pt}  
  {}          
\theoremstyle{hypo}
\newtheorem{hypo}[theorem]{(H}

\newtheoremstyle{hypoth}  
  {\topsep}   
  {\topsep}   
  {\itshape}  
  {1.5ex}       
  {\bfseries} 
  {)}         
  {8pt plus 1pt minus 1pt}  
  {}          
\theoremstyle{hypoth}
\newtheorem{hypoth}{(H}[theorem]

\begin{document}

\title{\textbf{\scshape Théorèmes de dualité pour les corps de fonctions sur des corps locaux supérieurs et applications arithmétiques}}
\author{Diego Izquierdo\\
\small
Département de mathématiques, Université Paris-Sud\\
\small
Bâtiment 430 - 91409 Orsay - France\\
\small
\texttt{diego.izquierdo@math.u-psud.fr}\\
\small
École Normale Supérieure\\
\small
45, Rue d'Ulm - 75005 Paris - France\\
\small
\texttt{diego.izquierdo@ens.fr}
}
\date{}
\normalsize
\maketitle
\setcounter{section}{-1}

\footnotesize

\textbf{Résumé.} Soit $K$ le corps des fonctions d'une courbe projective lisse $X$ sur un corps local supérieur $k$. On définit les groupes de Tate-Shafarevich d'un schéma en groupes commutatif en considérant les classes de cohomologie qui deviennent triviales sur chaque complété de $K$ provenant d'un point fermé de $X$. On établit des théorèmes de dualité arithmétique entre des groupes de Tate-Shafarevich pour les modules finis, pour les tores, pour les groupes de type multiplicatif, et même pour les complexes à deux termes de tores. On applique ces résultats à l'approximation faible pour les tores sur $K$ et à l'étude du principe local-global pour les $K$-torseurs sous un groupe linéaire connexe. On exhibe aussi des exemples et des contre-exemples au principe local-global pour les algèbres simples centrales sur $K$.\\
\vspace{20pt}

\textbf{Abstract.} Let $K$ be the function field of a smooth projective curve $X$ over a higher-dimensional local field $k$. We define Tate-Shafarevich groups of a commutative group scheme via cohomology classes locally trivial at each completion of $K$ coming from a closed point of $X$. We establish duality theorems between Tate-Shafarevich groups for finite groups schemes, for tori, for groups of multiplicative type, and even for 2-term complexes of tori. We apply these results to the weak approximation for tori over $K$ and to the study of the obstruction to the local-global principle for $K$-torsors under a connected linear algebraic group. We also give examples and counter-examples to the local-global principle for central simple algebras over $K$.

\normalsize

\newpage

\section{\scshape Introduction}

\subsection{Contexte et motivation}

\hspace{4ex} Les premiers théorèmes de dualité arithmétique portant sur la cohomologie galoisienne de certains corps ont été annoncés dans les années 1960 par John Tate. Ces résultats, qui se sont avérés depuis particulièrement utiles pour étudier de profonds problèmes arithmétiques comme le principe local-global ou l'approximation faible, concernaient la cohomologie de corps de petite dimension cohomologique ayant de fortes propriétés arithmétiques: les corps $p$-adiques et les corps de séries de Laurent à coefficients dans les corps finis dans le cadre local, les corps de nombres et les corps de fonctions de courbes projectives lisses sur des corps finis dans le cadre global. Ces théorèmes ont été assez facilement généralisés dans le cas local à certains corps de dimension cohomologique quelconque (finie), les corps locaux supérieurs. Par contre, une telle généralisation s'avère nettement plus difficile dans le cas des corps globaux et ce n'est que très récemment, à partir de l'article \cite{SVH} de Scheiderer et Van Hamel, que nous avons été témoins d'un regain d'intérêt pour des études dans cette direction.\\

\hspace{4ex} Deux grandes méthodes se sont développées ces dernières années afin de comprendre le principe local-global et l'approximation faible sur les corps globaux de dimension cohomologique quelconque. D'une part, dans les articles \cite{HS1} et \cite{HS2}, Harari, Scheiderer et Szamuely étudient certaines obstructions cohomologiques au principe de Hasse (théorèmes 5.1 et 6.1 de \cite{HS1}) et à l'approximation faible (théorèmes 3.3 et 4.2 de \cite{HS2}) pour le corps des fonctions d'une courbe projective lisse sur un corps $p$-adique, c'est-à-dire un corps de dimension cohomologique 3, en établissant préalablement des théorèmes de dualité arithmétique type Poitou-Tate pour les tores (théorèmes 4.1 de \cite{HS1} et 2.9 de \cite{HS2}). Cette méthode a ensuite été utilisée par Colliot-Thélène et Harari dans \cite{CTH} pour étudier les corps de fonctions sur $\mathbb{C}((t))$. D'autre part, dans une série d'articles parmi lesquels nous pouvons notamment citer \cite{HHK}, Harbater, Hartmann et Krashen ont développé une technique, dite du patching, qui leur permet d'étudier le principe de Hasse pour le corps des fonctions d'une courbe sur un corps à valuation discrète complet possédant un modèle projectif, intègre et normal. Cette technique a par la suite été utilisée par Colliot-Thélène, Parimala et Suresh pour établir le principe local-global pour l'isotropie des formes quadratiques (théorème 3.1 de \cite{CPS}) et pour les espaces homogènes sous certains groupes réductifs (théorème 4.3 de \cite{CPS}). Une différence essentielle distingue les résultats obtenus par les deux méthodes précédentes: s'il est vrai que dans les deux cas on étudie le corps des fonctions d'une courbe projective lisse sur un corps local, la première méthode tient uniquement compte des places provenant d'un point de codimension 1 de la courbe, alors que la deuxième tient compte de toutes les places provenant d'un point de codimension 1 d'un modèle entier de la courbe.\\

\hspace{4ex} Dans cet article, nous allons utiliser la méthode développée par Harari, Scheiderer et Szamuely pour généraliser leurs résultats à des corps de fonctions de courbes projectives lisses sur un corps local supérieur, ce qui fournit un cadre unifié permettant de traiter simultanément les corps de fonctions sur un corps fini, un corps $p$-adique ou $\mathbb{C}((t))$. L'étude de tels corps pose un certain nombre de difficultés supplémentaires et permet de mettre en évidence des phénomènes nouveaux. D'une part, elle demande à gérer la dimension cohomologique quelconque, alors que les travaux de Harari, Scheiderer et Szamuely ainsi que les travaux postérieurs de Colliot-Thélène et Harari concernent uniquement des corps de dimension cohomologique au plus 3. Rappelons que, d'un point de vue cohomologique, l'objet qui joue le rôle du dual d'un tore est son module des caractères en dimension cohomologique 2 et son tore dual en dimension cohomologique 3. Nous verrons que la situation est nettement plus compliquée à partir de la dimension cohomologique 4, puisque le dual d'un tore n'est plus un faisceau mais un complexe de faisceaux défini à partir des complexes de Bloch. Nous devrons donc faire appel à certains théorèmes difficiles et récents sur ces complexes, en particulier la conjecture de Beilinson-Lichtenbaum qui découle de la conjecture de Bloch-Kato. D'autre part, contrairement au cas étudié par Harari, Scheiderer et Szamuely, pour les corps de fonctions sur des corps locaux supérieurs, il n'y a pas, en général, d'annulation du deuxième groupe de Tate-Shafarevich de $\mathbb{G}_m$, et nous avons donc affaire à des groupes de Tate-Shafarevich infinis. \\

\hspace{4ex} Pour terminer ces généralités, mentionnons qu'à la fin de l'article, nous étendrons aussi certains théorèmes de dualité arithmétique de \cite{HS1} aux groupes de type multiplicatif et même aux complexes de deux tores, et insistons sur le fait que, comme dans les articles de Harari, Scheiderer et Szamuely et contrairement à ceux de Harbater, Hartmann et Krashen, tous les résultats que nous obtenons tiennent uniquement compte des places provenant d'un point de codimension 1 de la courbe projective lisse considérée. 

\subsection{Organisation de l'article et énoncés des théorèmes principaux}

\hspace{4ex} L'article est constitué de 7 parties. Les deux premières sections présentent des résultats préliminaires: la première permet d'établir notamment le lemme \ref{préliminaire} qui est utile dans toute la suite afin de construire des accouplements à valeurs dans $\mathbb{Q}/\mathbb{Z}$, et la deuxième est consacrée à des théorèmes de dualité arithmétique de type Poitou-Tate pour les modules finis (théorèmes \ref{PT fini} et \ref{PT fini suite}).\\

\hspace{4ex} Les théorèmes principaux sont établis dans les cinq parties suivantes. Dans la section 3, nous établissons des théorèmes de dualité arithmétique pour les tores en nous ramenant aux théorèmes analogues pour les modules finis:

\begin{theorem} (théorèmes \ref{PT tore}, \ref{PT tore exacte} et \ref{PT dual tore exacte})\\ \label{0}
Soit $d \geq 0$. Soit $k$ un corps $d$-local, c'est-à-dire un corps complet pour une valuation discrète dont le corps résiduel est $(d-1)$-local, les corps 0-locaux étant par définition les corps finis et $\mathbb{C}((t))$. On suppose que le corps $1$-local correspondant est de caractéristique 0. Soit $X$ une courbe projective lisse sur $k$. Soient $X^{(1)}$ l'ensemble de ses points fermés, $K$ son corps des fonctions et $T$ un $K$-tore. On note $\hat{T}$ le module des caractères de $T$, $\check{T}$ le module des cocaractères de $T$ et $\tilde{T} = \hat{T} \otimes \mathbb{Z}(d)$, où $\mathbb{Z}(d)$ est le $d$-ième complexe motivique.
\begin{itemize}
\item[(i)] On a alors des accouplements parfaits de groupes finis:
$$\Sha^1(T) \times \overline{\Sha^{d+2}( \tilde{T})} \rightarrow \mathbb{Q}/\mathbb{Z},$$
$$\Sha^{d+1}(\tilde{T}) \times \overline{\Sha^2( T)} \rightarrow \mathbb{Q}/\mathbb{Z},$$
où $\overline{A}$ désigne le quotient de $A$ par son sous-groupe divisible maximal pour chaque groupe abélien $A$ et $\Sha^r(T)$ (resp. $\Sha^r(\tilde{T})$) est le sous-groupe de $H^r(K,T)$ (resp. $H^r(K,\tilde{T})$) constitué des éléments dont la restriction à $H^r(K_v,T)$ (resp. $H^r(K_v,\tilde{T})$) est nulle pour chaque $v \in X^{(1)}$.
\item[(ii)] Soit $L$ une extension finie déployant $T$. Supposons que $\Sha^2(L,\mathbb{G}_m)$ est nul. Pour chaque groupe topologique abélien $A$, notons $A_{\wedge}$ la limite projective des $A/nA$ et $A^D$ le groupe des morphismes continus $A \rightarrow \mathbb{Q}/\mathbb{Z}$. On dispose alors d'une suite exacte à 7 termes:\\
\centerline{\xymatrix{
\Sha^{d+3}(\tilde{T})^D \ar[d] & 0  \ar[l] & \\
H^0(K,T)_{\wedge} \ar[r]& \mathbb{P}^{0}(T)_{\wedge} \ar[r] & H^{d+2}(K,\tilde{T})^D \ar[d]\\
H^{d+1}(K,\tilde{T})^D & \mathbb{P}^1(T) \ar[l] & H^1(K,T)\ar[l]
}}
où $\mathbb{P}^r(T)$ désigne un produit restreint des $H^r(K_v,T)$ pour $v \in X^{(1)}$ (voir section \ref{tore}), et une suite exacte à 8 termes: \\
\centerline{\xymatrix{
& \mathbb{P}^{d+1}(\tilde{T}) \ar[r] & H^1(K,T)^D \ar[d] &\\
(H^0(K,T)_{\wedge})^D\ar[d] & \mathbb{P}^{d+2}(\tilde{T})_{tors} \ar[l] & H^{d+2}(K,\tilde{T})\ar[l] &\\
H^{d+3}(K,\tilde{T}) \ar[r] & \mathbb{P}^{d+3}(\tilde{T}) \ar[r] &  (\varprojlim_n {_n}T(K^s))^D \ar[r] & 0,
}}
où $\mathbb{P}^r(\tilde{T})$ désigne un produit restreint des $H^r(K_v,\tilde{T})$ pour $v \in X^{(1)}$.
\end{itemize}
\end{theorem}

\hspace{4ex} Dans la section 4, nous généralisons le théorème précédent aux groupes de type multiplicatif et même aux complexes de deux tores:

\begin{theorem} (théorèmes \ref{PT gm} et \ref{cor})\\
On garde les notations de \ref{0}. Soit $G=[T_1 \rightarrow T_2]$ un complexe de deux $K$-tores placés en degrés $-1$ et $0$. Notons $\tilde{G}$ le cône de $\tilde{T_2} \rightarrow \tilde{T_1}$. 
\begin{itemize} 
\item[(i)] On a alors un accouplement parfait:
$$ \overline{\Sha^{0}(G)_{tors}} \times \overline{\Sha^{d+2}(\tilde{G})} \rightarrow \mathbb{Q}/\mathbb{Z}.$$
\item[(ii)] On suppose que le morphisme $\check{T_1} \rightarrow \check{T_2}$ ou le morphisme $\hat{T_2} \rightarrow \hat{T_1}$ est injectif. On a alors un accouplement parfait:
$$ \overline{\Sha^{1}(G)} \times \overline{\Sha^{d+1}(\tilde{G})} \rightarrow \mathbb{Q}/\mathbb{Z}.$$
\item[(iii)] On suppose que le morphisme $\check{T_1} \rightarrow \check{T_2}$ est injectif. On a alors un accouplement parfait:
$$ \overline{\Sha^{2}(G)} \times \overline{\Sha^{d}(\tilde{G})_{tors}} \rightarrow \mathbb{Q}/\mathbb{Z}.$$
\end{itemize}
\end{theorem}

\hspace{4ex} On remarque que, dans la première partie du théorème \ref{0}, nous avons dû quotienter $\Sha^{d+2}( \tilde{T})$ et $\Sha^2( T)$ par leurs sous-groupes divisibles maximaux et que, dans les deux autres parties, nous avons dû supposer la nullité de $\Sha^2(\mathbb{G}_m)$. Il est donc intéressant de montrer que, sous de bonnes hypothèses, les groupes $\Sha^{d+2}( \tilde{T})$ et $\Sha^2( T)$ sont finis et que le groupe $\Sha^2(\mathbb{G}_m)$ est nul. Cela fait l'objet de la section 5:

\begin{theorem} \textit{(théorèmes \ref{nul 1} et \ref{constante} et exemples \ref{ex0} et \ref{ex})} \label{0courbes} \\
On reprend les notations du théorème \ref{0}. Pour $ i \in \{0,1,...,d\}$, on note $k_i$ le corps $i$-local associé à $k$ et $\mathcal{O}_{k_i}$ son anneau des entiers.
\begin{itemize}
\item[(i)] Si $X$ est la droite projective ou une conique dans $\mathbb{P}^2_k$, alors $\Sha^2(\mathbb{G}_m)$ est nul.
\item[(ii)] Supposons que $k_1$ est un corps $p$-adique et que $X$ est une courbe sur $k$ de la forme $\text{Proj}(k[x,y,z]/(P(x,y,z))) $ où $P \in \mathcal{O}_{k_1}[x,y,z]$ est un polynôme homogène. Supposons aussi que $\text{Proj}(k_0[x,y,z]/(\overline{P}(x,y,z))) $ est une courbe lisse et géométriquement intègre. Alors $\Sha^2(\mathbb{G}_m)=0$.
\item[(iii)] Supposons que $k_0 = \mathbb{C}((t))$ et que $X$ est la courbe elliptique sur $k$ d'équation $y^2=x^3+Ax+B$ avec $A, B \in k_0$. Supposons de plus que la courbe elliptique sur $k_0$ d'équation $y^2=x^3+Ax+B$ admet une réduction modulo $t$ de type additif.  Alors $\Sha^2(\mathbb{G}_m)=0$.
\item[(iv)] Soit $p$ un nombre premier impair. Si $k= \mathbb{Q}_p((t_1))...((t_{d-1}))$ avec $d>1$ et $X$ est la courbe elliptique d'équation $y^2=x(1-x)(x-p)$, alors $\Sha^2(\mathbb{G}_m)\neq 0$.
\item[(v)] Si $k=\mathbb{C}((t_1))...((t_{d+1}))$ avec $d\geq 0$ et $X$ est la courbe elliptique d'équation $y^2=x(1-x)(x-t_1)$, alors $\Sha^2(\mathbb{G}_m)\neq 0$.
\end{itemize}
\end{theorem}

\hspace{4ex} Les preuves de (ii) et (iii) reposent notamment sur un théorème établi par Harbater, Hartmann et Krashen par la technique de patching (le théorème 3.3.6 de \cite{HHK}), et permettent donc de comprendre les liens entre les deux grandes méthodes présentées au début de l'introduction.\\

\hspace{4ex} Dans la section 6, nous appliquons les théorèmes de dualité arithmétique établis dans la section 3 à l'étude de l'approximation faible pour les tores (on se reportera à \ref{def} pour les définitions):

\begin{theorem} (théorème \ref{AF tore} et corollaire \ref{AFF tore})\\
On garde les notations de \ref{0}. Soit $S \subseteq X^{(1)}$ une partie finie.
\begin{itemize}
\item[(i)] On a une suite exacte:
$$0 \rightarrow \overline{T(K)}_S \rightarrow \prod_{v \in S} T(K_v) \rightarrow \overline{\Sha^{d+2}_S(\tilde{T})}^D \rightarrow \Sha^{1}(T) \rightarrow 0,$$
où $\overline{T(K)}_S$ désigne l'adhérence de $T(K)$ dans $\prod_{v \in S} T(K_v)$ et $\Sha^{d+2}_S(\tilde{T})$ est le sous-groupe de $H^{d+2}(K,\tilde{T})$ constitué des éléments dont la restriction à $H^{d+2}(K_v,\tilde{T})$ est nulle pour chaque $v \in X^{(1)} \setminus S$.
\item[(ii)] On a la suite exacte:
$$0 \rightarrow \overline{T(K)} \rightarrow \prod_{v \in X^{(1)}} T(K_v) \rightarrow (\Sha^{d+2}_{\omega}(\tilde{T}))^D \rightarrow (\Sha^{d+2}(\tilde{T}))^D \rightarrow 0,$$
où $ \overline{T(K)}$ désigne l'adhérence de $T(K)$ dans $\prod_{v \in X^{(1)}} T(K_v)$ et $\Sha^{d+2}_{\omega}(\tilde{T})$ est le sous-groupe de $H^{d+2}(K,\tilde{T})$ constitué des éléments dont la restriction à $H^{d+2}(K_v,\tilde{T})$ est nulle pour presque tout $v \in X^{(1)}$.
\item[(iii)] Le tore $T$ vérifie l'approximation faible si, et seulement si, $\Sha^{d+2}(\tilde{T}) = \Sha^{d+2}_{\omega}(\tilde{T})$.
\item[(iv)] Le tore $T$ vérifie l'approximation faible faible si, et seulement si, le groupe de torsion $\Sha^{d+2}_{\omega}(\tilde{T})$ est de type cofini. En particulier, lorsque $\Sha^{d+2}(L,\mathbb{Z}(d)) = 0$ pour une extension finie $L$ de $K$ déployant $T$, le tore $T$ vérifie l'approximation faible faible si, et seulement si, $\Sha^{d+2}_{\omega}(\tilde{T})$ est fini.
\item[(v)] Le tore $T$ vérifie l'approximation faible faible dénombrable si, et seulement si, le groupe de torsion $\Sha^{d+2}_{\omega}(\tilde{T})$ est dénombrable.
\end{itemize}
\end{theorem}

\hspace{4ex} Pour terminer, dans la section 7, nous utilisons les résultats des sections 3 et 4 pour étudier le principe local-global pour les $K$-espaces principaux homogènes sous un groupe linéaire connexe. Dans le cas où $k = \mathbb{C}((t))$, nous appliquons les méthodes de Borovoi (\cite{Bor}) et Sansuc (\cite{San}), ce qui permet notamment de montrer que la seule obstruction est l'obstruction de Brauer-Manin, alors que, dans les autres cas, nous suivons la méthode de Harari et Szamuely (\cite{HS1}).

\subsection{Remerciements}

Je tiens à remercier en premier lieu David Harari pour son soutien et ses conseils, ainsi que pour sa lecture soigneuse de ce texte: sans lui, ce travail n'aurait pas pu voir le jour. Je suis aussi très reconnaissant à Jean-Louis Colliot-Thélène, Tamás Szamuely et Bruno Kahn pour leurs commentaires et leurs remarques, et à Giancarlo Lucchini pour de nombreuses discussions. Je voudrais finalement remercier l'École Normale Supérieure et l'Université Paris-Sud pour leurs excellentes conditions de travail.

\subsection{Notations}

\textbf{Groupes abéliens.} Pour $A$ un groupe topologique abélien (éventuellement muni de la topologie discrète), $n>0$ un entier et $l$ un nombre premier, on notera:
\begin{itemize}
\item[$\bullet$] $A_{tors}$ la partie de torsion de $A$.
\item[$\bullet$] ${_n}A$ la partie de $n$-torsion de $A$.
\item[$\bullet$] $A\{l\}$ la partie de torsion $l$-primaire de $A$.
\item[$\bullet$] $A_{\text{non}-l}= \bigoplus_{p \neq l} A\{p\}$ où $p$ décrit tous les nombres premiers différents de $l$.
\item[$\bullet$] $A^{(l)}$ le complété pour la topologie $l$-adique de $A$.
\item[$\bullet$] $A^{\wedge}$ le complété profini de $A$.
\item[$\bullet$] $A_{\wedge}$ la limite projective des $A/nA$.
\item[$\bullet$] $A_{div}$ le sous-groupe divisible maximal de $A$. En général, il ne coïncide pas avec le sous-groupe constitué des éléments divisibles de $A$.
\item[$\bullet$] $\overline{A}$ le quotient de $A$ par $A_{div}$.
\item[$\bullet$] $A^D$ le groupe des morphismes continus $A \rightarrow \mathbb{Q}/\mathbb{Z}$.
\end{itemize}
\vspace{5 mm}

\textbf{Complexes.} Soit $\mathcal{A}$ une catégorie abélienne. Lorsque $A_0$, $A_1$, $A_2$ ..., $A_n$ sont des objets de $\mathcal{A}$ munis de morphismes $A_{i+1} \rightarrow A_{i}$ pour $0 \leq i \leq n-1$, on notera $[A_n \rightarrow A_{n-1} \rightarrow ... \rightarrow A_0]$ le complexe où tous les termes en degrés strictement inférieurs à $-n$ ou strictement positifs sont nuls et où $A_i$ est placé en degré $-i$ pour $0 \leq i \leq n$. Lorsque $f^{\bullet}: A^{\bullet} \rightarrow B^{\bullet}$ est un morphisme de complexes, on notera $[A^{\bullet} \rightarrow B^{\bullet}]$ le cône de $f^{\bullet}$.\\

\textbf{Faisceaux.} Sauf indication du contraire, tous les faisceaux sont considérés pour le petit site étale. Dans la suite, on fera souvent appel à des catégories dérivées de faisceaux: on entendra toujours par là la catégorie dérivée des faisceaux étales sur le schéma considéré. Pour $F$ un faisceau sur un schéma $X$, on note $\text{Hom}_X(F,-)$ (ou $\text{Hom}(F,-)$ s'il n'y a pas d'ambigüité) le foncteur qui à un faisceau $G$ sur $X$ assosie le groupe des morphismes de faisceaux de $F$ vers $G$ et $\underline{\text{Hom}}_X(F,-)$ (ou $\underline{\text{Hom}}(F,-)$ s'il n'y a pas d'ambigüité) le foncteur qui à un faisceau $G$ sur $X$ associe le faisceau étale $U \mapsto \text{Hom}_U(F|_U,G|_U)$. De même, pour $n>0$ et pour $F$ un faisceau de $\mathbb{Z}/n\mathbb{Z}$-modules sur un schéma $X$, on note $\text{Hom}_{X,\mathbb{Z}/n\mathbb{Z}}(F,-)$ (ou $\text{Hom}_{\mathbb{Z}/n\mathbb{Z}}(F,-)$ s'il n'y a pas d'ambigüité) le foncteur qui à un faisceau de $\mathbb{Z}/n\mathbb{Z}$-modules $G$ sur $X$ associe le groupe des morphismes de faisceaux de $F$ vers $G$ et $\underline{\text{Hom}}_{X,\mathbb{Z}/n\mathbb{Z}}(F,-)$ (ou $\underline{\text{Hom}}_{\mathbb{Z}/n\mathbb{Z}}(F,-)$ s'il n'y a pas d'ambigüité) le foncteur qui à un faisceau de $\mathbb{Z}/n\mathbb{Z}$-modules $G$ sur $X$ associe le faisceau étale $U \mapsto \text{Hom}_{U,\mathbb{Z}/n\mathbb{Z}}(F|_U,G|_U)$.\\

\textbf{Corps locaux supérieurs.} Les corps 0-locaux sont par définition les corps finis et le corps $\mathbb{C}((t))$. Pour $d \geq 1$, un corps $d$-local est un corps complet pour une valuation discrète dont le corps résiduel est $(d-1)$-local. \underline{On remarquera que cette} \underline{définition est plus générale que la définition standard.} Lorsque $k$ est un corps $d$-local, on notera $k_0$, $k_1$, ..., $k_d$ les corps tels que $k_0$ est fini ou $\mathbb{C}((t))$, $k_d=k$, et pour chaque $i$ le corps $k_i$ est le corps résiduel de $k_{i+1}$. On rappelle le théorème de dualité sur un corps $d$-local $k$: pour tout $\text{Gal}(k^s/k)$-module fini $M$ d'ordre $n$ premier à $\text{Car}(k_1)$, on a un accouplement parfait de groupes finis $H^r(k,M) \times H^{d+1-r}(k,\text{Hom}(M,\mu_n^{\otimes d})) \rightarrow H^{d+1}(k,\mu_n^{\otimes d}) \cong \mathbb{Z}/n\mathbb{Z}$. Ce théorème est énoncé et démontré dans \cite{MilADT} (théorème 2.17) lorsque $k_0 \neq \mathbb{C}((t))$. Il se prouve exactement de la même manière dans ce dernier cas: en effet, il suffit de procéder par récurrence à l'aide du lemme 2.18 de \cite{MilADT}, l'initialisation étant réduite à la dualité évidente $H^r(k_{-1},M) \times H^{-r}(k_{-1},\text{Hom}(M,\mathbb{Z}/n\mathbb{Z})) \rightarrow H^0(k_{-1},\mathbb{Z}/n\mathbb{Z})\cong\mathbb{Z}/n\mathbb{Z}$ pour le corps ``$-1$-local'' $k_{-1} = \mathbb{C}$.\\

\textbf{Groupes de Tate-Shafarevich.} Lorsque $L$ est le corps des fonctions d'une variété projective lisse géométriquement intègre $Y$ sur un corps $l$ et $M$ est un objet de la catégorie dérivée des $\text{Gal}(L^s/L)$-modules discrets, le $r$-ième groupe de Tate-Shafarevich de $M$ est, par définition, le groupe $\Sha^r(L,M) = \text{Ker}(H^r(L,M) \rightarrow \prod_{v \in Y^{(1)}} H^r(L_v,M))$. Dans la suite, il sera aussi utile d'introduire, pour chaque partie $S$ de $Y^{(1)}$, le groupe $\Sha^r_S(L,M) = \text{Ker}(H^r(L,M) \rightarrow \prod_{v \in Y^{(1)} \setminus S} H^r(L_v,M))$, ainsi que le groupe $\Sha^r_{\omega}(L,M)=\bigcup_{S \subseteq X^{(1)}, \text{ }S\text{ finie}} \Sha^r_S(L,M)$.\\

\textbf{Groupe de Brauer.} Lorsque $Z$ est un schéma, on note $\text{Br}(Z)$ le groupe de Brauer cohomologique $H^2(Z,\mathbb{G}_m)$. Si $Z$ est une $L$-variété pour un certain corps $L$, on note $\text{Br}_{\text{al}}(Z)$ le groupe de Brauer algébrique $\text{Ker}(\text{Br}(Z) \rightarrow \text{Br}(Z \times_k k^s))/\text{Im}(\text{Br}(L) \rightarrow \text{Br}(Z))$. Finalement, si $L$ est le corps des fonctions d'une variété projective lisse géométriquement intègre $Y$ sur un corps $l$, on notera $\Brusse(Z)$ le groupe $\text{Ker}(\text{Br}_{\text{al}}(Z) \rightarrow \prod_{v\in Y^{(1)}}\text{Br}_{\text{al}}(Z \times_L {L_v}))$. \\

\textbf{Cadre.} Dans toute la suite, $d$ désignera un entier naturel fixé (éventuellement nul), $k$ un corps $d$-local et $X$ une variété projective lisse géométriquement intègre sur $k$ de dimension $x$, avec $x>0$. On notera $X^{(1)}$ l'ensemble de ses points de codimension 1 et $K$ son corps des fonctions. Pour alléger, \underline{lorsque $k_0$ est fini, on supposera que le} \underline{corps $k_1$ est de caractéristique 0}: autrement dit, ou bien $k_0 = \mathbb{C}((t))$, ou bien $d\geq 1$ et $k_1$ est un corps $p$-adique. Le cas $\text{Car}(k_1) \neq 0$ sera brièvement traité à la fin de certaines parties en remarque, les preuves étant tout à fait analogues. Lorsque $M$ est un objet de la catégorie dérivée des $\text{Gal}(K^s/K)$-modules discrets, on notera $\Sha^r(M)$ (resp. $\Sha^r_S(M)$, resp. $\Sha^r_{\omega}(M)$) au lieu de $\Sha^r(K,M)$ (resp. $\Sha^r_S(K,M)$, resp. $\Sha^r_{\omega}(K,M)$).\\

\textbf{Cohomologie à support compact.} Pour $j: U \hookrightarrow X$ une immersion ouverte et $\mathcal{F}$ un faisceau sur $U$, le $r$-ième groupe de cohomologie à support compact est, par définition, le groupe $H^r_c(U,\mathcal{F}) = H^r(X,j_!\mathcal{F})$. On notera aussi $\mathbb{R}\Gamma_c(U,\mathcal{F})$ le complexe $\mathbb{R}\Gamma(X,j_!\mathcal{F})$, dont le $r$-ième groupe de cohomologie est le $r$-ième groupe de cohomologie à support compact de $\mathcal{F}$. De même, lorsque $\mathcal{F}^{\bullet}$ est un complexe de faisceaux sur $U$, on notera $H^r_c(U,\mathcal{F}^{\bullet})$ le groupe d'hypercohomologie $H^r(k, \mathbb{R}f_* j_! \mathcal{F}^{\bullet}) = H^r(X, j_! \mathcal{F}^{\bullet})$, où $f$ désigne le morphisme propre $X \rightarrow \text{Spec} \; k$. On remarquera que, contrairement à la définition classique de la cohomologie à support compact, la définition ci-dessus ne dépend pas du choix d'une compactification lisse de $U$ (plus précisément, nous avons choisi $X$ comme compactification lisse de $U$).

\subsection{Quelques rappels sur les complexes de Bloch et la cohomologie motivique}

Dans l'article \cite{Blo}, Bloch associe à chaque schéma $Y$ séparé de type fini sur un corps $E$ et à chaque entier naturel $i$ un complexe noté $z^i(Y,\cdot)$. Lorsque $Y$ est lisse, on note $\mathbb{Z}(i)$ (resp. $\mathbb{Z}(i)_{\text{Zar}}$) le complexe de faisceaux $z^i(-,\cdot)[-2i]$ sur le petit site étale (resp. sur le petit site de Zariski), et pour chaque groupe abélien $A$, on note $A(i)$ (resp. $A(i)_{\text{Zar}}$) le complexe $A \otimes \mathbb{Z}(i)$ (resp. $A \otimes \mathbb{Z}(i)_{\text{Zar}}$), qui coïncide avec le complexe $A \otimes^{\mathbf{L}} \mathbb{Z}(i)$ (resp. $A \otimes^{\mathbf{L}} \mathbb{Z}(i)_{\text{Zar}}$) puisque chaque terme de $\mathbb{Z}(i)$ (resp. $\mathbb{Z}(i)_{\text{Zar}}$) est un faisceau plat. Ce dernier fait sera utilisé de manière récurrente dans la suite.\\
On rappelle sans preuve les propriétés fondamentales du complexe $\mathbb{Z}(i)$ dont nous aurons besoin dans la suite:

\begin{theorem} \textbf{(Propriétés du complexe $\mathbb{Z}(i)$)}\\
Soit $Y$ un schéma séparé lisse de type fini sur un corps $E$. Soit $i \geq 0$.
\begin{itemize}
\item[(i)] Il existe un quasi-isomorphisme de complexes de faisceaux étales:
$$\mathbb{Z}(1)[1] \cong \mathbb{G}_m.$$
\item[(ii)] Le complexe $\mathbb{Z}(i)_{\text{Zar}}$ est concentré en degrés $\leq i$.
\item[(iii)] Soit $\alpha$ la projection du site étale de $Y$ sur le site de Zariski de $Y$. On a alors un isomorphisme:
$$\mathbb{Q}(i)_{\text{Zar}} \cong R\alpha_*\mathbb{Q}(i).$$
\item[(iv)] \textit{(Geisser-Levine)} Pour $m$ inversible dans $E$, on a un quasi-isomorphisme de complexes de faisceaux étales:
$$\mathbb{Z}/m\mathbb{Z}(i) \cong \mu_m^{\otimes i}.$$
\item[(v)] \textit{(Nesterenko-Suslin, Totaro)} On a un isomorphisme:
$$K_i^M(E) \cong H^i_{\text{Zar}}(E,\mathbb{Z}(i)_{\text{Zar}}),$$
où $K_*^M(E)$ désigne la $K$-théorie de Milnor de $E$.
\item[(vi)] \textit{(Conjecture de Bloch-Kato, prouvée par Rost-Voevodsky)} Pour $m$ inversible dans $E$, on a un isomorphisme:
$$K_i^M(E)/mK^M_i(E) \cong H^i(E,\mathbb{Z}/m\mathbb{Z}(i)).$$
\item[(vii)] \textit{(Conjecture de Beilinson-Lichtenbaum)} Si $j \leq i+1$, on a un isomorphisme:
$$H^j_{\text{Zar}}(Y,\mathbb{Z}(i)_{\text{Zar}}) \cong H^j(Y,\mathbb{Z}(i)).$$
\end{itemize}
\end{theorem}

\begin{proof}
\begin{itemize}
\item[(i)] Corollaire 6.4 de \cite{Blo}. 
\item[(ii)] Lemme 2.5 de \cite{Kah}.
\item[(iii)] Théorème 2.6 c) de \cite{Kah}.
\item[(iv)] Théorème 1.5 de \cite{GL}.
\item[(v)] \cite{NS} ou \cite{Tot}.
\item[(vi)] Le théorème a été prouvé par Rost et Voevodsky dans les articles \cite{SJ} et \cite{Voe}. On pourra aussi aller voir une esquisse de preuve dans l'exposé \cite{Rio}.
\item[(vii)] Suslin et Voevodsky ont démontré que la conjecture de Bloch-Kato implique la conjecture de Beilinson-Lichtenbaum en supposant la résolution de singularités dans \cite{SV}. L'hypothèse de résolution des singularités est supprimée dans \cite{GL}. Le cas de la caractéristique positive est traité dans \cite{GL2}.
\end{itemize}
\end{proof}

\begin{remarque}
En fait, le complexe $\mathbb{Z}(i)$ peut aussi être construit lorsque $Y$ est un schéma lisse sur le spectre d'un anneau de Dedekind, et dans ce cas, les propriétés (i), (ii), (iii), (iv) et (vii) restent vraies (on pourra aller voir \cite{Gei} et \cite{GeiDed}). Cela nous sera utile pour pouvoir appliquer ces constructions lorsque $Y$ est le complété de l'anneau local de $X$ en un point de codimension 1.
\end{remarque}

Pour terminer, on rappelle aussi que l'on dispose d'un accouplement $\mathbb{Z}(i) \otimes^{\mathbf{L}} \mathbb{Z}(j) \rightarrow \mathbb{Z}(i+j)$ pour chaque couple d'entiers naturels $(i,j)$ (voir par exemple \cite{Tot}), et on prouve la propriété suivante, bien connue des experts, pour laquelle on n'a pas trouvé de référence adéquate:

\begin{proposition}
Soit $n>0$. Soit $l$ un corps de caractéristique 0. Les propriétés suivantes sont équivalentes:
\begin{itemize}
\item[(i)] $\text{cd}(l) \leq n$.
\item[(ii)] Pour toute extension algébrique $L$ de $l$, on a $H^{n+2}(L,\mathbb{Z}(n)) = 0$.
\item[(ii')] Pour toute extension finie $L$ de $l$, on a $H^{n+2}(L,\mathbb{Z}(n)) = 0$.
\end{itemize}
\end{proposition}

\begin{proof} 
Supposons (i). Soit $L$ une extension algébrique de $l$. Comme $H^{n+2}(L,\mathbb{Q}(n)) = 0$, on dispose d'une surjection $H^{n+1}(L,\mathbb{Q}/\mathbb{Z}(n)) \rightarrow H^{n+2}(L,\mathbb{Z}(n))$. Or $H^{n+1}(L,\mathbb{Q}/\mathbb{Z}(n)) \cong \varinjlim_m H^{n+1}(L,\mu_m^{\otimes n})$. Puisque l'hypothèse (i) impose que $\text{cd}(L) \leq n$, on déduit que $H^{n+1}(L,\mathbb{Q}/\mathbb{Z}(n)) = 0$, d'où (ii).\\
Supposons (ii). Soit $p$ un nombre premier. Soit $L_p$ le sous-corps de $l^s$ fixé par un $p$-Sylow de $\text{Gal}(l^s/l)$. Le corps $L_p$ contient toutes les racines $p$-ièmes de l'unité de $l^s$, donc:
$$H^{n+1}(L_p,\mathbb{Z}/p\mathbb{Z}) = H^{n+1}(L_p,\mu_p^{\otimes n}).$$
Or, d'après la conjecture de Beilinson-Lichtenbaum, on a $H^{n+1}(L_p,\mathbb{Z}(n)) = 0$, ce qui entraîne que $H^{n+1}(L_p,\mu_p^{\otimes n}) = H^{n+2}(L_p,\mathbb{Z}(n))[p] = 0$. On en déduit que $H^{n+1}(L_p,\mathbb{Z}/p\mathbb{Z}) = 0$, ce qui implique (i).\\
Pour terminer, montrons que (ii') implique (ii). Soit $L$ une extension algébrique de $l$. Pour chaque extension finie $L'$ de $l$ contenue dans $L$, on a par hypothèse $H^{n+2}(L',\mathbb{Z}(n)) = 0$. En passant à la limite inductive sur $L'$, on obtient $H^{n+2}(L,\mathbb{Z}(n)) = 0$, ce qui achève la preuve.

\end{proof}

\begin{remarque}
Si $l$ est un corps de caractéristique $p > 0$, les propriétés suivantes sont équivalentes:
\begin{itemize}
\item[(i)] $\text{cd}(l) \leq n$.
\item[(ii)] Pour toute extension algébrique $L$ de $l$, on a $H^{n+2}(L,\mathbb{Z}(n))_{\text{non}-p} = 0$.
\item[(ii')] Pour toute extension finie $L$ de $l$, on a $H^{n+2}(L,\mathbb{Z}(n))_{\text{non}-p} = 0$.
\end{itemize}
La preuve est tout à fait analogue.
\end{remarque}

\section{\scshape Quelques résultats préliminaires}

Dans cette partie, nous allons présenter deux résultats préliminaires nécessaires pour la suite. D'une part, il sera utile de disposer de reformulations dans le langage des catégories dérivées du théorème de dualité sur un corps local supérieur et de la dualité de Poincaré. N'ayant pas trouvé de référence convenable, nous donnons une preuve du théorème suivant:

\begin{theorem} \textbf{(Dualité sur un corps local supérieur et dualité de Poincaré)}\\ \label{rappels}
Rappelons que nous avons supposé $\text{Car}(k_1)=0$. Soit $m>0$.
\begin{itemize}
\item[(i)] Notons $G = \text{Gal}(k^s/k)$. Soit $M^{\bullet}$ un complexe de $G$-modules discrets de $m$-torsion borné inférieurement. On a alors un isomorphisme dans la catégorie dérivée:
$$\mathbb{R}\text{Hom}_{G,\mathbb{Z}/m\mathbb{Z}} (M^{\bullet}, \mathbb{Z}/m\mathbb{Z}(d))[d+1] \cong \mathbb{R}\text{Hom}_{\mathbb{Z}/m\mathbb{Z}} (\mathbb{R}\Gamma_G M^{\bullet}, \mathbb{Z}/m\mathbb{Z}).$$
\item[(ii)] Soit $\overline{X}$ une $k^s$-variété projective lisse de dimension $x$. Soit $\overline{\mathcal{F}}$ un faisceau constructible de $m$-torsion sur un ouvert non vide $\overline{U}$ de $\overline{X}$. On a alors un isomorphisme dans la catégorie dérivée des groupes abéliens:
$$\mathbb{R} \text{Hom}_{\overline{U},\mathbb{Z}/m\mathbb{Z}}(\overline{\mathcal{F}},\mu_m^{\otimes x}) \cong \mathbb{R} \text{Hom}_{\mathbb{Z}/m\mathbb{Z}} (\mathbb{R}\Gamma_c(\overline{U},\overline{\mathcal{F}}), \mathbb{Z}/m\mathbb{Z})[-2x].$$
\end{itemize}
\end{theorem}

\begin{proof}
\begin{itemize}
\item[(i)] D'après le théorème de dualité sur un corps local supérieur (voir le théorème 2.17 de \cite{MilADT}), pour chaque $G$-module discret $M$ de $m$-torsion, on a un isomorphisme:
$$\text{Ext}^{r}_{G,\mathbb{Z}/m\mathbb{Z}}(M,\mu_m^{\otimes d}) \cong \text{Hom}_{\mathbb{Z}/m\mathbb{Z}}( H^{d+1-r}(k,M), \mathbb{Z}/m\mathbb{Z}).$$
Or, comme $\mathbb{Z}/m\mathbb{Z}$ est un $\mathbb{Z}/m\mathbb{Z}$-module injectif, on dispose des suites spectrales:
$$\text{Ext}^{s}_{G,\mathbb{Z}/m\mathbb{Z}}(H^{-t}(M^{\bullet}),\mu_m^{\otimes d}) \Rightarrow \text{Ext}^{s+t}_{G,\mathbb{Z}/m\mathbb{Z}}(M^{\bullet},\mu_m^{\otimes d}),$$
$$\text{Hom}_{\mathbb{Z}/m\mathbb{Z}}( H^{d+1-s}(k,H^{-t}(M^{\bullet})), \mathbb{Z}/m\mathbb{Z}) \Rightarrow \text{Hom}_{\mathbb{Z}/m\mathbb{Z}}( H^{d+1-s-t}(k,M^{\bullet}), \mathbb{Z}/m\mathbb{Z}).$$
On en déduit que:
$$\text{Ext}^{r}_{G,\mathbb{Z}/m\mathbb{Z}}(M^{\bullet},\mu_m^{\otimes d}) \cong \text{Hom}_{\mathbb{Z}/m\mathbb{Z}}( H^{d+1-r}(k,M^{\bullet}), \mathbb{Z}/m\mathbb{Z}).$$
Toujours parce que $\mathbb{Z}/m\mathbb{Z}$ est un $\mathbb{Z}/m\mathbb{Z}$-module injectif, on remarque que:
\begin{align*}
\text{Hom}_{\mathbb{Z}/m\mathbb{Z}}( H^{d+1-r}(k,M^{\bullet}), \mathbb{Z}/m\mathbb{Z}) &\cong H^{r-d-1}\mathbb{R}\text{Hom}_{\mathbb{Z}/m\mathbb{Z}}(\mathbb{R}\Gamma_G M^{\bullet},\mathbb{Z}/m\mathbb{Z}) \\
&\cong H^{r}\mathbb{R}\text{Hom}_{\mathbb{Z}/m\mathbb{Z}}(\mathbb{R}\Gamma_G M^{\bullet},\mathbb{Z}/m\mathbb{Z})[-d-1].
\end{align*}
On a donc des isomorphismes:
$$H^r \mathbb{R}\text{Hom}_{G,\mathbb{Z}/m\mathbb{Z}}(M^{\bullet},\mu_m^{\otimes d}) \cong H^{r}\mathbb{R}\text{Hom}_{\mathbb{Z}/m\mathbb{Z}}(\mathbb{R}\Gamma_G M^{\bullet},\mathbb{Z}/m\mathbb{Z})[-d-1].$$
Reste alors à contruire un morphisme dans la catégorie dérivée:
$$\mathbb{R}\text{Hom}_{G,\mathbb{Z}/m\mathbb{Z}}(M^{\bullet},\mu_m^{\otimes d}) \rightarrow \mathbb{R}\text{Hom}_{\mathbb{Z}/m\mathbb{Z}}(\mathbb{R}\Gamma_G M^{\bullet},\mathbb{Z}/m\mathbb{Z})[-d-1]$$
induisant les isomorphismes précédents. Pour ce faire, on remarque que $\mathbb{R}\Gamma_G M^{\bullet} = \mathbb{R}\text{Hom}_{G,\mathbb{Z}/m\mathbb{Z}}(\mathbb{Z}/m\mathbb{Z},M^{\bullet})$, ce qui fournit un morphisme:
$$\mathbb{R}\text{Hom}_{G,\mathbb{Z}/m\mathbb{Z}}(M^{\bullet},\mu_m^{\otimes d}) \rightarrow \mathbb{R}\text{Hom}_{\mathbb{Z}/m\mathbb{Z}}(\mathbb{R}\Gamma_G M^{\bullet},\mathbb{R}\text{Hom}_{G,\mathbb{Z}/m\mathbb{Z}}(\mathbb{Z}/m\mathbb{Z},\mu_m^{\otimes d})).$$
Comme $k$ est de dimension cohomologique $d+1$, le complexe $\mathbb{R}\text{Hom}_{G,\mathbb{Z}/m\mathbb{Z}}(\mathbb{Z}/m\mathbb{Z},\mu_m^{\otimes d})$ est quasi-isomorphe au complexe:
$$... \rightarrow \text{Hom}^{d-1}_{G,\mathbb{Z}/m\mathbb{Z}}(\mathbb{Z}/m\mathbb{Z},\mu_m^{\otimes d}) \rightarrow \text{Hom}^{d}_{G,\mathbb{Z}/m\mathbb{Z}}(\mathbb{Z}/m\mathbb{Z},\mu_m^{\otimes d}) \rightarrow H^{d+1}(k,\mu_m^{\otimes d}) \rightarrow 0 \rightarrow ...,$$
où $H^{d+1}(k,\mu_m^{\otimes d})$ est placé en degré $d+1$, ce qui, en tenant compte de l'isomorphisme $H^{d+1}(k,\mu_m^{\otimes d}) \cong \mathbb{Z}/m\mathbb{Z}$, permet de construire un morphisme:
$$\mathbb{R}\text{Hom}_{G,\mathbb{Z}/m\mathbb{Z}}(\mathbb{Z}/m\mathbb{Z},\mu_m^{\otimes d}) \rightarrow \mathbb{Z}/m\mathbb{Z}[-d-1].$$
Par composition, on obtient donc un morphisme dans la catégorie dérivée:
$$\mathbb{R}\text{Hom}_{G,\mathbb{Z}/m\mathbb{Z}}(M^{\bullet},\mu_m^{\otimes d}) \rightarrow \mathbb{R}\text{Hom}_{\mathbb{Z}/m\mathbb{Z}}(\mathbb{R}\Gamma_G M^{\bullet},\mathbb{Z}/m\mathbb{Z})[-d-1]$$
induisant les isomorphismes:
$$H^r \mathbb{R}\text{Hom}_{G,\mathbb{Z}/m\mathbb{Z}}(M^{\bullet},\mu_m^{\otimes d}) \cong H^{r}\mathbb{R}\text{Hom}_{\mathbb{Z}/m\mathbb{Z}}(\mathbb{R}\Gamma_G M^{\bullet},\mathbb{Z}/m\mathbb{Z})[-d-1].$$
C'est donc un isomorphisme dans la catégorie dérivée.
\item[(ii)] Le théorème de dualité de Poincaré (voir le théorème 8.5.3 de \cite{Fu} ou le théorème 11.1 de la partie VI de \cite{MilEC}) fournit un isomorphisme:
$$\text{Ext}^r_{\overline{U},\mathbb{Z}/m\mathbb{Z}}(\overline{\mathcal{F}},\mu_m^{\otimes x}) \cong \text{Hom}_{\mathbb{Z}/m\mathbb{Z}}(H^{2x-r}_c(\overline{U},\overline{\mathcal{F}}),\mathbb{Z}/m\mathbb{Z}).$$
Comme $\mathbb{Z}/m\mathbb{Z}$ est un $\mathbb{Z}/m\mathbb{Z}$-module injectif, on remarque que:
\begin{align*}
\text{Hom}_{\mathbb{Z}/m\mathbb{Z}}(H^{2x-r}_c(\overline{U},\overline{\mathcal{F}}),\mathbb{Z}/m\mathbb{Z}) &\cong H^{r-2x}\mathbb{R}\text{Hom}_{\mathbb{Z}/m\mathbb{Z}}(\mathbb{R}\Gamma_c\overline{\mathcal{F}},\mathbb{Z}/m\mathbb{Z}) \\
&\cong H^{r}\mathbb{R}\text{Hom}_{\mathbb{Z}/m\mathbb{Z}}(\mathbb{R}\Gamma_c\overline{\mathcal{F}},\mathbb{Z}/m\mathbb{Z})[-2x].
\end{align*}
On a donc des isomorphismes:
$$H^r \mathbb{R}\text{Hom}_{\overline{U},\mathbb{Z}/m\mathbb{Z}}(\overline{\mathcal{F}},\mu_m^{\otimes x}) \cong H^{r}\mathbb{R}\text{Hom}_{\mathbb{Z}/m\mathbb{Z}}(\mathbb{R}\Gamma_c\overline{\mathcal{F}},\mathbb{Z}/m\mathbb{Z})[-2x].$$
Comme $\overline{X}$ est de dimension cohomologique $2x$ et comme $H^{2x}(k,\mu_m^{\otimes x}) \cong \mathbb{Z}/m\mathbb{Z}$,  des méthodes tout à fait analogues à celles utlisées en (i) montrent que les isomorphismes précédents proviennent d'un morphisme:
$$\mathbb{R}\text{Hom}_{\overline{U},\mathbb{Z}/m\mathbb{Z}}(\overline{\mathcal{F}},\mu_m^{\otimes x}) \rightarrow \mathbb{R}\text{Hom}_{\mathbb{Z}/m\mathbb{Z}}(\mathbb{R}\Gamma_c\overline{\mathcal{F}},\mathbb{Z}/m\mathbb{Z})[-2x],$$
qui est donc un isomorphisme dans la catégorie dérivée.
\end{itemize}
\end{proof}

\begin{remarque}
L'assertion (i) reste vraie lorsque $\text{Car}(k_1)\neq 0$ à condition de supposer que $m$ est premier avec $\text{Car}(k_1)$. L'assertion (ii) reste vraie en remplaçant $k^s$ par n'importe quel corps séparablement clos de caractéristique première à $m$.
\end{remarque}

D'autre part, nous aurons besoin du calcul d'un groupe de cohomologie, analogue au lemme 1.1 de \cite{HS1}:

\begin{lemma}\label{préliminaire}
Soit $U$ un ouvert non vide de $X$. Alors $H^{d+2x+2}_c(U,\mathbb{Z}(d+x))\cong \mathbb{Q}/\mathbb{Z}$.
\end{lemma}

\begin{proof} 
\begin{itemize}
\item[$\bullet$] Calculons d'abord $H^{d+2x+2}(X,\mathbb{Z}(d+x))$. La suite exacte courte $0 \rightarrow \mathbb{Z}(d+x) \rightarrow \mathbb{Q}(d+x) \rightarrow \mathbb{Q}/\mathbb{Z}(d+x) \rightarrow 0$ fournit une suite exacte de cohomologie:
\begin{align*}
H^{d+2x+1}(X,\mathbb{Q}(d+x)) & \rightarrow H^{d+2x+1}(X,\mathbb{Q}/\mathbb{Z}(d+x)) \\ & \rightarrow H^{d+2x+2}(X,\mathbb{Z}(d+x)) \rightarrow H^{d+2x+2}(X,\mathbb{Q}(d+x)).
\end{align*}
Rappelons que, si $\alpha$ désigne la projection du petit site étale de $X$ sur le petit site de Zariski de $X$, alors le morphisme $\mathbb{Q}(d+x)_{\text{Zar}} \rightarrow R\alpha_*\mathbb{Q}(d+x)$ est un isomorphisme. Par conséquent, pour $r>0$, $H^r(X,\mathbb{Q}(d+x))\cong H^r_{\text{Zar}}(X,\mathbb{Q}(d+x)_{\text{Zar}})$, qui est nul pour $r > 2x+d$ puisque $\mathbb{Q}(d+x)$ est concentré en degrés $\leq d+x$ et $X$ est de dimension $x$. On en déduit un isomorphisme:
$$ H^{d+2x+2}(X,\mathbb{Z}(d+x))\cong H^{d+2x+1}(X,\mathbb{Q}/\mathbb{Z}(d+x)) \cong \varinjlim_m H^{d+2x+1}(X,\mu_m^{\otimes x+d}).$$
Dans la suite spectrale de Hochschild-Serre $H^r(k,H^s(\overline{X},\mu_m^{\otimes x+d})) \Rightarrow H^{r+s}(X,\mu_m^{\otimes x+d})$, les termes de gauche sont nuls dès que $r>d+1$ ou $s>2x$. Cela induit un isomorphisme:
$$H^{d+2x+1}(X,\mu_m^{\otimes x+d}) \cong H^{d+1}(k,H^{2x}(\overline{X},\mu_m^{\otimes x+d})) \cong H^{d+1}(k,H^{2x}(\overline{X},\mu_m^{\otimes x}) \otimes \mu_m^d).$$
Or l'application trace du théorème de dualité de Poincaré fournit un isomorphisme $H^{2x}(\overline{X},\mu_m^{\otimes x}) \cong \mathbb{Z}/m\mathbb{Z}$. Comme $k$ est $d$-local, cela impose des isomorphismes: $$H^{d+2x+1}(X,\mu_m^{\otimes x+d})\cong H^{d+1}(k, \mu_m^d) \cong \mathbb{Z}/m\mathbb{Z},$$
ce qui permet de conclure que $H^{d+2x+2}(X,\mathbb{Z}(d+x)) \cong \mathbb{Q}/\mathbb{Z}$.
\item[$\bullet$] Notons $Z = X \setminus U$ et écrivons maintenant la suite exacte de localisation:
\begin{align*}
H^{d+2x+1}&(Z,i^*\mathbb{Z}(x+d))  \rightarrow H^{d+2x+2}_c(U,\mathbb{Z}(x+d))\\ & \rightarrow H^{d+2x+2}(X,\mathbb{Z}(x+d)) \rightarrow H^{d+2x+2}(Z,i^*\mathbb{Z}(x+d)),
\end{align*}
où $i: Z \hookrightarrow X$ désigne l'immersion fermée. Comme avant, si l'on note $z = \dim Z$, on a un isomorphisme $H^{r+1}(Z,i^*\mathbb{Z}(x+d)) \cong H^r(Z, i^*\mathbb{Q}/\mathbb{Z}(x+d)) = H^r(Z, \mathbb{Q}/\mathbb{Z}(x+d))$ dès que $r > z + x + d$ et ces groupes sont nuls dès que $r > d+2z+1$. Comme $z < x$, ces inégalités sont vérifiées pour $r= d+2x$ et $r=d+2x+1$, d'où un isomorphisme: $ H^{d+2x+2}_c(U,\mathbb{Z}(x+d))\cong H^{d+2x+2}(X,\mathbb{Z}(x+d))$. En tenant compte de la première étape, on obtient un isomorphisme $H^{d+2x+2}_c(U,\mathbb{Z}(d+x))\cong \mathbb{Q}/\mathbb{Z}$.
\end{itemize}
\end{proof}

\begin{remarque}
Si $p=\text{Car}(k_1)>0$, on obtient un isomorphisme $H^{d+2x+2}_c(U,\mathbb{Z}(d+x))_{\text{non}-p}\cong (\mathbb{Q}/\mathbb{Z})_{\text{non}-p}$.
\end{remarque}

\section{\scshape Modules finis}

Dans cette section, nous allons traiter le cas des $ \text{Gal}(K^s/K)$-modules finis. C'est le cas le plus simple, et il sera essentiel dans la suite puisque, afin de traiter les tores ou les groupes de type multiplicatif, nous nous ramènerons toujours à des modules finis. Dans cette section, tout reste analogue au cas où $k$ est un corps $p$-adique traité dans les articles \cite{HS1} et \cite{HS2}. Notre premier but consiste à établir un théorème de dualité de type Poitou-Tate sur le corps $K$. 

\begin{proposition} \textbf{(Dualité globale d'Artin-Verdier pour les modules finis)} \label{AV fini}
Soit $U$ un ouvert non vide de $X$. Soit $\mathcal{F}$ un schéma en groupes fini étale abélien sur $U$ et notons $\mathcal{F}' = \underline{Hom} (\mathcal{F},\mathbb{Q}/\mathbb{Z}(d+x))$. On a alors un accouplement parfait de groupes finis pour chaque $r \in \mathbb{Z}$:
$$H^r(U,\mathcal{F}') \times H^{d+2x+1-r}_c(U,\mathcal{F}) \rightarrow \mathbb{Q}/\mathbb{Z}.$$
\end{proposition}

\begin{proof}
Soit $m \geq 0$ tel que $\mathcal{F}$ est de $m$-torsion. Notons $G = \text{Gal}(k^s/k)$. En notant $\overline{U} = U \times_k k^s$ et $\overline{\mathcal{F}}$ la restriction de $\mathcal{F}$ à $\overline{U}$, on a des isomorphismes:
\begin{align}
\mathbb{R} \text{Hom}_{U,\mathbb{Z}/m\mathbb{Z}} (\mathcal{F},\mathbb{Z}/m\mathbb{Z}(d+x))
 & \cong \mathbb{R}\Gamma_G \mathbb{R} \text{Hom}_{\overline{U},\mathbb{Z}/m\mathbb{Z}}(\overline{\mathcal{F}},\mathbb{Z}/m\mathbb{Z}(d+x))\\
& \cong \mathbb{R}\Gamma_G \mathbb{R} \text{Hom}_{\mathbb{Z}/m\mathbb{Z}} (\mathbb{R}\Gamma_c(\overline{U},\overline{\mathcal{F}}), \mathbb{Z}/m\mathbb{Z}(d))[-2x]\\
& \cong  \mathbb{R}\text{Hom}_{G,\mathbb{Z}/m\mathbb{Z}} (\mathbb{R}\Gamma_c(\overline{U},\overline{\mathcal{F}}), \mathbb{Z}/m\mathbb{Z}(d))[-2x]\\
& \cong  \mathbb{R}\text{Hom}_{\mathbb{Z}/m\mathbb{Z}} (\mathbb{R}\Gamma_c(U,\mathcal{F}), \mathbb{Z}/m\mathbb{Z})[-2x-d-1]
\end{align}
où la ligne (1) découle d'une suite spectrale de Grothendieck (voir le corollaire 10.8.3 de \cite{Wei}), la ligne (2) découle de la dualité de Poincaré (voir \ref{rappels}), la ligne (3) découle d'une autre suite spectrale de Grothendieck, et la ligne (4) découle du théorème de dualité locale sur le corps $d$-local $k$ (voir \ref{rappels}). Cela fournit un isomorphisme $$\text{Ext}^r_{U,\mathbb{Z}/m\mathbb{Z}} (\mathcal{F},\mathbb{Z}/m\mathbb{Z}(d+x))) \cong \text{Ext}^{r-2x-d-1}_{\mathbb{Z}/m\mathbb{Z}} (\mathbb{R}\Gamma_c(U,\mathcal{F}),\mathbb{Z}/m\mathbb{Z}).$$
Or le terme de gauche est isomorphe à $H^r(U,\underline{Hom}(\mathcal{F},\mathbb{Z}/m\mathbb{Z}(d+x)))$ et celui de droite à $\text{Hom}_{\mathbb{Z}/m\mathbb{Z}} (H^{2x+d+1-r}_c(U,\mathcal{F}),\mathbb{Z}/m\mathbb{Z})$. D'où le lemme.
\end{proof}

Dans tout le reste de cette section, on suppose que:
\begin{hypo}\label{40}
\begin{minipage}[t]{12.72cm}
$x=1$, c'est-à-dire $X$ est une courbe, 
\end{minipage}
\end{hypo}
et on se donne un $\text{Gal}(K^s/K)$-module discret fini $F$, ainsi qu'un schéma en groupes fini étale $\mathcal{F}$ sur un ouvert non vide $U_0$ de $X$ tel que $F = \mathcal{F} \times_{U_0} \text{Spec} K$. On note $\mathcal{F}' = \underline{\text{Hom}}(\mathcal{F},\mathbb{Q}/\mathbb{Z}(d+1))$ et ${F}' = \underline{\text{Hom}}({F},\mathbb{Q}/\mathbb{Z}(d+1))$. Comme $\mathcal{F}$ est fini étale et comme $k$ est de caractéristique 0, on remarquera que le faisceau $\mathcal{F}'$ est localement constant à tiges finies sur $U_0$ et donc qu'il est représenté par un schéma en groupes fini étale (voir la proposition V.1.1 de \cite{MilEC}). Pour $U$ ouvert de $U_0$, on définit:
$$D^r_{sh}(U,\mathcal{F}) = \text{Ker} \left( H^r(U,\mathcal{F}) \rightarrow \prod_{v \in X^{(1)}} H^r(K_v,F) \right) $$
$$\mathcal{D}^r(U,\mathcal{F}') = \text{Im}(H^r_c(U,\mathcal{F}') \rightarrow H^r(K,F')).$$

\begin{proposition} \label{Sha fini}
On rappelle que l'on a supposé (H \ref{40}). Soit $r \in \mathbb{N}$. Il existe un ouvert non vide $U_1$ de $U_0$ tel que, pour tout ouvert non vide $U$ de $U_1$, on a $\mathcal{D}^r(U,\mathcal{F}) = \Sha^r(F)$.
\end{proposition}

\begin{proof} Les commutativités des diagrammes que nous utilisons implicitement dans cette preuve sont contenues dans la proposition 4.3 de \cite{CTH}. \\
Remarquons d'abord que $\mathcal{D}^r(U_0,\mathcal{F})$ est fini et que, pour $V \subseteq U$ deux ouverts de $U_0$, on a $\mathcal{D}^r(V,\mathcal{F}) \subseteq \mathcal{D}^r(U,\mathcal{F})$. Par conséquent, il existe $U_1$ ouvert non vide de $U_0$ tel que, pour tout ouvert $U$ de $U_1$, on a $\mathcal{D}^r(U,\mathcal{F}) = \mathcal{D}^r(U_1,\mathcal{F})$. \\
Soit maintenant $x \in \mathcal{D}^r(U_1,\mathcal{F})$. Alors, pour chaque ouvert non vide $U$ de $U_1$, on a $x \in \mathcal{D}^r(U,\mathcal{F})$ et donc, comme la suite $H^r_c(U,\mathcal{F}) \rightarrow H^r(U,\mathcal{F}) \rightarrow \bigoplus_{v \in X \setminus U} H^r(K_v,F)$ est exacte, on a $x \in \text{Ker}\left( H^r(K,F) \rightarrow \bigoplus_{v \in X\setminus U} H^r(K_v,F) \right) $. Cela étant vrai pour chaque ouvert $U$ de $U_1$, on déduit que $x \in \Sha^r(F)$, c'est-à-dire que $\mathcal{D}^r(U_1,\mathcal{F}) \subseteq \Sha^r(F)$.\\
Considérons à présent $x \in \Sha^r(F)$. Alors il existe $U$ un ouvert non vide de $U_1$ et $\tilde{x}\in H^r(U,\mathcal{F})$ tels que $\tilde{x}$ s'envoie sur $x$ par $H^r(U,\mathcal{F}) \rightarrow H^r(K,F)$. Comme $x \in \Sha^r(F)$, on a $\tilde{x} \in \text{Ker}\left( H^r(U,\mathcal{F}) \rightarrow \bigoplus_{v \in X \setminus U} H^r(K_v,F) \right)  = \text{Im}(H^r_c(U,\mathcal{F}) \rightarrow H^r(U,\mathcal{F}))$. Par conséquent, $x \in \mathcal{D}^r(U,\mathcal{F}) = \mathcal{D}^r(U_1,\mathcal{F})$ et $\Sha^r(F) \subseteq \mathcal{D}^r(U_1,\mathcal{F})$.\\
Finalement, on a $\mathcal{D}^r(U,\mathcal{F}) = \Sha^r(F)$ pour chaque ouvert non vide $U$ de $U_1$.
\end{proof}

\begin{proposition} \label{suite fini}
On rappelle que l'on a supposé (H \ref{40}). Soit $r \in \mathbb{N}$. Pour chaque ouvert non vide $U$ de $U_0$, on dispose d'une suite exacte:
$$\bigoplus_{v \in X^{(1)}} H^{r-1}(K_v,F') \rightarrow H^r_c(U,\mathcal{F}') \rightarrow \mathcal{D}^r(U,\mathcal{F}') \rightarrow 0.$$
\end{proposition}

\begin{proof} 
Sauf indication du contraire, les commutativités des diagrammes que nous utilisons implicitement dans cette preuve sont contenues dans la proposition 4.3 de \cite{CTH}. \\
Soit $U$ un ouvert non vide de $U_0$. Pour chaque ouvert non vide $V$ de $U$, on dispose d'un diagramme commutatif:\\
\centerline{\xymatrix{
\bigoplus_{v \in X \setminus U} H^{r-1}(K_v,F') \ar[r]\ar[d] & H^r_c(U,\mathcal{F}')\\
\bigoplus_{v \in X \setminus V} H^{r-1}(K_v,F') \ar[r] & H^r_c(V,\mathcal{F}')\ar[u]
}}
Cela permet de définir un morphisme $\bigoplus_{v \in X^{(1)}} H^{r-1}(K_v,F') \rightarrow H^r_c(U,\mathcal{F}') $. On vérifie aisément que $\bigoplus_{v \in X^{(1)}} H^{r-1}(K_v,F') \rightarrow H^r_c(U,\mathcal{F}') \rightarrow \mathcal{D}^r(U,\mathcal{F}') \rightarrow 0$ est un complexe. Reste donc à prouver l'inclusion $\text{Ker} (H^r_c(U,\mathcal{F}') \rightarrow \mathcal{D}^r(U,\mathcal{F}')) \subseteq \text{Im}(\bigoplus_{v \in X^{(1)}} H^{r-1}(K_v,F') \rightarrow H^r_c(U,\mathcal{F}'))$.\\
Pour ce faire, on se donne $\alpha \in \text{Ker} (H^r_c(U,\mathcal{F}') \rightarrow \mathcal{D}^r(U,\mathcal{F}'))$ et, comme dans la proposition 4.2 de \cite{HS2}, on remarque que, pour $V \subseteq U$, on a un diagramme commutatif dont la première ligne est exacte:\\
\centerline{\xymatrix{
H^r_c(V,\mathcal{F}') \ar[r]&H^r_c(U,\mathcal{F}') \ar[r]\ar[d] & \bigoplus_{v \in U \setminus V} H^r(k(v),F')\ar[d]\\
& H^r(K,F') \ar[r] & \bigoplus_{v \in U \setminus V} H^r(K_v,F')
}}
où le morphisme vertical $H^r(k(v),F') \rightarrow H^r(K_v,F')$ est obtenu par composition:
$$H^r(k(v),F') \cong H^r(\mathcal{O}_v^h,\mathcal{F}') \rightarrow H^r(K_v^h,F') \cong H^r(K_v,F').$$
La flèche $H^r(\mathcal{O}_v^h,\mathcal{F}') \rightarrow H^r(K_v^h,F')$ s'identifie en cohomologie galoisienne à la flèche d'inflation $H^r(G_v/I_v,F'^{I_v}) \rightarrow H^r(G_v,F')$, où $G_v = \text{Gal}(K_v^{h,s}/K_v^h)$ et $I_v$ est le sous-groupe d'inertie. Or le module galoisien $F'$ est non ramifié en $v$ (puisque $v \in U_0$) et donc $F'^{I_v} = F'$. Comme la projection $G_v \rightarrow G_v/I_v$ admet une section, on en déduit que $H^r(\mathcal{O}_v^h,\mathcal{F}') \rightarrow H^r(K_v^h,F')$ et $H^r(k(v),F') \rightarrow H^r(K_v,F')$ sont injectifs.\\
Prenons pour $V$ un ouvert de $U$ tel que l'image de $\alpha$ dans $H^r(V,\mathcal{F}')$ est nulle. En exploitant les diagrammes précédents, on voit alors que $\alpha$ provient d'un élément $\tilde{\alpha}$ de $H^r_c(V,\mathcal{F}')$ dont l'image dans $\bigoplus_{v \in X \setminus V} H^r(K_v,F')$ est nulle. Par conséquent, $\alpha$ provient de $\bigoplus_{v \in X \setminus V} H^{r-1}(K_v,F')$, ce qui prouve l'inclusion $\text{Ker} (H^r_c(U,\mathcal{F}') \rightarrow \mathcal{D}^r(U,\mathcal{F}')) \subseteq \text{Im}(\bigoplus_{v \in X^{(1)}} H^{r-1}(K_v,F') \rightarrow H^r_c(U,\mathcal{F}'))$.
\end{proof}

\begin{theorem} \textbf{(Dualité de Poitou-Tate pour les modules finis)} \label{PT fini}\\
On rappelle que l'on a supposé (H \ref{40}), c'est-à-dire que $X$ est une courbe, et que ${F}'$ désigne $\underline{\text{Hom}}({F},\mathbb{Q}/\mathbb{Z}(d+1))$. Soit $r \in \mathbb{Z}$. On dispose d'une dualité parfaite de groupes finis:
$$\Sha^r(F) \times \Sha^{d+3-r}(F') \rightarrow \mathbb{Q}/\mathbb{Z}.$$
\end{theorem}

\begin{proof} 
Si $r<0$ ou si $r>d+3$, le résultat découle immédiatement par dimension cohomologique. On suppose donc $0 \leq r \leq d+3$.\\
À l'aide de la proposition \ref{suite fini}, pour $U$ ouvert de $U_1$, on dispose des deux suites exactes:
$$\bigoplus_{v \in X^{(1)}} H^{d+2-r}(K_v,F') \rightarrow H^{d+3-r}_c(U,\mathcal{F}') \rightarrow \mathcal{D}^{d+3-r}(U,\mathcal{F}') \rightarrow 0,$$
$$0 \rightarrow D^r_{sh}(U,\mathcal{F}) \rightarrow H^r(U,\mathcal{F}) \rightarrow \prod_{v \in X^{(1)}} H^r(K_v,F).$$
Or $(\bigoplus_{v \in X^{(1)}} H^{d+2-r}(K_v,F'))^D \cong \prod_{v \in X^{(1)}} H^r(K_v,F)$ d'après la théorème de dualité locale sur un corps $d+1$-local et $H^{d+3-r}_c(U,\mathcal{F}')^D \cong  H^r(U,\mathcal{F})$ d'après le théorème \ref{AV fini}. En utilisant la proposition 4.3 de \cite{CTH}, on en déduit que $\mathcal{D}^{d+3-r}(U,\mathcal{F}')^D \cong D^r_{sh}(U,\mathcal{F})$. En passant à la limite inductive sur $U$ et en utilisant la proposition \ref{Sha fini}, on obtient $\Sha^{d+3-r}(F')^D \cong \Sha^r(F)$, d'où le résultat.
 \end{proof}
 
 Dans le reste de cette section, nous cherchons à établir une suite exacte de type Poitou-Tate sur le corps $K$.
 
\begin{proposition} \label{nr fini}
On rappelle que l'on a supposé (H \ref{40}). Soient $r \in \{ 1,2,...,d+1\}$ et $v \in X^{(1)}$. Alors:
\begin{itemize}
\item[(i)] $H^r( \mathcal{O}_v,\mathcal{F})$ est un sous-groupe de $H^r(K_v,F)$,
\item[(ii)] $H^{d+2-r}( \mathcal{O}_v,\mathcal{F}')$ est un sous-groupe de $H^{d+2-r}(K_v,F')$,
\item[(iii)] $H^r( \mathcal{O}_v,\mathcal{F})$ et $H^{d+2-r}( \mathcal{O}_v,\mathcal{F}')$ sont les annulateurs l'un de l'autre dans l'accouplement parfait:
$$H^r(K_v,F) \times H^{d+2-r}(K_v,F')  \rightarrow \mathbb{Q}/\mathbb{Z}.$$
\end{itemize}
\end{proposition}

\begin{remarque}
La propriété (ii) découle bien sûr de (i) par symétrie.
\end{remarque}

\begin{proof} 
Notons $G_v = \text{Gal}(K_v^s/K_v)$, $I_v = \text{Gal}(K_v^s/K_v^{nr})$ et $g_v = \text{Gal}(K_v^{nr}/K_v)=\text{Gal}(k(v)^s/k(v))$. La suite exacte $1 \rightarrow I_v \rightarrow G_v \rightarrow g_v \rightarrow 1$ est scindée. En exploitant la suite spectrale de Hochschild-Serre, on obtient donc des suites exactes courtes:
$$0 \rightarrow H^r(g_v,F) \rightarrow H^r(G_v,F) \rightarrow H^{r-1}(g_v,H^1(I_v,F)) \rightarrow 0,$$
$$0 \rightarrow H^{d+2-r}(g_v,F') \rightarrow H^{d+2-r}(G_v,F') \rightarrow H^{d+1-r}(g_v,H^1(I_v,F')) \rightarrow 0.$$
Comme $H^r( \mathcal{O}_v,\mathcal{F}) = H^r(g_v,F)$ et $H^{d+2-r}( \mathcal{O}_v,\mathcal{F}') = H^{d+2-r}(g_v,F')$, on obtient (i) et (ii). De plus, en appliquant le théorème de dualité locale sur $k(v)$, on obtient:
$$H^{r-1}(g_v, H^1(I_v,F)) \cong H^{r-1}(g_v,F(-1)) \cong H^{d+2-r}(g_v,F')^D.$$
On en déduit que $\frac{|H^r(K_v,F)|}{|H^r(\mathcal{O}_v,\mathcal{F})|} = |H^{d+2-r}(\mathcal{O}_v,\mathcal{F}')|$.\\
De plus, comme $H^{d+2}(\mathcal{O}_v,\mathbb{Q}/\mathbb{Z}(d+1)) = H^{d+2}(g_v,\mathbb{Q}/\mathbb{Z}(d+1))=0$ par dimension cohomologique, on sait que l'accouplement $H^r(K_v,F) \times H^{d+2-r}(K_v,F')  \rightarrow \mathbb{Q}/\mathbb{Z}$ est nul sur $H^r( \mathcal{O}_v,\mathcal{F})\times H^{d+2-r}( \mathcal{O}_v,\mathcal{F}')$. La propriété (iii) en découle immédiatement.
\end{proof}

Dans la suite, on notera $\mathbb{P}^r(F)$ le produit restreint des groupes $H^r(K_v,F)$ pour $v \in X^{(1)}$ par rapport aux groupes $H^r(\mathcal{O}_v,\mathcal{F})$: c'est un groupe abélien que l'on munit de la topologie produit restreint. Ainsi, $\mathbb{P}^r(F)$ est muni d'une structure de groupe abélien localement compact. On remarque aisément que $\mathbb{P}^0(F) = \prod_{v \in X^{(1)}} H^0(K_v,F)$. Comme, pour chaque $v \in U_0^{(1)}$, le corps $k(v)$ est de dimension cohomologique $d+1$, on a $H^{d+2} (\mathcal{O}_v,\mathcal{F})=H^{d+2}(k(v),\mathcal{F}_{\overline{v}}) = 0$, ce qui entraîne que $\mathbb{P}^{d+2}(F) = \bigoplus_{v \in X^{(1)}} H^{d+2}(K_v,F)$.

\begin{proposition}\label{bouts suite fini}
On rappelle que l'on a supposé (H \ref{40}). Soit $r \in \{ 1,2,..., d+2\}$. On a une suite exacte:\\
\centerline{\xymatrix{H^r(K,F) \ar[r]& \mathbb{P}^r(F) \ar[r] & H^{d+2-r}(K,F')^D}}
où le morphisme $\mathbb{P}^r(F) \rightarrow H^{d+2-r}(K,F')^D$ est défini par:
$$(f_v) \mapsto (f' \mapsto \sum_{v \in X^{(1)}} (f_v,f'_v)_v),$$
pour $(f_v) \in \mathbb{P}^r(F)$ et $f' \in H^{d+2-r}(K,F')$.
\end{proposition}

\begin{proof} 
Pour chaque ouvert $V \subseteq U_0$, on dispose d'une suite exacte de groupes finis:
$$H^r_c(V,\mathcal{F}) \rightarrow H^r(V,\mathcal{F}) \rightarrow \bigoplus_{v \in X \setminus V} H^r(K_v,F) \rightarrow H^{r+1}_c(V,\mathcal{F}).$$
Soient $U$ et $V$ deux ouverts de $U_0$ tels que $V \subseteq U$. Comme d'après le lemme 2.2 de \cite{HS2} un élément de $H^r(V,\mathcal{F})$ provient de $H^r(U,\mathcal{F})$ si, et seulement si, son image dans $H^r(K_v,F)$ provient de $H^r(\mathcal{O}_v,\mathcal{F})$ pour chaque $v \in U \setminus V$, on obtient une suite exacte:
$$H^r(U,\mathcal{F}) \rightarrow \bigoplus_{v \in X \setminus U} H^r(K_v,F) \oplus \bigoplus_{v \in U \setminus V} H^r(\mathcal{O}_v,\mathcal{F}) \rightarrow H^{r+1}_c(V,\mathcal{F}).$$
Or $H^{r+1}_c(V,\mathcal{F}) \cong H^{d+2-r}(V,\mathcal{F}')^D$ et $\varinjlim_V  H^{d+2-r}(V,\mathcal{F}') = H^{d+2-r}(K,F')$. Par conséquent, en passant à la limite projective sur $V$, on obtient un complexe:
$$H^r(U,\mathcal{F}) \rightarrow \prod_{v \in X \setminus U} H^r(K_v,F) \times \prod_{v \in U} H^r(\mathcal{O}_v,\mathcal{F}) \rightarrow H^{d+2-r}(K,F')^D.$$
Montrons que ce complexe est en fait une suite exacte. Pour ce faire, choisissons un élément $(f_v)_{v \in X^{(1)}} \in \prod_{v \in X \setminus U} H^r(K_v,F) \times \prod_{v \in U} H^r(\mathcal{O}_v,\mathcal{F}) $ dont l'image dans $H^{d+2-r}(K,F')^D$ est nulle. On remarque alors aisément que l'image de $(f_v)_{v \in X \setminus V}$ est nulle dans $ H^{r+1}_c(V,\mathcal{F})$ pour chaque ouvert $V$ de $U$. Cela impose que l'image réciproque $E_V$ de $(f_v)_{v \in X \setminus V}$ dans $H^r(U,\mathcal{F})$ est non vide. Comme pour $W \subseteq V$ on a $E_W\subseteq E_V$ et comme le groupe $H^r(U,\mathcal{F})$ est fini, l'intersection $\bigcap_V E_V$ est non vide. On en déduit que $(f_v)_{v \in X^{(1)}}$ est dans l'image de $H^r(U,\mathcal{F}) \rightarrow \prod_{v \in X \setminus U} H^r(K_v,F) \times \prod_{v \in U} H^r(\mathcal{O}_v,\mathcal{F})$ et donc que la suite
$$H^r(U,\mathcal{F}) \rightarrow \prod_{v \in X \setminus U} H^r(K_v,F) \times \prod_{v \in U} H^r(\mathcal{O}_v,\mathcal{F}) \rightarrow H^{d+2-r}(K,F')^D$$
est exacte. Il suffit alors de prendre la limite inductive sur $U$.
\end{proof}

\begin{theorem} \textbf{(Suite exacte de Poitou-Tate pour les modules finis)} \label{PT fini suite}\\
On rappelle que l'on a supposé (H \ref{40}), c'est-à-dire que $X$ est une courbe, et que ${F}'$ désigne $\underline{\text{Hom}}({F},\mathbb{Q}/\mathbb{Z}(d+1))$. On a une suite exacte à $3(d+3)$ termes:\\
\centerline{\xymatrix{
0 \ar[r] & H^0(K,F) \ar[r] & \mathbb{P}^0(F) \ar[r] & H^{d+2}(K,F')^D \ar[d]\\
& H^{d+1}(K,F')^D \ar[d] & \mathbb{P}^1(F) \ar[l] & H^1(K,F)\ar[l]\\
& H^2(K,F) \ar[r] & \mathbb{P}^2(F) \ar[r] & H^d(K,F')^D \ar[d]\\
&\ar[d] &...\ar@{..}[l]&\ar@{..}[l]\\
& H^{d+1}(K,F) \ar[r] & \mathbb{P}^{d+1}(F) \ar[r] & H^{1}(K,F')^D \ar[d]\\
0 & H^{0}(K,F')^D \ar[l] & \mathbb{P}^{d+2}(F) \ar[l] & H^{d+2}(K,F)\ar[l]\\
}}
\end{theorem}

\begin{proof}  
Par dualité, on remarque aisément que $\mathbb{P}^{d+2}(F) \rightarrow H^{0}(K,F')^D$ est surjective. Comme de plus la première ligne est duale de la dernière, la proposition précédente prouve que toutes les lignes sont exactes.\\
Les flèches verticales $H^r(K,F')^D \rightarrow H^{d+3-r}(K,F)$ sont définies par la composée:
$$H^{d+3-r}(K,F')^D \twoheadrightarrow \Sha^{d+3-r}(F')^D \cong \Sha^{r}(F) \hookrightarrow H^{r}(K,F).$$
Pour vérifier l'exactitude au niveau de ces flèches, il suffit de voir que la suite duale de:
$$0 \rightarrow \Sha^{d+3-r}(F') \rightarrow H^{d+3-r}(K,F') \rightarrow \mathbb{P}^{d+3-r}(F')$$
est encore exacte. Pour ce faire, il suffit de vérifier que l'image $I$ de $H^{d+3-r}(K,F')$ dans $\mathbb{P}^{d+3-r}(F')$ est discrète. Fixons $U$ un ouvert de $U_0$. Tout élément de $ I \cap (\prod_{v \in X \setminus U} H^{d+3-r}(K_v,F') \times \prod_{v \in U} H^{d+3-r}(\mathcal{O}_v,\mathcal{F}'))$ provient de $H^{d+3-r}(U,\mathcal{F}')$, qui est un groupe fini. On en déduit que le groupe $ I \cap (\prod_{v \in X \setminus U} H^{d+3-r}(K_v,F') \times \prod_{v \in U} H^{d+3-r}(\mathcal{O}_v,\mathcal{F}'))$ est fini, et donc $I$ est discret dans $\mathbb{P}^{d+3-r}(F')$.
\end{proof}

\begin{remarque}
Si on ne suppose pas que $k_1$ est de caractéristique 0, toutes les propriétés de cette section restent valables à condition de supposer que $F$ est d'ordre premier à $\text{Car}(k_1)$.
\end{remarque}

\section{ \scshape Tores}\label{tore}

Fixons $a$ un élément de $\{0,1,2,...,d+x\}$ et $\hat{T}$ un $\text{Gal}(K^s/K)$-module qui, comme groupe abélien, est libre de type fini. Notons $\check{T} = \text{Hom}(\hat{T},\mathbb{Z})$. Soit $\hat{\mathcal{T}}$ un faisceau défini sur un ouvert $U_0$ de $X$, localement isomorphe à un faisceau constant libre de type fini et étendant $\hat{T}$. On pose $\check{\mathcal{T}} = \text{\underline{Hom}}(\hat{T},\mathbb{Z})$. Soient $T=\check{T} \otimes^{\mathbf{L}} \mathbb{Z}(a)[1] $ et $\mathcal{T} = \check{\mathcal{T}} \otimes^{\mathbf{L}} \mathbb{Z}(a)[1]$. On dispose alors d'un accouplement:
$$\mathcal{T} \otimes^{\mathbf{L}} (\hat{\mathcal{T}} \otimes^{\mathbf{L}} \mathbb{Z}(d+x-a)[1]) \rightarrow \mathbb{Z}(d+x)[2],$$
d'où un accouplement en hypercohomologie:
$$H^r(U,\mathcal{T}) \times H^{d+2x+1-r}_c(U,\hat{\mathcal{T}} \otimes^{\mathbf{L}} \mathbb{Z}(d+x-a)) \rightarrow H^{d+2x+2}_c(U,\mathbb{Z}(d+x)) \cong \mathbb{Q}/\mathbb{Z}$$
pour chaque ouvert non vide $U$ de $U_0$. De plus, comme on a un morphisme $\mathbb{Q}/\mathbb{Z}(d+x-a) \rightarrow \mathbb{Z}(d+x-a)[1]$, l'accouplement précédent induit aussi un accouplement:
$$H^r(U,\mathcal{T}) \times H^{d+2x-r}_c(U,\hat{\mathcal{T}} \otimes^{\mathbf{L}} \mathbb{Q}/\mathbb{Z}(d+x-a)) \rightarrow \mathbb{Q}/\mathbb{Z}.$$
On notera dans la suite $\tilde{T}_t = \hat{T} \otimes^{\mathbf{L}} \mathbb{Q}/\mathbb{Z}(d+x-a)$, $\tilde{T} = \hat{T} \otimes^{\mathbf{L}} \mathbb{Z}(d+x-a)$, $\tilde{\mathcal{T}}_t = \hat{\mathcal{T}} \otimes^{\mathbf{L}} \mathbb{Q}/\mathbb{Z}(d+x-a)$ et $\tilde{\mathcal{T}} = \hat{\mathcal{T}} \otimes^{\mathbf{L}} \mathbb{Z}(d+x-a)$. C'est $\tilde{T}$ qui nous intéresse, mais parfois il sera plus facile de travailler avec $\tilde{T}_t$.

\begin{remarque}
Lorsque $a=1$, les complexes $T$ et $\mathcal{T}$ sont quasi-isomorphes à des tores.
\end{remarque}

Nous cherchons à établir un théorème de dualité de type Poitou-Tate pour le complexe $T$. Pour ce faire, l'idée consiste à se ramener au cadre de la section précédente en travaillant avec $\tilde{T}_t$, qui est une limite inductive de faisceaux localement constants à tiges finies, plutôt qu'avec le complexe $\tilde{T}$.

\begin{theorem} \textbf{(Artin-Verdier pour les tores)} \label{AV tore}\\
Soit $r \in \mathbb{Z}$. Soit $U$ un ouvert non vide de $U_0$. Pour chaque nombre premier $l$, on dispose d'un accouplement parfait de groupes finis:
$$(H^r(U,\mathcal{T})\{l\})^{(l)} \times H^{d+2x-r}_c(U,\tilde{\mathcal{T}}_t)^{(l)}\{l\} \rightarrow \mathbb{Q}_l/\mathbb{Z}_l.$$
\end{theorem}

\begin{proof} 
Soit $l$ un nombre premier. Pour chaque entier naturel $n$, on dispose du triangle distingué:
$$\mathcal{T} \rightarrow \mathcal{T} \rightarrow \check{\mathcal{T}} \otimes \mathbb{Z}/l^n\mathbb{Z}(a)[1] \rightarrow \mathcal{T}[1],$$
d'où des suites exactes:
$$0 \rightarrow H^{r-1}(U,\mathcal{T})/l^n \rightarrow H^r(U,\check{\mathcal{T}} \otimes \mathbb{Z}/l^n\mathbb{Z}(a)) \rightarrow {_{l^n}}H^r(U,\mathcal{T}) \rightarrow 0.$$
En passant à la limite inductive sur $n$, on obtient une suite exacte:
$$0 \rightarrow H^{r-1}(U,\mathcal{T})\otimes \mathbb{Q}_l/\mathbb{Z}_l \rightarrow \varinjlim_n H^r(U,\check{\mathcal{T}} \otimes \mathbb{Z}/l^n\mathbb{Z}(a)) \rightarrow H^r(U,\mathcal{T})\{ l\} \rightarrow 0.$$
Par conséquent, pour chaque entier naturel $m$, on a un isomorphisme $(\varinjlim_n H^r(U,\check{\mathcal{T}} \otimes \mathbb{Z}/l^n\mathbb{Z}(a)))/l^m \rightarrow H^r(U,\mathcal{T})\{ l\}/l^m$. En passant à la limite projective sur $m$, on obtient:
$$(\varinjlim_n H^r(U,\check{\mathcal{T}} \otimes \mathbb{Z}/l^n\mathbb{Z}(a)))^{(l)} \cong H^r(U,\mathcal{T})\{ l\}^{(l)}.$$
D'autre part, on remarque que $\check{\mathcal{T}} \otimes \mathbb{Z}/l^n\mathbb{Z}(a)$ s'identifie à un faisceau localement constant à tiges finies, qui est bien sûr représentable par un schéma en groupes fini étale, et que:
$$(\check{\mathcal{T}} \otimes \mathbb{Z}/l^n\mathbb{Z}(a))' = \text{\underline{Hom}}(\check{\mathcal{T}} \otimes \mu_{l^n}^{\otimes a}, \mu_{l^n}^{\otimes d + x}) = \hat{\mathcal{T}} \otimes \mu_{l^n}^{\otimes d+x-a} = {_{l^n}}\tilde{\mathcal{T}}_t.$$ 
À l'aide de la suite exacte:
$$0 \rightarrow {_{l^n}}\tilde{\mathcal{T}}_t \rightarrow \tilde{\mathcal{T}}_t \rightarrow \tilde{\mathcal{T}}_t \rightarrow 0,$$
on obtient une suite exacte:
$$0 \rightarrow H^{d+2x-r}_c(U,\tilde{\mathcal{T}}_t)/l^n \rightarrow H^{d+2x+1-r}_c(U,{_{l^n}}\tilde{\mathcal{T}}_t) \rightarrow {_{l^n}}H^{d+2x+1-r}_c(U,\tilde{\mathcal{T}}_t) \rightarrow 0.$$
En passant à la limite projective sur $n$, on obtient une suite exacte:
$$0 \rightarrow H^{d+2x-r}_c(U,\tilde{\mathcal{T}}_t)^{(l)} \rightarrow \varprojlim_n H^{d+2x+1-r}_c(U,{_{l^n}}\tilde{\mathcal{T}}_t) \rightarrow \varprojlim_n {_{l^n}}H^{d+2x+1-r}_c(U,\tilde{\mathcal{T}}_t) \rightarrow 0.$$
Le groupe $\varprojlim_n {_{l^n}}H^{d+2x+1-r}_c(U,\tilde{\mathcal{T}}_t)$ étant sans torsion, on obtient un isomorphisme: 
$$H^{d+2x-r}_c(U,\tilde{\mathcal{T}}_t)^{(l)}\{l\} \rightarrow \left( \varprojlim_n H^{d+2x+1-r}_c(U,{_{l^n}}\tilde{\mathcal{T}}_t)\right) \{l\}.$$
Comme $H^r(U,\check{\mathcal{T}} \otimes \mathbb{Z}/l^n\mathbb{Z}(a))$ et $H^{d+2x+1-r}_c(U,{_{l^n}}\tilde{\mathcal{T}}_t)$ sont duaux pour chaque entier naturel $n$ d'après le théorème \ref{AV fini}, on obtient un accouplement parfait de groupes finis:
$$(H^r(U,\mathcal{T})\{l\})^{(l)} \times H^{d+2x-r}_c(U,\tilde{\mathcal{T}}_t)^{(l)}\{l\} \rightarrow \mathbb{Q}_l/\mathbb{Z}_l.$$
\end{proof}

\begin{remarque}\label{AV tore 2}
Par un raisonnement tout à fait analogue, on obtient un accouplement parfait de groupes finis:
$$(H^r(U,\mathcal{T})\{l\})^{(l)} \times H^{d+2x+1-r}_c(U,\tilde{\mathcal{T}})^{(l)}\{l\} \rightarrow \mathbb{Q}_l/\mathbb{Z}_l.$$
\end{remarque}

\underline{Dans toute la suite de cette section, on supposera (H \ref{40}), c'est-à-dire que $X$ est une} \underline{courbe.}
Cela impose en particulier que $a \in \{0,1,...,d+1\}$ et que $\tilde{T} = \hat{T} \otimes^{\mathbf{L}} \mathbb{Z}(d+1-a)$. 

\begin{proposition} \textbf{(Dualité locale pour les tores)} \label{local tore}\\
On rappelle que l'on a supposé (H \ref{40}). Soit $v \in X^{(1)}$.
\begin{itemize}
\item[(i)] On a un accouplement parfait de groupes finis:
$$H^a(K_v,T) \times H^{d+2-a}(K_v,\tilde{T}) \rightarrow \mathbb{Q}/\mathbb{Z}.$$
\item[(ii)] Pour chaque entier naturel $n$, on a des isomorphismes de groupes finis:
$$_nH^a(K_v,T) \cong (H^{d+2-a}(K_v,\tilde{T})/n)^D \cong (H^{d+1-a}(K_v,\tilde{T}_t)/n)^D.$$
\item[(iii)] On a un accouplement parfait entre un groupe profini et un groupe de torsion:
$$H^{a-1}(K_v,T)^{\wedge} \times H^{d+3-a}(K_v,\tilde{T}) \rightarrow \mathbb{Q}/\mathbb{Z}.$$
\end{itemize}

\end{proposition}

\begin{proof}
\begin{itemize}
\item[(i)]
Pour chaque $n>0$, on dispose des triangles distingués:
$$ \check{T} \otimes \mathbb{Z}/n\mathbb{Z}(a) \rightarrow T \rightarrow T \rightarrow (\check{T} \otimes \mathbb{Z}/n\mathbb{Z}(a))[1],$$
$$\tilde{T} \rightarrow \tilde{T} \rightarrow \hat{T} \otimes \mathbb{Z}/n\mathbb{Z}(d+1-a) \rightarrow \tilde{T}[1].$$
Cela fournit des suites exactes de cohomologie:
\begin{gather}
0 \rightarrow H^{a-1}(K_v,T)/n \rightarrow H^a(K_v,\check{T} \otimes \mathbb{Z}/n\mathbb{Z}(a)) \rightarrow {_n}H^a(K_v,T) \rightarrow 0,\\
\begin{split}
0 \rightarrow ({_n}H^{a+1}(K_v,T))^D \rightarrow H^{a+1}(K_v,\check{T} \otimes \mathbb{Z}/n\mathbb{Z}(a))^D \rightarrow \\\rightarrow (H^a(K_v,T)/n)^D \rightarrow 0,
\end{split}\\
\begin{split}
0 \rightarrow ({_n}H^{d+3-a}(K_v, \tilde{T}))^D  \rightarrow (H^{d+2-a}(K_v,\hat{T} \otimes \mathbb{Z}/n\mathbb{Z}(d+1-a) ))^D \rightarrow\\ \rightarrow (H^{d+2-a}(K_v, \tilde{T})/n)^D \rightarrow 0,
\end{split}\\
\begin{split}
0 \rightarrow H^{d+1-a}(K_v, \tilde{T})/n \rightarrow H^{d+1-a}(K_v,\hat{T} \otimes \mathbb{Z}/n\mathbb{Z}(d+1-a) ) \rightarrow \\ \rightarrow {_n}H^{d+2-a}(K_v, \tilde{T}) \rightarrow 0.
\end{split}
\end{gather}
Les suites (5) et (7) s'insèrent dans un diagramme commutatif:
\begin{multline*}
\xymatrix{
0 \ar[r] & H^{a-1}(K_v,T)/n \ar[r] \ar[d] & H^a(K_v,\check{T} \otimes \mathbb{Z}/n\mathbb{Z}(a)) \ar[r] \ar[d]^{\cong} & ...\\
0 \ar[r] & (_nH^{d+3-a}(K_v,\tilde{T}))^D \ar[r] & H^{d+2-a}(K_v,\hat{T} \otimes \mathbb{Z}/n\mathbb{Z}(d+1-a) )^D \ar[r] & ...
}\\
\xymatrix{
... \ar[r] & _nH^a(K_v,T) \ar[r] \ar[d] & 0\\
... \ar[r] & (H^{d+2-a}(K_v,\tilde{T})/n)^D \ar[r] & 0.
}
\end{multline*}
La flèche verticale de droite est donc surjective.\\
De plus, si $L$ est une extension finie de $K_v$ de degré $n_0$ déployant $\hat{T}$, comme $H^{a+1}(L,\mathbb{Z}(a)) = H^{d+2-a}(L,\mathbb{Z}(d+1-a)) = 0$ d'après la conjecture de Beilinson-Lichtenbaum, un argument de restriction-corestriction montre immédiatement que $H^a(K_v,T)$ et $H^{d+2-a}(K_v,\tilde{T})$ sont de $n_0$-torsion. Ainsi, le groupe $H^a(K_v,T) = {_{n_0}}H^a(K_v,T)$ est fini et on a $H^{d+2-a}(K_v,\tilde{T}) \cong H^{d+2-a}(K_v,\tilde{T})/n_0$. Comme $ _{n_0}H^a(K_v,T) \rightarrow (H^{d+2-a}(K_v,\tilde{T})/n_0)^D$ est surjectif, on déduit que $H^{d+2-a}(K_v,\tilde{T})$ est fini. Nous avons donc prouvé pour l'instant que $H^a(K_v,T)$ et $H^{d+2-a}(K_v,\tilde{T})$ sont finis et que le morphisme $H^a(K_v,T) \rightarrow (H^{d+2-a}(K_v,\tilde{T}))^D$ est surjectif. \\
Montrons maintenant que le cardinal de $H^{d+2-a}(K_v,\tilde{T})$ est supérieur ou égal à celui de $H^a(K_v,T)$.
Pour ce faire, remarquons que les suites (6) et (8) s'insèrent dans un diagramme commutatif:
\begin{multline*}
\xymatrix{
0 \ar[r] & H^{d+1-a}(K_v,\tilde{T})/n \ar[r] \ar[d] & H^{d+1-a}(K_v,\hat{T}\otimes \mathbb{Z}/n\mathbb{Z}(d+1-a)) \ar[r] \ar[d]^{\cong} & ... \\
0 \ar[r] & (_nH^{a+1}(K_v,T))^D \ar[r] & H^{a+1}(K_v,\check{T} \otimes \mathbb{Z}/n\mathbb{Z}(a))^D \ar[r] & ...
}\\
\xymatrix{
... \ar[r]  & _nH^{d+2-a}(K_v,\tilde{T}) \ar[r] \ar[d] & 0\\
... \ar[r] & (H^a(K_v,T)/n)^D \ar[r] & 0.
}
\end{multline*}
Par conséquent, le morphisme vertical de droite est surjectif. En choisissant $n=n_0$, on déduit que l'ordre de $H^{d+2-a}(K_v,\tilde{T})$ est supérieur à l'ordre de $H^a(K_v,T)$. On a donc bien un accouplement parfait de groupes finis:
$$H^a(K_v,T) \times H^{d+2-a}(K_v,\tilde{T}) \rightarrow \mathbb{Q}/\mathbb{Z}.$$
\item[(ii)] La propriété (i) fournit immédiatement des isomorphismes $_nH^a(K_v,T) \cong (H^{d+2-a}(K_v,\tilde{T})/n)^D $. Montrons que $H^{d+2-a}(K_v,\tilde{T})/n \cong H^{d+1-a}(K_v,\tilde{T}_t)/n$.\\
De façon tout à fait analogue à (i), on dispose d'un diagramme commutatif à lignes exactes:
\begin{multline*}
\xymatrix{
0 \ar[r] & H^{a-1}(K_v,T)/n \ar[r] \ar[d] & H^a(K_v,\check{T} \otimes \mathbb{Z}/n\mathbb{Z}(a)) \ar[r] \ar[d]^{\cong} & ... \\
0 \ar[r] & (_nH^{d+2-a}(K_v,\tilde{T}_t))^D \ar[r] & H^{d+2-a}(K_v,\hat{T} \otimes \mathbb{Z}/n\mathbb{Z}(d+1-a) )^D \ar[r] & ... 
}\\
\xymatrix{
... \ar[r]  & _nH^a(K_v,T) \ar[r] \ar[d] & 0\\
... \ar[r] & (H^{d+1-a}(K_v,\tilde{T}_t)/n)^D \ar[r] & 0.
}
\end{multline*}
La flèche $_nH^a(K_v,T) \rightarrow (H^{d+1-a}(K_v,\tilde{T}_t)/n)^D$ est donc surjective. L'exploitation du diagramme commutatif:\\
\centerline{\xymatrix{
_nH^a(K_v,T) \ar[r] \ar[rd]& (H^{d+2-a}(K_v,\tilde{T})/n)^D \ar[d]\\
& (H^{d+1-a}(K_v,\tilde{T}_t)/n)^D
}}
permet alors de conclure que la flèche $H^{d+1-a}(K_v,\tilde{T}_t)/n \rightarrow  H^{d+2-a}(K_v,\tilde{T})/n $ est injective.\\
Par ailleurs, la suite exacte courte de complexes $0 \rightarrow \tilde{T} \rightarrow  \hat{T} \otimes \mathbb{Q}(d+1-a) \rightarrow  \tilde{T}_t \rightarrow 0$ fournit une suite exacte de cohomologie $H^{d+1-a}(K_v,\tilde{T}_t) \rightarrow H^{d+2-a}(K_v,\tilde{T}) \rightarrow H^{d+2-a}(K_v,\hat{T} \otimes \mathbb{Q}(d+1-a))$. Or:
\begin{itemize}
\item[$\bullet$] le groupe $H^{d+2-a}(K_v,\hat{T} \otimes \mathbb{Q}(d+1-a))$ est divisible,
\item[$\bullet$] si $L$ est une extension finie de $K_v$ déployant $\hat{T}$, on a $H^{d+2-a}(L, \mathbb{Q}(d+1-a)) \cong H^{d+2-a}_{\text{Zar}}(L, \mathbb{Q}(d+1-a)_{\text{Zar}})=0$ car $\mathbb{Q}(d+1-a)_{\text{Zar}}$ est concentré en degrés $\leq d+1-a$, et donc, par un argument de restriction-corestriction, le groupe $H^{d+2-a}(K_v,\hat{T} \otimes \mathbb{Q}(d+1-a))$ est d'exposant fini.
\end{itemize}
Cela entraîne que $H^{d+2-a}(K_v,\hat{T} \otimes \mathbb{Q}(d+1-a)) = 0$ et donc que l'application $H^{d+1-a}(K_v,\tilde{T}_t) \rightarrow H^{d+2-a}(K_v,\tilde{T})$ est surjective. Par conséquent, la flèche $H^{d+1-a}(K_v,\tilde{T}_t)/n \rightarrow  H^{d+2-a}(K_v,\tilde{T})/n $ est injective et surjective, donc un isomorphisme. 
\item[(iii)] Le premier diagramme de la preuve de (i) montre que les applications $H^{a-1}(K_v,T)/n \rightarrow (_nH^{d+3-a}(K_v,\tilde{T}))^D$ sont des isomorphismes. En passant à la limite projective sur $n$:
$$H^{a-1}(K_v,T)^{\wedge} \cong H^{d+3-a}(K_v,\tilde{T})^D.$$ 
\end{itemize}
\end{proof}

\begin{proposition} \label{nature tore}
On rappelle que l'on a supposé (H \ref{40}). Soit $U$ un ouvert non vide de $U_0$.
\begin{itemize}
\item[(i)] Le groupe $\Sha^a(T)$ est d'exposant fini.
\item[(ii)] Pour $r \geq 0$, les groupes $H^r(U,\tilde{\mathcal{T}}_t)$ sont de torsion de type cofini.
\item[(iii)] Pour $r \geq 0$, les groupes $H^r_c(U,\tilde{\mathcal{T}}_t)$ sont de torsion de type cofini.
\item[(iv)] Les groupes $H^{d+3-a}(U,\tilde{\mathcal{T}})$ et $H^{d+4-a}(U,\tilde{\mathcal{T}})$ sont de torsion de type cofini.
\end{itemize}
\end{proposition}

\begin{proof}
\begin{itemize}
\item[(i)] Soit $L$ une extension finie de $K$ déployant $\hat{T}$. Comme la conjecture de Beilinson-Lichtenbaum impose que $H^{a+1}(L,\mathbb{Z}(a)) = 0$, un argument de restriction-corestriction montre que $H^{a}(K,T)$ est d'exposant fini.
\item[(ii)] Soit $r \geq 0$. Remarquons d'abord que $\tilde{\mathcal{T}}_t$ est un faisceau de torsion, et donc $H^r(U,\tilde{\mathcal{T}}_t)$ est de torsion. De plus, la suite exacte de Kummer fournit une surjection $H^r(U,{_n}\tilde{\mathcal{T}}_t) \rightarrow {_n}H^r(U,\tilde{\mathcal{T}}_t)$ pour chaque entier naturel $n$. Comme $H^r(U,{_n}\tilde{\mathcal{T}}_t)$ est fini pour chaque $n$, on déduit que $H^r(U,\tilde{\mathcal{T}}_t)$ est de torsion de type cofini.
\item[(iii)] La preuve est tout à fait analogue à celle de la propriété (ii).
\item[(iv)] Soit $U' \rightarrow U$ un morphisme fini étale de degré $n$ déployant $\mathcal{T}$. Comme $H^{d+3-a}(U',\mathbb{Q}(d+1-a))$ et $H^{d+4-a}(U',\mathbb{Q}(d+1-a))$ sont nuls, un argument de restriction-corestriction prouve que $H^{d+3-a}(U,\hat{\mathcal{T}} \otimes \mathbb{Q}(d+1-a))$ et $H^{d+4-a}(U,\hat{\mathcal{T}} \otimes \mathbb{Q}(d+1-a))$ sont d'exposant $n$. De plus, ce sont des groupes divisibles. Ils sont donc nuls. Par conséquent, les suites exactes:
$$H^{d+2-a}(U,\tilde{\mathcal{T}}_t) \rightarrow H^{d+3-a}(U,\tilde{\mathcal{T}}) \rightarrow H^{d+3-a}(U,\hat{\mathcal{T}} \otimes \mathbb{Q}(d+1-a)),$$
$$H^{d+3-a}(U,\tilde{\mathcal{T}}_t) \rightarrow H^{d+4-a}(U,\tilde{\mathcal{T}}) \rightarrow H^{d+4-a}(U,\hat{\mathcal{T}} \otimes \mathbb{Q}(d+1-a))$$
fournissent des surjections $H^{d+2-a}(U,\tilde{\mathcal{T}}_t) \rightarrow H^{d+3-a}(U,\tilde{\mathcal{T}})$ et $H^{d+3-a}(U,\tilde{\mathcal{T}}_t) \rightarrow H^{d+4-a}(U,\tilde{\mathcal{T}})$. On en déduit que $H^{d+3-a}(U,\tilde{\mathcal{T}})$ et $H^{d+4-a}(U,\tilde{\mathcal{T}})$ sont de torsion de type cofini. 
\end{itemize}
\end{proof}

Pour $U$ ouvert non vide de $U_0$, posons $\mathcal{D}^{d+2-a}(U,\tilde{\mathcal{T}}_t) = \text{Im} (H^{d+2-a}_c(U,\tilde{\mathcal{T}}_t) \rightarrow H^{d+2-a}(K,\tilde{T}_t))$ et $\mathcal{D}^{d+3-a}(U,\tilde{\mathcal{T}}) = \text{Im} (H^{d+3-a}_c(U,\tilde{\mathcal{T}}) \rightarrow H^{d+3-a}(K,\tilde{T}))$.

\begin{proposition} \label{Sha tore}
On rappelle que l'on a supposé (H \ref{40}). Pour chaque nombre premier $l$, il existe un ouvert $U_1$ de $U_0$ tel que, pour tout ouvert $U \subseteq U_1$, on a:
$$\mathcal{D}^{d+2-a}(U,\tilde{\mathcal{T}}_t)\{ l \} = \mathcal{D}^{d+2-a}(U_1,\tilde{\mathcal{T}}_t)\{ l \} = \Sha^{d+2-a}(\tilde{T}_t)\{ l \},$$
$$\mathcal{D}^{d+3-a}(U,\tilde{\mathcal{T}})\{ l \} = \mathcal{D}^{d+3-a}(U_1,\tilde{\mathcal{T}})\{ l \} = \Sha^{d+3-a}(\tilde{T})\{ l \}.$$
\end{proposition}

\begin{proof}
On commence par remarquer que, pour $v \in X^{(1)}$, on a un diagramme commutatif:\\
\centerline{\xymatrix{
H^{d+3-a}(K_v^h,\tilde{T}) \ar[r] \ar[d] & H^{d+3-a}(K_v,\tilde{T})\ar[d]\\
H^{d+2-a}(K_v^h,\tilde{T}_t) \ar[r] & H^{d+2-a}(K_v,\tilde{T}_t)
}}
Comme $H^{d+2-a}(K_v^h,\tilde{T} \otimes \mathbb{Q})$, $H^{d+3-a}(K_v^h,\tilde{T} \otimes \mathbb{Q})$, $H^{d+2-a}(K_v,\tilde{T} \otimes \mathbb{Q})$ et $H^{d+3-a}(K_v,\tilde{T} \otimes \mathbb{Q})$ sont nuls, on déduit que les flèches verticales sont des isomorphismes. De plus, la flèche $H^{d+2-a}(K_v^h,\tilde{T}_t) \rightarrow H^{d+2-a}(K_v,\tilde{T}_t)$ est aussi un isomorphisme puisque pour chaque $n$ le morphisme $H^{d+2-a}(K_v^h,\tilde{T} \otimes \mathbb{Z}/n\mathbb{Z}) \rightarrow H^{d+2-a}(K_v,\tilde{T}  \otimes \mathbb{Z}/n\mathbb{Z})$ est un isomorphisme. On en déduit que la restriction $H^{d+3-a}(K_v^h,\tilde{T}) \rightarrow H^{d+3-a}(K_v,\tilde{T})$ est un isomorphisme.\\
Dans la suite de cette preuve, les commutativités des diagrammes que nous utilisons implicitement sont contenues dans la proposition 4.3 de \cite{CTH}. \\
Fixons maintenant un nombre premier $l$. Pour des ouverts $V \subseteq V'$ de $U$, on a $\mathcal{D}^{d+2-a}(V,\tilde{\mathcal{T}}_t) \subseteq \mathcal{D}^{d+2-a}(V',\tilde{\mathcal{T}}_t)$ et $\mathcal{D}^{d+3-a}(V,\tilde{\mathcal{T}}) \subseteq \mathcal{D}^{d+3-a}(V',\tilde{\mathcal{T}})$. Or, grâce à la proposition précédente, on sait que $\mathcal{D}^{d+2-a}(U,\tilde{\mathcal{T}}_t)$ et $\mathcal{D}^{d+3-a}(U,\tilde{\mathcal{T}})$ sont de torsion de type cofini. Par conséquent, il existe un ouvert $U_1$ de $U$ tel que, pour tout ouvert $V \subseteq U_1$, on a $\mathcal{D}^{d+2-a}(V,\tilde{\mathcal{T}}_t)\{l\} = \mathcal{D}^{d+2-a}(U_1,\tilde{\mathcal{T}}_t)\{l\}$ et $\mathcal{D}^{d+3-a}(V,\tilde{\mathcal{T}})\{l\} = \mathcal{D}^{d+2-a}(U_1,\tilde{\mathcal{T}})\{l\}$. Une chasse au diagramme permet alors d'établir que $\mathcal{D}^{d+2-a}(U_1,\tilde{\mathcal{T}}_t)\{l\} = \Sha^{d+2-a}(\tilde{T}_t)\{ l \}$ et $\mathcal{D}^{d+3-a}(U_1,\tilde{\mathcal{T}})\{l\} = \Sha^{d+3-a}(\tilde{T})\{ l \}$.
\end{proof}

\begin{remarque}
La proposition précédente impose en particulier que $\Sha^{d+2-a}(\tilde{T}_t)$ est de torsion de type cofini. De plus, comme $\Sha^{d+2-a}(\tilde{T}_t) = \Sha^{d+3-a}(\tilde{T})$, cela prouve aussi que $\Sha^{d+3-a}(\tilde{T})$ est de torsion de type cofini.
\end{remarque}

\begin{proposition}\label{suite tore}
On rappelle que l'on a supposé (H \ref{40}). Soit $U$ un ouvert non vide de $U_0$. La suite $$ \bigoplus_{v \in X^{(1)}} H^{d+1-a}(K_v,\tilde{T}_t) \rightarrow H^{d+2-a}_c(U,\tilde{\mathcal{T}}_t) \rightarrow \mathcal{D}^{d+2-a}(U,\tilde{\mathcal{T}}_t) \rightarrow 0$$
est exacte.
\end{proposition}

\begin{proof}
Soit $m>0$ un entier. On sait que $\hat{\mathcal{T}}$ représente un faisceau localement constant libre de type fini sur $U$, et donc, comme $k$ est de caractéristique 0, le faisceau $\hat{\mathcal{T}} \otimes \mathbb{Z}/m\mathbb{Z}(d+1-a)$ est localement constant à tiges finies sur $U$. Il est donc représentable par un schéma en groupes fini étale d'après la proposition V.1.1 de \cite{MilEC}, et grâce à la proposition \ref{suite fini}, on en déduit une suite exacte:
\begin{align*} \bigoplus_{v \in X^{(1)}} H^{d+1-a}(K_v,\hat{T} \otimes \mathbb{Z}/m\mathbb{Z}(d+1-a)) &\rightarrow H^{d+2-a}_c(U,\hat{\mathcal{T}} \otimes \mathbb{Z}/m\mathbb{Z}(d+1-a)) \\& \rightarrow \mathcal{D}^{d+2-a}(U,\hat{\mathcal{T}} \otimes \mathbb{Z}/m\mathbb{Z}(d+1-a)) \rightarrow 0.
\end{align*}
Il suffit alors de passer à la limite inductive sur $m$.
\end{proof}

\begin{theorem} \textbf{(Dualité de Poitou-Tate pour les tores)} \label{PT tore}\\
On rappelle que l'on a supposé (H \ref{40}), c'est-à-dire que $X$ est une courbe. Sous de telles hypothèses, $a \in \{0,1,...,d+1\}$, $T = \check{T} \otimes^{\mathbf{L}} \mathbb{Z}(a)[1]$ et $\tilde{T} = \hat{T} \otimes^{\mathbf{L}} \mathbb{Z}(d+1-a)$. On a un accouplement parfait de groupes finis:
$$\Sha^a(T) \times \overline{\Sha^{d+3-a}( \tilde{T})} \rightarrow \mathbb{Q}/\mathbb{Z}.$$
\end{theorem}

\begin{proof}
Fixons $l$ un nombre premier. Pour $U$ ouvert de $U_1$, on pose $D^a_{sh}(U,\mathcal{T}) = \text{Ker} (H^a(U,\mathcal{T}) \rightarrow \prod_{v \in X^{(1)}} H^a(K_v,T))$. D'une part, on a la suite exacte:
$$ 0 \rightarrow D^a_{sh}(U,\mathcal{T})\{l\} \rightarrow H^a(U,\mathcal{T})\{l\} \rightarrow \prod_{v \in X^{(1)}} H^a(K_v,T)\{l\}.$$
Comme $\prod_{v \in X^{(1)}} H^a(K_v,T)\{l\}$ est d'exposant fini d'après la conjecture de Beilinson-Lichtenbaum, sa partie divisible est triviale, et donc la partie divisible de $H^a(U,\mathcal{T})\{l\}$ est contenue dans celle de $D^a_{sh}(U,\mathcal{T})\{l\}$. On en déduit une suite exacte:
$$ 0 \rightarrow \overline{D^a_{sh}(U,\mathcal{T})\{l\}} \rightarrow \overline{H^a(U,\mathcal{T})\{l\}} \rightarrow \prod_{v \in X^{(1)}} H^a(K_v,T)\{l\}.$$
D'autre part, d'après \ref{suite tore}, on a la suite exacte:
$$ \bigoplus_{v \in X^{(1)}} H^{d+1-a}(K_v,\tilde{T}_t) \rightarrow H^{d+2-a}_c(U,\tilde{\mathcal{T}}_t) \rightarrow \mathcal{D}^{d+2-a}(U,\tilde{\mathcal{T}}_t) \rightarrow 0.$$
Comme $H^{d+2-a}_c(U,\tilde{\mathcal{T}}_t)$ est un groupe de torsion de type cofini:
$$ \bigoplus_{v \in X^{(1)}} \overline{H^{d+1-a}(K_v,\tilde{T}_t)} \rightarrow \overline{H^{d+2-a}_c(U,\tilde{\mathcal{T}}_t)} \rightarrow \overline{\mathcal{D}^{d+2-a}(U,\tilde{\mathcal{T}}_t)} \rightarrow 0.$$
Avec la proposition 4.3 de \cite{CTH}, on obtient donc un diagramme commutatif à lignes exactes:\\
\footnotesize
\centerline{\xymatrix{
0 \ar[r]& \overline{D^a_{sh}(U,\mathcal{T})\{l\}} \ar[r]\ar[d] & \overline{H^a(U,\mathcal{T})\{l\}} \ar[r]\ar[d]& \prod_{v \in X^{(1)}} H^a(K_v,T)\{l\}\ar[d] \\
0 \ar[r]& (\overline{\mathcal{D}^{d+2-a}(U,\tilde{\mathcal{T}}_t)}\{l\})^D \ar[r] & (\overline{H^{d+2-a}_c(U,\tilde{\mathcal{T}}_t)}\{l\})^D \ar[r]& (\bigoplus_{v \in X^{(1)}} \overline{H^{d+1-a}(K_v,\tilde{T}_t)}\{l\})^D
}}
\normalsize
D'après la proposition \ref{local tore}, la flèche verticale de droite est un isomorphisme. De plus, on a $\overline{H^{d+2-a}_c(U,\tilde{\mathcal{T}}_t)}\{l\} = H^{d+2-a}_c(U,\tilde{\mathcal{T}}_t)^{(l)}\{l\}$, et donc le morphisme vertical du milieu est un isomorphisme d'après \ref{AV tore}. On en déduit que $\overline{D^a_{sh}(U,\mathcal{T})\{l\}} \cong (\overline{\mathcal{D}^{d+2-a}(U,\tilde{\mathcal{T}}_t)}\{l\})^D$. Comme $\Sha^a(T)$ n'admet pas de sous-groupe divisible, en passant à la limite sur $U$ et en utilisant la propriété \ref{Sha tore}, on obtient que:
$$\Sha^a(T)\{l\} \cong (\overline{\Sha^{d+2-a}(\tilde{T}_t)}\{l\})^D \cong (\overline{\Sha^{d+3-a}(\tilde{T})}\{l\})^D.$$
\end{proof}

Nous voulons maintenant établir une suite exacte de type Poitou-Tate pour les tores.

\begin{proposition}\label{nr tore}
On rappelle que l'on a supposé (H \ref{40}). Soit $v \in U_0^{(1)}$.
\begin{itemize}
\item[(i)] L'image de $H^a(\mathcal{O}_v,\mathcal{T})$ dans $H^a(K_v,T)$ et l'image de $H^{d+2-a}(\mathcal{O}_v,\tilde{\mathcal{T}})$ dans $H^{d+2-a}(K_v,\tilde{T})$ sont annulateurs l'un de l'autre dans l'accouplement parfait:
$$H^a(K_v,T) \times  H^{d+2-a}(K_v,\tilde{T}) \rightarrow \mathbb{Q}/\mathbb{Z}.$$
\item[(ii)]  Supposons que $a=1$. Alors l'annulateur de $H^{d+2}(\mathcal{O}_v,\tilde{\mathcal{T}}) \subseteq H^{d+2}(K_v,\tilde{T})$ dans l'accouplement parfait:
$$H^{0}(K_v,T)^{\wedge} \times H^{d+2}(K_v,\tilde{T}) \rightarrow \mathbb{Q}/\mathbb{Z}$$
est $H^{0}(\mathcal{O}_v,\mathcal{T})^{\wedge} \subseteq H^{0}(K_v,T)^{\wedge}$.
\end{itemize}
\end{proposition}

\begin{proof}
\begin{itemize}
\item[(i)] Comme $\mathcal{O}_v$ est de dimension 1 et $\mathbb{Q}(d+1)_{\text{Zar}}$ est concentré en degrés $\leq d+1$, on a $H^{d+3}(\mathcal{O}_v,\mathbb{Q}(d+1)) \cong H^{d+3}_{\text{Zar}}(\mathcal{O}_v,\mathbb{Q}(d+1)_{\text{Zar}}) =0$, d'où une surjection $H^{d+2}(\mathcal{O}_v,\mathbb{Q}/\mathbb{Z}(d+1)) \rightarrow H^{d+3}(\mathcal{O}_v,\mathbb{Z}(d+1))$. Or $H^{d+2}(\mathcal{O}_v,\mathbb{Q}/\mathbb{Z}(d+1)) \cong H^{d+2}(k(v),\mathbb{Q}/\mathbb{Z}(d+1)) = 0$ par dimension cohomologique. On en déduit que $H^{d+3}(\mathcal{O}_v,\mathbb{Z}(d+1)) = 0$, ce qui prouve la nullité de l'accouplement:
$H^a(\mathcal{O}_v,\mathcal{T}) \times H^{d+2-a}(\mathcal{O}_v,\tilde{\mathcal{T}}) \rightarrow \mathbb{Q}/\mathbb{Z}$.\\
Reste à montrer que, si $t \in H^a(K_v,T)$ est orthogonal à $\text{Im}(H^{d+2-a}(\mathcal{O}_v,\tilde{\mathcal{T}}) \rightarrow H^{d+2-a}(K_v,\tilde{T}))$, alors $t \in \text{Im}(H^a(\mathcal{O}_v,\mathcal{T}) \rightarrow H^a(K_v,T))$. Considérons donc $t \in H^a(K_v,T)$ orthogonal à $\text{Im}(H^{d+2-a}(\mathcal{O}_v,\tilde{\mathcal{T}}) \rightarrow H^{d+2-a}(K_v,\tilde{T}))$. Fixons $n>0$. L'image de $t$ dans $H^a(K_v,T \otimes \mathbb{Z}/n\mathbb{Z})$ est alors orthogonale à $H^{d+1-a}(\mathcal{O}_v,\tilde{\mathcal{T}} \otimes \mathbb{Z}/n\mathbb{Z}) \subseteq H^{d+1-a}(K_v,\tilde{T} \otimes \mathbb{Z}/n\mathbb{Z})$. Grâce à la proposition \ref{nr fini}, on déduit qu'elle est dans $ H^a(\mathcal{O}_v,T \otimes \mathbb{Z}/n\mathbb{Z})$.\\
Considérons maintenant le diagramme commutatif à lignes exactes:\\
\centerline{\xymatrix{
& H^a(\mathcal{O}_v,\mathcal{T}) \ar[r]\ar[d] &H^a(\mathcal{O}_v,\mathcal{T} \otimes \mathbb{Z}/n\mathbb{Z}) \ar[r]\ar[d] &H^{a+1}(\mathcal{O}_v,\mathcal{T}) \ar[d]\\
H^a(K_v,T) \ar[r]^{\cdot n} & H^a(K_v,T) \ar[r] & H^a(K_v, T \otimes \mathbb{Z}/n\mathbb{Z}) \ar[r] &H^{a+1}(K_v,T)
}}
Montrons que la flèche verticale de droite est injective. Pour ce faire, remarquons que le diagramme suivant est commutatif à lignes exactes:\\
\centerline{\xymatrix{
H^{a}(\mathcal{O}_v,\mathcal{T} \otimes \mathbb{Q}) \ar[d]\ar[r] &H^{a}(\mathcal{O}_v,\mathcal{T} \otimes \mathbb{Q}/\mathbb{Z}) \ar[d]\ar[r] &H^{a+1}(\mathcal{O}_v,\mathcal{T}) \ar[d]\ar[r] &H^{a+1}(\mathcal{O}_v,\mathcal{T} \otimes \mathbb{Q}) \ar[d]\\
H^{a}(K_v,T \otimes \mathbb{Q})\ar[r] &H^{a}(K_v,T \otimes \mathbb{Q}/\mathbb{Z})\ar[r] &H^{a+1}(K_v,T )\ar[r] &H^{a+1}(K_v,T \otimes \mathbb{Q})
}}
Or $H^{a+1}(\mathcal{O}_v,\mathcal{T} \otimes \mathbb{Q}) = H^{a}(K_v,\mathcal{T} \otimes \mathbb{Q}) = 0$ et le morphisme 
$H^{a}(\mathcal{O}_v,\mathcal{T} \otimes \mathbb{Q}/\mathbb{Z}) \rightarrow H^{a}(K_v,T \otimes \mathbb{Q}/\mathbb{Z})$ est injectif d'après la proposition \ref{nr fini}. Une chasse au diagramme prouve alors que $H^{a+1}(\mathcal{O}_v,\mathcal{T}) \rightarrow H^{a+1}(K_v,T)$ est injectif.\\
Cela prouve qu'il existe $t_n \in \text{Im}(H^a(\mathcal{O}_v,\mathcal{T}) \rightarrow H^a(K_v,T))$ tel que $t-t_n \in nH^a(K_v,T)$. On en déduit que l'image de $t$ dans le groupe fini $\text{Coker}(H^a(\mathcal{O}_v,\mathcal{T}) \rightarrow H^a(K_v,T))$ est divisible, donc nulle.
\item[(ii)] Comme dans (i), on montre que l'accouplement $H^{0}(\mathcal{O}_v,\mathcal{T})^{\wedge} \times H^{d+2}(\mathcal{O}_v,\tilde{\mathcal{T}}) \rightarrow \mathbb{Q}/\mathbb{Z}$ est nul. Fixons maintenant $n>0$ et considérons $t \in H^0(K_v,T)/n$ orthogonal à ${_n}H^{d+2}(\mathcal{O}_v,\tilde{\mathcal{T}}) $. L'image de $t$ dans $H^0(K_v,T\otimes\mathbb{Z}/n\mathbb{Z})$ est alors orthogonale à $H^{d+1}(\mathcal{O}_v,\tilde{\mathcal{T}} \otimes \mathbb{Z}/n\mathbb{Z}) \subseteq H^{d+1}(K_v,\tilde{T}\otimes \mathbb{Z}/n\mathbb{Z}))$ et appartient donc à $H^0(\mathcal{O}_v,\mathcal{T}\otimes\mathbb{Z}/n\mathbb{Z})$. Écrivons le diagramme commutatif à lignes exactes:\\
\centerline{\xymatrix{
0 \ar[r]& H^0(\mathcal{O}_v, \mathcal{T})/n \ar[r]\ar[d] &H^0(\mathcal{O}_v, \mathcal{T} \otimes \mathbb{Z}/n\mathbb{Z}) \ar[r]\ar[d] &H^1(\mathcal{O}_v, \mathcal{T}) \ar[d]\\
0 \ar[r] & H^0(K_v,T)/n \ar[r] & H^0(K_v,T \otimes \mathbb{Z}/n\mathbb{Z}) \ar[r] \ar[r]& H^1(K_v,T)
}}
Le morphisme $H^1(\mathcal{O}_v,\mathcal{T}) \rightarrow H^1(K_v,T)$ s'identifie à un morphisme d'inflation en cohomologie galoisienne: il est donc injectif. De plus, $H^0(\mathcal{O}_v, \mathcal{T} \otimes \mathbb{Z}/n\mathbb{Z}) \rightarrow H^0(K_v,T \otimes \mathbb{Z}/n\mathbb{Z})$ est injectif. Une chasse au diagramme permet alors d'établir que $H^{0}(\mathcal{O}_v,\mathcal{T})/n \rightarrow H^{0}(K_v,T)/n$ est injectif et que $t \in H^0(\mathcal{O}_v, \mathcal{T})/n$, ce qui achève la preuve.
\end{itemize}
\end{proof}

Dans la suite de cette section, on supposera que
\begin{hypo}\label{40'}
\begin{minipage}[t]{12.72cm}
$a=1$, c'est-à-dire $T = \check{T} \otimes^{\mathbf{L}} \mathbb{Z}(1)[1]$ est quasi-isomorphe à un tore et $\tilde{T} = \hat{T} \otimes^{\mathbf{L}} \mathbb{Z}(d)$.
\end{minipage}
\end{hypo}
Pour chaque $r \geq 0$, on notera $\mathbb{P}^{r}(T)$ le produit restreint des $H^r(K_v,T)$ pour $v \in X^{(1)}$ par rapport aux $H^r_{\text{nr}}(K_v,T) = \text{Im}(H^r(\mathcal{O}_v,\mathcal{T}) \rightarrow H^r(K_v,T))$, muni de sa topologie naturelle. De même, on notera $\mathbb{P}^{r}(\tilde{T})$ le produit restreint des $H^r(K_v,\tilde{T})$ pour $v \in X^{(1)}$ par rapport aux $H^r_{\text{nr}}(K_v,\tilde{T}) = \text{Im}(H^r(\mathcal{O}_v,\tilde{\mathcal{T}}) \rightarrow H^r(K_v,\tilde{T}))$, muni de sa topologie naturelle. Il convient de munir $\mathbb{P}^{d+2}(\tilde{T})_{tors}$ d'une topologie qui n'est pas induite par la topologie produit restreint sur $\mathbb{P}^{d+2}(\tilde{T})$. Pour ce faire, remarquons que, pour chaque $n>0$, le groupe ${_n}\mathbb{P}^{d+2}(\tilde{T})$ est le produit restreint des ${_n}H^{d+2}(K_v,\tilde{T})$ par rapport aux ${_n}H^{d+2}(\mathcal{O}_v,\tilde{\mathcal{T}}) \subseteq {_n}H^{d+2}(K_v,\tilde{T})$. Nous pouvons donc le munir de sa topologie produit restreint. La topologie sur $\mathbb{P}^{d+2}(\tilde{T})_{tors} = \varinjlim {_n}\mathbb{P}^{d+2}(\tilde{T})$ sera alors tout simplement la topologie limite inductive.

\begin{corollary}\label{dual P tore}
On rappelle que l'on a supposé (H \ref{40}) et (H \ref{40'}). La dualité locale induit un accouplement parfait $\mathbb{P}^1(T) \times \mathbb{P}^{d+1}(\tilde{T}) \rightarrow \mathbb{Q}/\mathbb{Z}$ et un isomorphisme $\mathbb{P}^{0}(T)_{\wedge}\rightarrow (\mathbb{P}^{d+2}(\tilde{T})_{tors})^D $.
\end{corollary}

\begin{proof}
La première affirmation découle immédiatement de l'assertion (i) de \ref{nr tore}. Quant à la seconde, on remarque que, pour chaque $n >0$, $\mathbb{P}^0(T)/n$ s'identifie au produit restreint des $H^0(K_v,T)/n$ par rapport aux $H^0(\mathcal{O}_v,\mathcal{T})/n$ puisque $H^0(\mathcal{O}_v,\mathcal{T})/n \subseteq H^0(K_v,T)/n$. On déduit de l'assertion (ii) de \ref{nr tore} que $\mathbb{P}^0(T)/n \cong ({_n}\mathbb{P}^{d+2}(\tilde{T}))^D$. En passant à la limite projective sur $n$, on obtient l'isomorphisme $\mathbb{P}^{0}(T)_{\wedge} \rightarrow (\mathbb{P}^{d+2}(\tilde{T})_{tors})^D$.
\end{proof}

\begin{proposition} \label{bouts tore}
On rappelle que l'on a supposé (H \ref{40}) et (H \ref{40'}).
\begin{itemize}
\item[(i)] On a une suite exacte:
\begin{align*}
H^{d+2}(K,\tilde{T}) \rightarrow \mathbb{P}^{d+2}(\tilde{T})_{tors}& \rightarrow (H^{0}(K,T)_{\wedge})^D \rightarrow H^{d+3}(K,\tilde{T}) \rightarrow \\ &\rightarrow \mathbb{P}^{d+3}(\tilde{T})_{tors} \rightarrow  (\varprojlim_n {_n}T(K^s))^D \rightarrow 0.
\end{align*}
\item[(ii)] Supposons que $\Sha^2(T)$ est fini. Alors on a une suite exacte:
$$H^1(K,T) \rightarrow \mathbb{P}^1(T) \rightarrow H^{d+1}(K,\tilde{T})^D.$$
\end{itemize}
\end{proposition}

\begin{proof}
\begin{itemize}
\item[(i)] Soit $n>0$. D'après \ref{PT fini}, on dispose d'une suite exacte:
\begin{align*}
&H^{d+1}(K, \tilde{T} \otimes \mathbb{Z}/n\mathbb{Z})  \rightarrow \mathbb{P}^{d+1}(\tilde{T} \otimes \mathbb{Z}/n\mathbb{Z}) \rightarrow H^{0}(K,T \otimes \mathbb{Z}/n\mathbb{Z})^D \rightarrow \\
& \rightarrow  H^{d+2}(K, \tilde{T} \otimes \mathbb{Z}/n\mathbb{Z}) \rightarrow \mathbb{P}^{d+2}(\tilde{T} \otimes \mathbb{Z}/n\mathbb{Z}) \rightarrow H^{-1}(K,T \otimes \mathbb{Z}/n\mathbb{Z})^D \rightarrow 0
\end{align*}
On en déduit un diagramme commutatif dont la première ligne est exacte:\\
\scriptsize
\begin{multline*}
\xymatrix{
\varinjlim_n H^{d+1}(K, \tilde{T} \otimes \mathbb{Z}/n\mathbb{Z}) \ar[d] \ar[r] & \varinjlim_n \mathbb{P}^{d+1}(\tilde{T} \otimes \mathbb{Z}/n\mathbb{Z}) \ar[d] \ar[r] & \varinjlim_n H^{0}(K,T \otimes \mathbb{Z}/n\mathbb{Z})^D\ar[d]\ar[r]& ... \\
H^{d+2}(K,\tilde{T}) \ar[r] & \mathbb{P}^{d+2}(\tilde{T})_{tors} \ar[r] & (H^{0}(K,T)_{\wedge})^D\ar[r] & ...
}\\
\xymatrix{
... \ar[r]& \varinjlim_n H^{d+2}(K, \tilde{T} \otimes \mathbb{Z}/n\mathbb{Z}) \ar[d] \ar[r] & \varinjlim_n \mathbb{P}^{d+2}(\tilde{T} \otimes \mathbb{Z}/n\mathbb{Z}) \ar[d]\ar[r] &  \varinjlim_n H^{-1}(K,T \otimes \mathbb{Z}/n\mathbb{Z})^D\ar@{=}[d]\ar[r]& 0\\
... \ar[r] & H^{d+3}(K,\tilde{T}) \ar[r] & \mathbb{P}^{d+3}(\tilde{T})_{tors} \ar[r] &  \varinjlim_n H^{-1}(K,T \otimes \mathbb{Z}/n\mathbb{Z})^D\ar[r]& 0
}
\end{multline*}
\normalsize
où:
\begin{itemize}
\item[$\bullet$] le morphisme $\mathbb{P}^{d+2}(\tilde{T})_{tors} \rightarrow (H^{0}(K,T)_{\wedge})^D$ est le dual de $H^{0}(K,T)_{\wedge} \rightarrow \mathbb{P}^{0}(T)_{\wedge}$  (voir le corollaire \ref{dual P tore}).
\item[$\bullet$] le morphisme $\varinjlim_n H^{d+1}(K, \tilde{T} \otimes \mathbb{Z}/n\mathbb{Z}) \rightarrow H^{d+2}(K,\tilde{T})$ est un isomorphisme puisqu'il s'insère dans une suite exacte:
$$H^{d+1}(K, \tilde{T} \otimes \mathbb{Q}) \rightarrow \varinjlim_n H^{d+1}(K, \tilde{T} \otimes \mathbb{Z}/n\mathbb{Z}) \rightarrow H^{d+2}(K,\tilde{T}) \rightarrow H^{d+2}(K, \tilde{T} \otimes \mathbb{Q}).$$
et $H^{d+1}(K, \tilde{T} \otimes \mathbb{Q})$ et $H^{d+2}(K, \tilde{T} \otimes \mathbb{Q})$ sont nuls. De même, le morphisme  $\varinjlim_n H^{d+2}(K, \tilde{T} \otimes \mathbb{Z}/n\mathbb{Z}) \rightarrow H^{d+3}(K,\tilde{T})$ est un isomorphisme.
\item[$\bullet$] le morphisme $\varinjlim_n \mathbb{P}^{d+1}(\tilde{T} \otimes \mathbb{Z}/n\mathbb{Z}) \rightarrow \mathbb{P}^{d+2}(\tilde{T})_{tors}$ est surjectif puisque $H^{d+1}(K_v, \tilde{T} \otimes \mathbb{Z}/n\mathbb{Z}) \rightarrow {_n}H^{d+2}(K_v,\tilde{T})$ (pour $v \in X^{(1)}$ et $n>0$) et $H^{d+1}(\mathcal{O}_v, \tilde{\mathcal{T}} \otimes \mathbb{Z}/n\mathbb{Z}) \rightarrow {_n}H^{d+2}(\mathcal{O}_v,\tilde{\mathcal{T}})$ (pour $v \in U_0^{(1)}$ et $n>0$) sont surjectifs.
\item[$\bullet$] le morphisme $\varinjlim_n H^{0}(K,T \otimes \mathbb{Z}/n\mathbb{Z})^D \rightarrow (H^{0}(K,T)_{\wedge})^D$ est obtenu à partir des morphismes $H^{0}(K,T)/n \rightarrow H^{0}(K,T \otimes \mathbb{Z}/n\mathbb{Z})$ par dualisation et passage à la limite sur $n$. Or ce dernier morphisme s'insère dans la suite exacte:
$$0 \rightarrow  H^{0}(K,T)/n \rightarrow H^{0}(K,T \otimes \mathbb{Z}/n\mathbb{Z}) \rightarrow {_n}H^{1}(K,T) \rightarrow 0.$$
Comme $H^{1}(K,T)$ est d'exposant fini, cela impose que $\varinjlim_n (H^{0}(K,T \otimes \mathbb{Z}/n\mathbb{Z})^D) \rightarrow (H^{0}(K,T)_{\wedge})^D$ est un isomorphisme.
\item[$\bullet$] le morphisme $\varinjlim_n \mathbb{P}^{d+2}(\tilde{T} \otimes \mathbb{Z}/n\mathbb{Z}) \rightarrow \mathbb{P}^{d+3}(\tilde{T})_{tors}$ est un isomorphisme puisqu'il s'identifie au morphisme $\bigoplus_{v \in X^{(1)}} H^{d+2}(K_v,\tilde{T} \otimes \mathbb{Q}/\mathbb{Z}) \rightarrow \bigoplus_{v \in X^{(1)}} H^{d+3}(K_v,\tilde{T})$.
\item[$\bullet$] les morphismes $ (H^{0}(K,T)_{\wedge})^D \rightarrow H^{d+3}(K,\tilde{T})$ et $\mathbb{P}^{d+3}(\tilde{T})_{tors} \rightarrow \varinjlim_n H^{-1}(K,T \otimes \mathbb{Z}/n\mathbb{Z})^D$ sont les seuls qui font commuter le diagramme.
\end{itemize}
Une chasse au diagramme prouve alors que la suite 
\begin{align*}
H^{d+2}(K,\tilde{T}) \rightarrow \mathbb{P}^{d+2}(\tilde{T})_{tors}& \rightarrow (H^{0}(K,T)_{\wedge})^D \rightarrow H^{d+3}(K,\tilde{T}) \rightarrow \\ &\rightarrow \mathbb{P}^{d+3}(\tilde{T})_{tors} \rightarrow  (\varprojlim_n {_n}T(K^s))^D \rightarrow 0
\end{align*}
 est exacte.
\item[(ii)] Nous sommes dans le cas $a=1$, ce qui impose que $T$ et $\mathcal{T}$ sont quasi-isomorphes à des tores. Pour chaque ouvert non vide $V \subseteq U_0$, on dispose d'une suite exacte:
$$H^1_c(V,\mathcal{T}) \rightarrow H^1(V,\mathcal{T}) \rightarrow \bigoplus_{v \in X \setminus V} H^1(K_v^h,T) \rightarrow H^{2}_c(V,\mathcal{T}).$$
Pour $v \in U_0^{(1)}$, on a un isomorphisme $H^1(K_v,T) \cong H^1(K_v^h,T)$ (on pourra aller voir le corollaire 3.2 de \cite{HS1} ou le lemme 2.7 de \cite{HS4}). Par conséquent, on obtient, pour chaque ouvert non vide $V$ de $U_0$, une suite exacte:
$$H^1_c(V,\mathcal{T}) \rightarrow H^1(V,\mathcal{T}) \rightarrow \bigoplus_{v \in X \setminus V} H^1(K_v,T) \rightarrow H^{2}_c(V,\mathcal{T}).$$
Soient maintenant $U$ et $V$ deux ouverts de $U_0$ tels que $V \subseteq U$. D'après le lemme de Harder (corollaire A.8 de \cite{GP}), un élément de $H^1(V,\mathcal{T})$ provient de $H^1(U,\mathcal{T})$ si, et seulement si, son image dans $H^1(K_v,T)$ provient de $H^1(\mathcal{O}_v,\mathcal{T})$ pour chaque $v \in U \setminus V$. On obtient donc une suite exacte:
$$H^1(U,\mathcal{T}) \rightarrow \bigoplus_{v \in X \setminus U} H^1(K_v,T) \oplus \bigoplus_{v \in U \setminus V} H^1(\mathcal{O}_v,\mathcal{T}) \rightarrow H^{2}_c(V,\mathcal{T}).$$
Soit $m>0$ le degré d'une extension déployant $T$. Le groupe $H^1(K_v,T)$ étant de $m$-torsion pour $v \in X^{(1)}$, on obtient une suite exacte de groupes finis:
$$H^1(U,\mathcal{T})/m \rightarrow \bigoplus_{v \in X \setminus U} H^1(K_v,T) \oplus \bigoplus_{v \in U \setminus V} H^1(\mathcal{O}_v,\mathcal{T}) \rightarrow {_m}H^{2}_c(V,\mathcal{T}).$$
Grâce à la proposition \ref{Sha tore}, choisissons $V$ assez petit pour que $\mathcal{D}^2(V,\mathcal{T})\{l\} \cong \Sha^2(T)\{l\}$ pour chaque nombre premier $l$ divisant $m$, et montrons que pour un tel nombre premier le groupe $H^2_c(V,\mathcal{T})\{l\}$ est fini. \\
On sait que le groupe $H^2_c(V,\mathcal{T})\{l\}$ est de type cofini et qu'il s'insère dans une suite exacte:
$$0 \rightarrow \text{Ker}(H^2_c(V,\mathcal{T}) \rightarrow H^2(K,T)) \rightarrow H^2_c(V,\mathcal{T}) \rightarrow \mathcal{D}^2(V,\mathcal{T}) \rightarrow 0.$$
Le groupe $\mathcal{D}^2(V,\mathcal{T})\{l\}$ étant fini, il suffit donc de montrer que $\text{Ker}(H^2_c(V,\mathcal{T}) \rightarrow H^2(K,T))$ est d'exposant fini. Considérons alors un morphisme fini étale $\tilde{V} \rightarrow V$ de degré $m$ avec $\tilde{V}$ intègre déployant $\mathcal{T}$. Notons $\tilde{K}$ le corps des fonctions de $\tilde{V}$. Comme $\tilde{V}$ est régulier, le morphisme $H^2(\tilde{V}, \mathcal{T}) \rightarrow H^2(\tilde{K},T) $ est injectif (voir la proposition II.4.5.3 de \cite{Tam}), et un argument de restriction-corectriction montre que $\text{Ker}(H^2(V,\mathcal{T}) \rightarrow H^2(K,T))$ est d'exposant $m$. De plus, le théorème de Hilbert 90 impose que, pour chaque $v \in X \setminus V$, le groupe $H^1(K_v,T)$ est d'exposant $m$. Par conséquent, la suite exacte $\bigoplus_{v \in X \setminus V} H^1(K_v,T) \rightarrow H^2_c(V,\mathcal{T}) \rightarrow H^2(V,\mathcal{T})$ permet de conclure que $\text{Ker}(H^2_c(V,\mathcal{T}) \rightarrow H^2(K,T))$ est d'exposant $m^2$, ce qui établit la finitude de $H^2_c(V,\mathcal{T})\{l\}$ pour chaque premier $l$ divisant $m$.\\
En utilisant la remarque \ref{AV tore 2} et en observant que $H^{d+1}(V,\tilde{\mathcal{T}})\{l\}$ est de type cofini, on obtient une injection:
$${_m}H^{2}_c(V,\mathcal{T})\hookrightarrow \prod_{l \mid m} H^{d+1}(V,\tilde{\mathcal{T}})\{l\}^D,$$
d'où une suite exacte:
$$H^1(U,\mathcal{T})/m \rightarrow \bigoplus_{v \in X \setminus U} H^1(K_v,T) \oplus \bigoplus_{v \in U \setminus V} H^1(\mathcal{O}_v,\mathcal{T}) \rightarrow \prod_{l \mid m} H^{d+1}(V,\tilde{\mathcal{T}})\{l\}^D.$$
Comme $H^{d+1}(K,\tilde{T})$ est de $m$-torsion d'après la conjecture de Beilinson-Lichtenbaum, en passant à la limite projective sur $V$ on déduit la suite exacte:
$$H^1(U,\mathcal{T}) \rightarrow \prod_{v \in X \setminus U} H^1(K_v,T) \times \prod_{v \in U} H^1(\mathcal{O}_v,\mathcal{T}) \rightarrow H^{d+1}(K,\tilde{T})^D.$$
Il suffit alors de passer à la limite inductive sur $U$ pour conclure.
\end{itemize}
\end{proof}

\begin{lemma} \label{x}
On rappelle que l'on a supposé (H \ref{40}) et (H \ref{40'}). Soient $V \subseteq U$ deux ouverts non vides de $U_0$. Soit $\alpha \in H^{d+2}(V,\tilde{\mathcal{T}})$. Si, pour chaque $v \in U \setminus V$, l'image de $\alpha$ dans $H^{d+2}(K_v,\tilde{T})$ provient de $H^{d+2}(\mathcal{O}_v,\tilde{\mathcal{T}})$, alors $\alpha$ provient de $H^{d+2}(U,\tilde{\mathcal{T}})$.
\end{lemma}

\begin{proof}
Comme $H^{d+2}(V,\tilde{\mathcal{T}} \otimes \mathbb{Q})$ est nul, on a une surjection $H^{d+1}(V,\tilde{\mathcal{T}} \otimes \mathbb{Q}/\mathbb{Z}) \rightarrow H^{d+2}(V,\tilde{\mathcal{T}})$. On en déduit que $H^{d+2}(V,\tilde{\mathcal{T}})$ est de torsion. Soit donc $n$ tel que $\alpha$ est de $n$-torsion et choisissons $\alpha_n \in H^{d+1}(V,\tilde{\mathcal{T}} \otimes \mathbb{Z}/n\mathbb{Z})$ relevant $\alpha$. Pour $v \in U \setminus V$, soit $\alpha_{n,v}$ l'image de $\alpha_n$ dans $H^{d+1}(K_v,\tilde{T} \otimes \mathbb{Z}/n\mathbb{Z})$. Par hypothèse, pour chaque $v \in U \setminus V$, l'image de $\alpha_{n,v}$ dans $H^{d+2}(K_v,\tilde{T})$ provient de $H^{d+2}(\mathcal{O}_v,\tilde{\mathcal{T})}$. Comme $H^{d+2}(\mathcal{O}_v,\tilde{\mathcal{T}} \otimes \mathbb{Q})=H^{d+2}(K_v,\tilde{T} \otimes \mathbb{Q}) = H^{d+1}(K_v,\tilde{T} \otimes \mathbb{Q})=0$, on dispose d'un diagramme commutatif:\\
\centerline{\xymatrix{
\varinjlim_m H^{d+1}(\mathcal{O}_v, \tilde{\mathcal{T}} \otimes \mathbb{Z}/m\mathbb{Z}) \ar@{->>}[r] \ar[d] & H^{d+2}(\mathcal{O}_v,\tilde{\mathcal{T}})\ar[d]\\
\varinjlim_m H^{d+1}(K_v,\tilde{T} \otimes \mathbb{Z}/m\mathbb{Z}) \ar[r]^-{\cong} & H^{d+1}(K_v,\tilde{T})
}}
On en déduit qu'il existe un multiple $m$ de $n$ tel que, pour chaque $v \in U \setminus V$, l'image de $\alpha_{n,v}$ dans $H^{d+1}(K_v,\tilde{T} \otimes \mathbb{Z}/m\mathbb{Z})$ provient de $H^{d+1}(\mathcal{O}_v,\tilde{\mathcal{T}} \otimes \mathbb{Z}/m\mathbb{Z})$. D'après le lemme 2.2 de \cite{HS2}, cela entraîne que l'image de $\alpha_n$ dans $H^{d+1}(V,\tilde{\mathcal{T}} \otimes \mathbb{Z}/m\mathbb{Z})$ provient de $H^{d+1}(U,\tilde{\mathcal{T}} \otimes \mathbb{Z}/m\mathbb{Z})$, d'où le résultat.
\end{proof}

\begin{lemma}\label{nul dual}
On rappelle que l'on a supposé (H \ref{40}). Le groupe $\Sha^{d+2}(\mathbb{Z}(d))$ est nul si, et seulement si, le groupe $\Sha^2(\mathbb{G}_m)$ l'est aussi.
\end{lemma}

\begin{proof}
Pour chaque entier naturel $n>0$, la conjecture de Beilinson-Lichtenbaum fournit un diagramme commutatif à lignes exactes:\\
\footnotesize
\centerline{\xymatrix{
0 \ar[r] & H^{d+1}(K, \mu_n^{\otimes d}) \ar[r]\ar[d] & H^{d+2}(K,\mathbb{Z}(d)) \ar[r]\ar[d] & H^{d+2}(K,\mathbb{Z}(d))\ar[d]\\
0 \ar[r] & \prod_{v \in X^{(1)}} H^{d+1}(K_v, \mu_n^{\otimes d}) \ar[r] & \prod_{v \in X^{(1)}}H^{d+2}(K_v,\mathbb{Z}(d)) \ar[r] & \prod_{v \in X^{(1)}}H^{d+2}(K_v,\mathbb{Z}(d))
}}
\normalsize
On en déduit que $\Sha^{d+1}(\mu_n^{\otimes d}) \cong {_n}\Sha^{d+2}(\mathbb{Z}(d))$. Comme $\Sha^{d+2}(\mathbb{Z}(d))$ est de torsion, on obtient avec le théorème \ref{PT fini}:
$$\Sha^{d+2}(\mathbb{Z}(d)) \cong \varinjlim_n \Sha^{d+1}(\mu_n^{\otimes d}) \cong \varinjlim_n \Sha^2(\mu_n)^D.$$
Or, par un argument tout à fait analogue à celui qui précède, on a $\Sha^2(\mu_n) \cong {_n}\Sha^2(\mathbb{G}_m)$. Par conséquent:
$$\Sha^{d+2}(\mathbb{Z}(d)) \cong (\varprojlim_n {_n}\Sha^2(\mathbb{G}_m))^D.$$
On en déduit que la partie divisible de $\Sha^{d+2}(\mathbb{Z}(d))$ est nulle si, et seulement si, la partie divisible de $\Sha^2(\mathbb{G}_m)$ est nulle (en fait, les parties divisibles de $\Sha^{d+2}(\mathbb{Z}(d))$ et $\Sha^2(\mathbb{G}_m)$ sont non canoniquement isomorphes). Comme $\Sha^{d+2}(\mathbb{Z}(d))$ et $\Sha^2(\mathbb{G}_m)$ sont divisibles d'après \ref{PT tore}, le lemme en découle.
\end{proof}

\begin{remarque}
Plus généralement, si $a \neq 1$, on prouve que $\Sha^{d+3-a}(\mathbb{Z}(d+1-a))$ est nul si, et seulement si, $\Sha^{a+2}(\mathbb{Z}(a))$ l'est.
\end{remarque}

\begin{theorem} \textbf{(Suite exacte de Poitou-Tate pour les tores)} \label{PT tore exacte}\\
Rappelons que nous avons supposé (H \ref{40}) et (H \ref{40'}), c'est-à-dire que $X$ est une courbe, $T = \check{T} \otimes^{\mathbf{L}} \mathbb{Z}(1)[1]$ est quasi-isomorphe à un tore et $\tilde{T} = \hat{T} \otimes^{\mathbf{L}} \mathbb{Z}(d)$. Soit $L$ une extension finie déployant $T$. Supposons que $\Sha^2(L,\mathbb{G}_m)$ est nul. On dispose alors d'une suite exacte à 7 termes:\\
\centerline{\xymatrix{
\Sha^{d+3}(\tilde{T})^D \ar[d] & 0  \ar[l] & \\
H^0(K,T)_{\wedge} \ar[r]& \mathbb{P}^{0}(T)_{\wedge} \ar[r] & H^{d+2}(K,\tilde{T})^D \ar[d]\\
H^{d+1}(K,\tilde{T})^D & \mathbb{P}^1(T) \ar[l] & H^1(K,T)\ar[l]
}}
\end{theorem}

\begin{remarque}\label{scheiderer}
\begin{itemize}
\item[$(i)$] La nullité du groupe $\Sha^2(\mathbb{G}_m)$ est étudiée dans la section \ref{1.1}. Elle est en particulier établie dans les trois cas suivants:
\begin{itemize}
\item[$\bullet$] $X$ est la droite projective ou une quadrique sur $k$ (théorème 1.2);
\item[$\bullet$] $k_1$ est un corps $p$-adique, $X$ est une courbe de la forme $\text{Proj}(k[x,y,z]/(P(x,y,z))) $ où $P \in \mathcal{O}_{k_1}[x,y,z]$ est un polynôme homogène, et $\text{Proj}(k_0[x,y,z]/(\overline{P}(x,y,z))) $ est une courbe lisse et géométriquement intègre (théorème 1.14);
\item[$\bullet$] $k_0 = \mathbb{C}((t))$, $X$ est la courbe elliptique sur $k$ d'équation $y^2=x^3+Ax+B$ avec $A, B \in k_0$, et la courbe elliptique sur $k_0$ d'équation $y^2=x^3+Ax+B$ admet une réduction modulo $t$ de type additif (théorème 1.14).
\end{itemize}
Toujours dans la section \ref{1.1}, on exhibe des exemples où $\Sha^2(\mathbb{G}_m)\neq 0$ dès que $k$ n'est pas un corps $p$-adique (exemples 1.16 et 1.18).
\item[$(ii)$] Le groupe $\Sha^{d+3}(\tilde{T})$ est étudié dans la section \ref{1.2}. On montre en particulier qu'il est nul si $k$ est un corps $p$-adique et fini si $k = \mathbb{C}((t))((t_1))$.
\end{itemize}
\end{remarque}

\begin{proof}
\begin{itemize}
\item[$\bullet$] L'exactitude de la deuxième ligne découle de la proposition \ref{bouts tore} et de la nullité de $\Sha^2(L,\mathbb{G}_m)$. 
\item[$\bullet$] On dispose de la suite exacte $0 \rightarrow \Sha^{d+2}(\tilde{T}) \rightarrow  H^{d+2}(K,\tilde{T}) \rightarrow \mathbb{P}^{d+2}(\tilde{T})_{tors}$. En utilisant la nullité de $\Sha^2(L,\mathbb{G}_m)$, le lemme \ref{nul dual} et le théorème \ref{PT tore}, on remarque que la suite duale s'écrit:
$$\mathbb{P}^0(T)_{\wedge} \rightarrow H^{d+2}(K,\tilde{T})^D \rightarrow \Sha^1(T) \rightarrow 0.$$
Pour montrer l'exactitude de cette suite, il suffit d'établir que le morphisme $H^{d+2}(K,\tilde{T}) \rightarrow \mathbb{P}^{d+2}(\tilde{T})_{tors}$ est strict, c'est-à-dire que son image $I$ est discrète. Pour ce faire, on fixe $U$ un ouvert de $U_0$ et on remarque grâce au lemme \ref{x} que tout élément de $I \cap (\prod_{v \in X \setminus U} {_n}H^{d+2}(K_v,\tilde{T}) \times\prod_{v \in U} {_n}H^{d+2}(\mathcal{O}_v, \tilde{\mathcal{T}})) \subseteq \mathbb{P}^{d+2}(\tilde{T})_{tors}$ provient de $H^{d+2}(U,\tilde{\mathcal{T}})$. Ce dernier groupe est de torsion de type cofini. Par conséquent, $I \cap (\prod_{v \in X \setminus U} {_n}H^{d+2}(K_v,\tilde{T}) \times\prod_{v \in U} {_n}H^{d+2}(\mathcal{O}_v, \tilde{\mathcal{T}}))$ est de $n$-torsion et de type cofini, donc fini. Par définition de la topologie de $\mathbb{P}^{d+2}(\tilde{T})_{tors}$, $I$ est discret, et la suite $\mathbb{P}^0(T)_{\wedge} \rightarrow H^{d+2}(K,\tilde{T})^D \rightarrow \Sha^1(T) \rightarrow 0$ est exacte.
\item[$\bullet$] On dispose de la suite exacte $H^{d+2}(K,\tilde{T}) \rightarrow \mathbb{P}^{d+2}(\tilde{T})_{tors} \rightarrow (H^{0}(K,T)_{\wedge})^D \rightarrow \Sha^{d+3}(\tilde{T}) \rightarrow 0$.
La suite duale s'écrit:
$$0 \rightarrow \Sha^{d+3}(\tilde{T})^D \rightarrow H^{0}(K,T)_{\wedge} \rightarrow \mathbb{P}^0(T)_{\wedge} \rightarrow H^{d+2}(K,\tilde{T})^D.$$
On a déjà vu que le morphisme $H^{d+2}(K,\tilde{T}) \rightarrow \mathbb{P}^{d+2}(\tilde{T})_{tors}$ est strict. Comme $\Sha^{d+3}(\tilde{T})$ est discret, le morphisme $H^{d+2}(K,\tilde{T})^D \rightarrow \Sha^1(T)$ est aussi strict. Montrons qu'il en est de même du morphisme $\mathbb{P}^{d+2}(\tilde{T})_{tors} \rightarrow (H^{0}(K,T)_{\wedge})^D$. Le groupe localement compact $\mathbb{P}^{d+2}(\tilde{T})_{tors}$ est muni d'une topologie qui en fait une réunion dénombrable  d'espaces compacts, et, le groupe $(H^{0}(K,T)_{\wedge})^D$ étant localement compact, l'image $J$ de $\mathbb{P}^{d+2}(\tilde{T})_{tors} \rightarrow (H^{0}(K,T)_{\wedge})^D$ est localement compacte car elle est fermée dans $(H^{0}(K,T)_{\wedge})^D$. Comme le morphisme $\mathbb{P}^{d+2}(\tilde{T})_{tors} \rightarrow J$ est surjectif, on déduit de \cite{HR} que c'est un morphisme strict. Cela impose immédiatement que $\mathbb{P}^{d+2}(\tilde{T})_{tors} \rightarrow (H^{0}(K,T)_{\wedge})^D$ est strict, et la suite $0 \rightarrow H^{0}(K,T)_{\wedge} \rightarrow \mathbb{P}^0(T)_{\wedge} \rightarrow H^{d+2}(K,\tilde{T})^D$ est donc bien exacte.
\end{itemize}
\end{proof}

\begin{theorem} \textbf{(Suite exacte de Poitou-Tate pour le dual d'un tore)} \label{PT dual tore exacte}\\
Rappelons que nous avons supposé (H \ref{40}) et (H \ref{40'}), c'est-à-dire que $X$ est une courbe, $T = \check{T} \otimes^{\mathbf{L}} \mathbb{Z}(1)[1]$ est quasi-isomorphe à un tore et $\tilde{T} = \hat{T} \otimes^{\mathbf{L}} \mathbb{Z}(d)$. Soit $L$ une extension finie déployant $T$. Supposons que $\Sha^2(L,\mathbb{G}_m)$ est nul. On dispose d'une suite exacte à 8 termes:\\
\centerline{\xymatrix{
& \mathbb{P}^{d+1}(\tilde{T}) \ar[r] & H^1(K,T)^D \ar[d] &\\
(H^0(K,T)_{\wedge})^D\ar[d] & \mathbb{P}^{d+2}(\tilde{T})_{tors} \ar[l] & H^{d+2}(K,\tilde{T})\ar[l] &\\
H^{d+3}(K,\tilde{T}) \ar[r] & \mathbb{P}^{d+3}(\tilde{T}) \ar[r] &  (\varprojlim_n {_n}T(K^s))^D \ar[r] & 0.
}}
\end{theorem}

\begin{proof}
\begin{itemize}
\item[$\bullet$] L'exactitude des six derniers termes découle de \ref{bouts tore}.
\item[$\bullet$] Remarquons que la suite de groupes d'exposant fini:
$$0 \rightarrow \Sha^{1}(T) \rightarrow H^{1}(K,T) \rightarrow \mathbb{P}^{1}(T).$$
est exacte. En utilisant la nullité de $\Sha^2(L,\mathbb{G}_m)$, le lemme \ref{nul dual} et le théorème \ref{PT tore}, on remarque que sa suite duale s'écrit: 
$$\mathbb{P}^{d+1}(\tilde{T}) \rightarrow H^{1}(K,T)^D \rightarrow \Sha^{d+2}(\tilde{T}) \rightarrow 0.$$ 
Pour vérifier son exactitude, il suffit de montrer que le morphisme $H^{1}(K,T) \rightarrow \mathbb{P}^{1}(T)$ est strict, c'est-à-dire que son image est discrète. Pour ce faire, on fixe $U$ un ouvert de $U_0$ et on remarque grâce au lemme de Harder (corollaire A.8 de \cite{GP}) que tout élément de $I \cap (\prod_{v \in X \setminus U} H^{1}(K_v,T) \times\prod_{v \in U} H^{1}(\mathcal{O}_v, \mathcal{T})) \subseteq \mathbb{P}^{1}(T)$ provient de $H^{1}(U,\mathcal{T})$. Comme $\mathbb{P}^{1}(T)$ est d'exposant fini $n$ et $H^{1}(U,\mathcal{T})/n$ est fini, on déduit que $I \cap (\prod_{v \in X \setminus U} H^{1}(K_v,T) \times\prod_{v \in U} H^{1}(\mathcal{O}_v, \mathcal{T}))$ est fini. Cela entraîne que $I$ est discret et donc que la suite $\mathbb{P}^{d+1}(\tilde{T}) \rightarrow H^{1}(K,T)^D \rightarrow \Sha^{d+2}(\tilde{T}) \rightarrow 0$ 
est exacte. 
\end{itemize}
\end{proof}

\begin{remarque}
En combinant les deux théorèmes précédents, on retrouve, dans le cas où $k$ est $p$-adique, la suite à 9 termes du théorème 2.9 de \cite{HS2} à condition d'utiliser la remarque \ref{scheiderer}.
\end{remarque}

\begin{remarque}
Supposons $k_1$ est de caractéristique $p>0$. Dans ce cas:
\begin{itemize}
\item[$\bullet$] le théorème \ref{AV tore} et la remarque \ref{AV tore 2} restent valables pour $l \neq p$.
\item[$\bullet$] l'assertion (i) de \ref{local tore} donne un accouplement parfait de groupes finis:
$$H^a(K_v,T)_{\text{non}-p} \times H^{d+2-a}(K_v,\tilde{T})_{\text{non}-p} \rightarrow \mathbb{Q}/\mathbb{Z},$$
l'assertion (ii) reste valable pour $n$ non divisible par $p$, et l'assection (iii) donne un accouplement parfait:
$$\varprojlim_{p \nmid n} H^{a-1}(K_v,T)/n \times H^{d+3-a}(K_v,\tilde{T})_{\text{non}-p} \rightarrow \mathbb{Q}/\mathbb{Z}.$$
\item[$\bullet$] concernant la proposition \ref{nature tore}, le groupe $\Sha^a(T)$ est d'exposant fini et les groupes $H^r(U,\tilde{\mathcal{T}}_t)_{\text{non}-p}$, $H^r(U,\tilde{\mathcal{T}}_t)_{\text{non}-p}$, $H^{d+3-a}(U,\tilde{\mathcal{T}})_{\text{non}-p}$ et $H^{d+4-a}(U,\tilde{\mathcal{T}})_{\text{non}-p}$ sont de torsion de type cofini. 
\item[$\bullet$] la proposition \ref{Sha tore} reste vraie pour $l \neq p$.
\item[$\bullet$] concernant \ref{suite tore}, on a une suite exacte:
$$ \bigoplus_{v \in X^{(1)}} H^{d+1-a}(K_v,\tilde{T}_t)_{\text{non}-p} \rightarrow H^{d+2-a}_c(U,\tilde{\mathcal{T}}_t)_{\text{non}-p} \rightarrow \mathcal{D}^{d+2-a}(U,\tilde{\mathcal{T}}_t)_{\text{non}-p} \rightarrow 0.$$
\item[$\bullet$] le théorème \ref{PT tore} fournit un accouplement parfait de groupes finis:
$$\Sha^a(T)_{\text{non}-p} \times \overline{\Sha^{d+3-a}(\tilde{T})_{\text{non}-p}} \rightarrow \mathbb{Q}/\mathbb{Z}.$$
\end{itemize}
\end{remarque}

\section{\scshape Groupes de type multiplicatif}

Dans toute cette partie, on suppose (H \ref{40}), c'est-à-dire que $X$ est une courbe. Soit $a\in \{0,1,..., d+1\}$ fixé. Soient $\hat{T_1}$ et $\hat{T_2}$ deux $\text{Gal}(K^s/K)$-modules qui, comme groupes abéliens, sont libres de type fini. Notons $\check{T_1} = \text{Hom}(\hat{T_1},\mathbb{Z})$ et $\check{T_2} = \text{Hom}(\hat{T_2},\mathbb{Z})$. Posons aussi $T_1=\check{T_1} \otimes^{\mathbf{L}} \mathbb{Z}(a)[1] $ et $T_2=\check{T_2} \otimes^{\mathbf{L}} \mathbb{Z}(a)[1] $. Comme dans les sections précédentes, introduisons $\tilde{T_1} = \hat{T_1} \otimes^{\mathbf{L}} \mathbb{Z}(d+1-a)$, $\tilde{T_2} = \hat{T_2} \otimes^{\mathbf{L}} \mathbb{Z}(d+1-a)$, $\tilde{T_1}_t = \hat{T_1} \otimes^{\mathbf{L}} \mathbb{Q}/\mathbb{Z}(d+1-a)$ et $\tilde{T_2}_t = \hat{T_2} \otimes^{\mathbf{L}} \mathbb{Q}/\mathbb{Z}(d+1-a)$, ainsi que $T_{1t}=\check{T_1} \otimes^{\mathbf{L}} \mathbb{Q}/\mathbb{Z}(a)[1]$ et $T_{2t}=\check{T_2} \otimes^{\mathbf{L}} \mathbb{Q}/\mathbb{Z}(a)[1]$. \\
Considérons maintenant un morphisme $\hat{\rho}: \hat{T_2} \rightarrow \hat{T_1}$. Un tel morphisme induit des morphismes $\rho: T_1 \rightarrow T_2$, $\check{\rho}: \check{T_1} \rightarrow \check{T_2}$, $\tilde{\rho}: \tilde{T_2} \rightarrow \tilde{T_1}$, $\tilde{\rho}_t: \tilde{T_2}_t\rightarrow \tilde{T_1}_t$ et $\rho_t: T_{1t} \rightarrow T_{2t}$. Soit $G = [T_1 \rightarrow T_2]$ le cône de $\rho$, de sorte que l'on a un triangle distingué $T_1 \rightarrow T_2 \rightarrow G \rightarrow T_1[1]$. Soit $\tilde{G} = [\tilde{T_2} \rightarrow \tilde{T_1}]$ le cône du morphisme $\tilde{\rho}$, de sorte que l'on a alors un triangle distingué $\tilde{T_2} \rightarrow \tilde{T_1} \rightarrow \tilde{G} \rightarrow \tilde{T_2}[1]$. On notera aussi $G_t = [T_{1t} \rightarrow T_{2t}]$ le cône de $\rho_t$ et $\tilde{G}_t = [\tilde{T_2}_t \rightarrow \tilde{T_1}_t]$ le cône de $\tilde{\rho}_t$, et on remarquera que l'on a des triangles distingués:
$$G \rightarrow G \otimes \mathbb{Q} \rightarrow G_t \rightarrow G[1],$$
$$\tilde{G} \rightarrow \tilde{G} \otimes \mathbb{Q} \rightarrow \tilde{G}_t \rightarrow \tilde{G}[1].$$
Donnons-nous maintenant $\hat{\mathcal{T}_1}$ et $\hat{\mathcal{T}_2}$ des faisceaux définis sur un ouvert $U_0$ de $X$, localement isomorphes à des faisceaux constants libres de type fini et étendant respectivement $\hat{T_1}$ et $\hat{T_2}$, ainsi qu'un morphisme $\hat{\mathcal{T}_2} \rightarrow \hat{\mathcal{T}_1}$ étendant $\hat{\rho}$. Notons $\check{\mathcal{T}_1} = \text{Hom}(\hat{\mathcal{T}_1},\mathbb{Z})$ et $\check{\mathcal{T}_2} = \text{Hom}(\hat{\mathcal{T}_2},\mathbb{Z})$, ainsi que $\mathcal{T}_1=\check{\mathcal{T}_1} \otimes^{\mathbf{L}} \mathbb{Z}(a)[1] $ et $\mathcal{T}_2=\check{\mathcal{T}_2} \otimes^{\mathbf{L}} \mathbb{Z}(a)[1] $, et introduisons $\tilde{\mathcal{T}_1} = \hat{\mathcal{T}_1} \otimes^{\mathbf{L}} \mathbb{Z}(d+1-a)$, $\tilde{\mathcal{T}_2} = \hat{\mathcal{T}_2} \otimes^{\mathbf{L}} \mathbb{Z}(d+1-a)$, $\tilde{\mathcal{T}_1} = \hat{\mathcal{T}_1} \otimes^{\mathbf{L}} \mathbb{Q}/\mathbb{Z}(d+1-a)$ et $\tilde{\mathcal{T}_2} = \hat{\mathcal{T}_2} \otimes^{\mathbf{L}} \mathbb{Q}/\mathbb{Z}(d+1-a)$. Le morphisme $\hat{\mathcal{T}_2} \rightarrow \hat{\mathcal{T}_1}$ induit bien sûr des morphismes $\mathcal{T}_1 \rightarrow \mathcal{T}_2$, $\check{\mathcal{T}_1} \rightarrow \check{\mathcal{T}_2}$, $\tilde{\mathcal{T}_2} \rightarrow \tilde{\mathcal{T}_1}$ et $\tilde{\mathcal{T}_2}_t\rightarrow \tilde{\mathcal{T}_1}_t$. Soient $\mathcal{G} = [\mathcal{T}_1 \rightarrow \mathcal{T}_2]$ et $\tilde{\mathcal{G}} = [\tilde{\mathcal{T}_2} \rightarrow \tilde{\mathcal{T}_1}]$, de sorte que l'on a des triangles distingués $\mathcal{T}_1 \rightarrow \mathcal{T}_2 \rightarrow \mathcal{G} \rightarrow \mathcal{T}_1[1]$ et $\tilde{\mathcal{T}_2} \rightarrow \tilde{\mathcal{T}_1} \rightarrow \tilde{\mathcal{G}} \rightarrow \tilde{\mathcal{T}_2}[1]$. On notera aussi $\mathcal{G}_t = [\mathcal{T}_{1t} \rightarrow \mathcal{T}_{2t}]$ et $\tilde{\mathcal{G}}_t = [\tilde{\mathcal{T}_2}_t \rightarrow \tilde{\mathcal{T}_1}_t]$, et on remarquera que l'on a un triangle distingué:
$$\mathcal{G} \rightarrow \mathcal{G} \otimes \mathbb{Q} \rightarrow \mathcal{G}_t \rightarrow \mathcal{G}[1],$$
$$\tilde{\mathcal{G}} \rightarrow \tilde{\mathcal{G}} \otimes \mathbb{Q} \rightarrow \tilde{\mathcal{G}}_t \rightarrow \tilde{\mathcal{G}}[1].$$

\begin{remarque}
Soit $G_{tm}$ un groupe de type multiplicatif sur $K$. Il existe deux tores $T_1$ et $T_2$ s'insérant dans une suite exacte:
$$1 \rightarrow G_{tm} \rightarrow T_1 \rightarrow T_2 \rightarrow 0.$$
On en déduit que $G_{tm}$ est quasi-isomorphe au complexe $[T_1 \rightarrow T_2][-1]$. 
\end{remarque}

\begin{remarque}
Lorsque $a=d=1$, on remarque que $\tilde{G}[1]$ est le complexe $ [T_2' \rightarrow T_1']$, où $T_1'$ et $T_2'$ désignent les tores duaux de $T_1$ et $T_2$ respectivement.
\end{remarque}

Dans la suite, nous cherchons à établir un théorème de dualité type Poitou-Tate pour $G$. Comme dans la section \ref{tore}, le dual de $G$ qui apparaîtra dans le théorème sera le complexe $\tilde{G}$. Cependant, pour le prouver, il sera utile de faire appel au complexe $\tilde{G}_t$ qui peut être approché par des modules finis auxquels on peut appliquer les théorèmes de la section qui leur était consacrée.

\begin{lemma}\textbf{(Définition de l'accouplement)}\label{acc gm}\\
On rappelle que l'on a supposé (H \ref{40}). On dispose d'accouplements naturels:
$$G \otimes^{\mathbf{L}} \tilde{G} \rightarrow \mathbb{Z}(d+1)[2],$$
$$\mathcal{G} \otimes^{\mathbf{L}} \tilde{\mathcal{G}} \rightarrow \mathbb{Z}(d+1)[2].$$
\end{lemma}

\begin{proof}
Nous allons construire le premier accouplement, le deuxième étant tout à fait analogue.\\
Remarquons que $[\check{T_1} \rightarrow \check{T_2}] \otimes^{\mathbf{L}} [\hat{T_2} \rightarrow \hat{T_1}]$ s'identifie au complexe:\\
\centerline{\xymatrixcolsep{5pc}\xymatrix{[\check{T_1} \otimes \hat{T_2} \ar[r]^-{(-Id \otimes \hat{\rho}) \oplus (\check{\rho} \otimes Id)} & (\check{T_1} \otimes \hat{T_1}) \oplus (\check{T_2} \otimes \hat{T_2}) \ar[r]^-{\check{\rho} \otimes Id + Id \otimes \hat{\rho}} & \check{T_2} \otimes \hat{T_1}].}}
On vérifie aisément que le diagramme suivant commute:\\
\centerline{\xymatrix{
\check{T_1} \otimes \hat{T_2} \ar[r]^{Id \otimes \hat{\rho}} \ar[d]^{\check{\rho} \otimes Id} & \check{T_1} \otimes \hat{T_1}\ar[d]\\
\check{T_2} \otimes \hat{T_2} \ar[r] & \mathbb{Z}
}}
ce qui permet de définir un morphisme $[\check{T_1} \rightarrow \check{T_2}] \otimes^{\mathbf{L}} [\hat{T_2} \rightarrow \hat{T_1}] \rightarrow \mathbb{Z}[1]$. On a donc des morphismes:
\begin{align*}
G \otimes^{\mathbf{L}} \tilde{G}& = ([\check{T_1} \rightarrow \check{T_2}] \otimes^{\mathbf{L}} \mathbb{Z}(a)[1]) \otimes^{\mathbf{L}} ([\hat{T_2} \rightarrow \hat{T_1}] \otimes^{\mathbf{L}} \mathbb{Z}(d+1-a)) \\ & \rightarrow \mathbb{Z}[1]  \otimes^{\mathbf{L}} \mathbb{Z}(a)[1] \otimes^{\mathbf{L}} \mathbb{Z}(d+1-a) \rightarrow \mathbb{Z}(d+1)[2].
\end{align*}
\end{proof}

On en déduit en particulier des accouplements:
$$H^r(U,\mathcal{G}) \times H^{d+2-r}_c(U,\tilde{\mathcal{G}}) \rightarrow \mathbb{Q}/\mathbb{Z}.$$
$$H^r(K_v,G) \times H^{d+1-r}(K_v,\tilde{G}) \rightarrow \mathbb{Q}/\mathbb{Z},$$
pour chaque ouvert non vide $U$ de $U_0$.

\begin{proposition} \textbf{(Artin-Verdier dans le cadre fini)} \label{AV gm fini}\\
On rappelle que l'on a supposé (H \ref{40}). Pour $r \in \mathbb{Z}$, on a un accouplement parfait de groupes finis:
$$H^r(U,\mathcal{G} \otimes^{\mathbf{L}} \mathbb{Z}/n\mathbb{Z}) \times H^{d+1-r}_c(U,\tilde{\mathcal{G}} \otimes^{\mathbf{L}} \mathbb{Z}/n\mathbb{Z}) \rightarrow \mathbb{Q}/\mathbb{Z}.$$ 
\end{proposition}

\begin{proof}
En remarquant que $\mathbb{Z}/n\mathbb{Z} \otimes^{\mathbf{L}} \mathbb{Z}/n\mathbb{Z}$ s'identifie au complexe $[\mathbb{Z} \rightarrow \mathbb{Z} \oplus \mathbb{Z} \rightarrow \mathbb{Z}]$ où la première flèche est donnée par $x \mapsto (nx,-nx)$ et la deuxième par $(x,y) \mapsto n(x+y)$, on peut définir un accouplement $\mathbb{Z}/n\mathbb{Z} \otimes^{\mathbf{L}} \mathbb{Z}/n\mathbb{Z} \rightarrow \mathbb{Z}[1]$ en envoyant $(x,y) \in \mathbb{Z} \oplus \mathbb{Z}$ sur $x+y \in \mathbb{Z}$. Cela permet de définir un accouplement:
$$(\mathcal{G} \otimes^{\mathbf{L}} \mathbb{Z}/n\mathbb{Z}) \otimes^{\mathbf{L}} (\tilde{\mathcal{G}} \otimes^{\mathbf{L}} \mathbb{Z}/n\mathbb{Z}) \rightarrow \mathbb{Z}(d+1)[3] ,$$  d'où un accouplement $$H^r(U,\mathcal{G} \otimes \mathbb{Z}/n\mathbb{Z}) \times H^{d+1-r}_c(U,\tilde{\mathcal{G}} \otimes \mathbb{Z}/n\mathbb{Z}) \rightarrow \mathbb{Q}/\mathbb{Z}.$$
En exploitant les triangles distingués $\mathcal{T}_1 \otimes^{\mathbf{L}} \mathbb{Z}/n\mathbb{Z}  \rightarrow \mathcal{T}_2 \otimes^{\mathbf{L}} \mathbb{Z}/n\mathbb{Z}  \rightarrow \mathcal{G} \otimes^{\mathbf{L}} \mathbb{Z}/n\mathbb{Z} \rightarrow \mathcal{T}_1 \otimes^{\mathbf{L}} \mathbb{Z}/n\mathbb{Z}[1]$ et
$\tilde{\mathcal{T}_2} \otimes^{\mathbf{L}} \mathbb{Z}/n\mathbb{Z} \rightarrow \tilde{\mathcal{T}_1} \otimes^{\mathbf{L}} \mathbb{Z}/n\mathbb{Z} \rightarrow \tilde{\mathcal{G}} \otimes^{\mathbf{L}} \mathbb{Z}/n\mathbb{Z} \rightarrow \tilde{\mathcal{T}_2} \otimes^{\mathbf{L}} \mathbb{Z}/n\mathbb{Z}[1]$ et en faisant appel à la proposition \ref{AV fini}, on dispose d'un diagramme commutatif à lignes exactes de groupes finis:\\
\footnotesize{\xymatrix{
H^{r}(U,\mathcal{T}_1 \otimes^{\mathbf{L}} \mathbb{Z}/n\mathbb{Z}) \ar[r] \ar[d]^{\cong} & H^{r}(U,\mathcal{T}_2 \otimes^{\mathbf{L}} \mathbb{Z}/n\mathbb{Z}) \ar[r] \ar[d]^{\cong} &H^r(U,\mathcal{G}\otimes^{\mathbf{L}} \mathbb{Z}/n\mathbb{Z}) \ar[r] \ar[d] & ...\\
H^{d+2-r}_c(U,\tilde{\mathcal{T}_1}\otimes^{\mathbf{L}} \mathbb{Z}/n\mathbb{Z})^D \ar[r] & H^{d+2-r}_c(U,\tilde{\mathcal{T}_2}\otimes^{\mathbf{L}} \mathbb{Z}/n\mathbb{Z})^D \ar[r] & H^{d+1-r}_c(U,\tilde{\mathcal{G}}\otimes^{\mathbf{L}} \mathbb{Z}/n\mathbb{Z})^D \ar[r] & ...
}\\

\xymatrix{
\hspace{110pt}...\ar[r] &H^{r+1}(U,\mathcal{T}_1 \otimes^{\mathbf{L}} \mathbb{Z}/n\mathbb{Z}) \ar[r] \ar[d]^{\cong} &H^{r+1}(U,\mathcal{T}_1 \otimes^{\mathbf{L}} \mathbb{Z}/n\mathbb{Z}) \ar[d]^{\cong}\\
\hspace{110pt}...\ar[r] & H^{d+1-r}_c(U,\tilde{\mathcal{T}_1}\otimes^{\mathbf{L}} \mathbb{Z}/n\mathbb{Z})^D \ar[r] & H^{d+1-r}_c(U,\tilde{\mathcal{T}_2}\otimes^{\mathbf{L}} \mathbb{Z}/n\mathbb{Z})^D. 
}
}
\normalsize
On en déduit que la flèche verticale centrale est un isomorphisme.
\end{proof}

\begin{theorem} \textbf{(Artin-Verdier pour les groupes de type multiplicatif)} \label{AV gm}\\
On rappelle que l'on a supposé (H \ref{40}). Soit $r \in \mathbb{Z}$. Pour chaque nombre premier $l$, on a un accouplement parfait de groupes finis:
$$H^r(U,\mathcal{G})\{l\}^{(l)} \times H^{d+2-r}_c(U,\tilde{\mathcal{G}})^{(l)}\{l\} \rightarrow \mathbb{Q}/\mathbb{Z}.$$
\end{theorem}

\begin{proof}
Pour chaque entier naturel $n$, on dispose des triangles distingués:
$$\mathcal{G} \rightarrow \mathcal{G} \rightarrow \mathcal{G} \otimes^{\mathbf{L}} \mathbb{Z}/l^n\mathbb{Z} \rightarrow \mathcal{G}[1],$$
$$\tilde{\mathcal{G}} \rightarrow \tilde{\mathcal{G}} \rightarrow \tilde{\mathcal{G}} \otimes^{\mathbf{L}} \mathbb{Z}/l^n\mathbb{Z} \rightarrow \tilde{\mathcal{G}}[1].$$
On en déduit des suites exactes de groupes finis:
$$0 \rightarrow H^{r-1}(U,\mathcal{G})/l^n \rightarrow H^{r-1}(U,\mathcal{G} \otimes^{\mathbf{L}} \mathbb{Z}/l^n\mathbb{Z}) \rightarrow {_{l^n}}H^r(U,\mathcal{G}) \rightarrow 0,$$
$$0 \rightarrow H^{d+2-r}_c(U,\tilde{\mathcal{G}})/l^n \rightarrow H^{d+2-r}_c(U,\tilde{\mathcal{G}} \otimes^{\mathbf{L}} \mathbb{Z}/l^n\mathbb{Z}) \rightarrow {_{l^n}}H^{d+3-r}_c(U,\tilde{\mathcal{G}})\rightarrow 0.$$
En passant à la limite inductive dans la première suite et à la limite projective dans la deuxième, on obtient des suites exactes:
$$0 \rightarrow H^{r-1}(U,\mathcal{G}) \otimes \mathbb{Q}_l/\mathbb{Z}_l \rightarrow \varinjlim_n H^{r-1}(U,\mathcal{G} \otimes^{\mathbf{L}} \mathbb{Z}/l^n\mathbb{Z}) \rightarrow H^r(U,\mathcal{G})\{l\}\rightarrow 0,$$
$$0 \rightarrow H^{d+2-r}_c(U,\tilde{\mathcal{G}})^{(l)} \rightarrow \varprojlim_n H^{d+2-r}_c(U,\tilde{\mathcal{G}} \otimes^{\mathbf{L}} \mathbb{Z}/l^n\mathbb{Z}) \rightarrow \varprojlim_n {_{l^n}}H^{d+3-r}_c(U,\tilde{\mathcal{G}})\rightarrow 0.$$
En remarquant que $H^{r-1}(U,\mathcal{G}) \otimes \mathbb{Q}_l/\mathbb{Z}_l$ est divisible et que $\varprojlim_n {_{l^n}}H^{d+3-r}_c(U,\tilde{\mathcal{G}})$ est sans torsion et en utilisant la proposition \ref{AV gm fini}, on obtient alors des isomorphismes:
\begin{align*}
 H^r(U,\mathcal{G})\{l\}^{(l)} & \cong (\varinjlim_n H^{r-1}(U,\mathcal{G} \otimes^{\mathbf{L}} \mathbb{Z}/l^n\mathbb{Z}))^{(l)}
 \\& \cong (\varinjlim_n H^{d+2-r}_c(U,\tilde{\mathcal{G}} \otimes^{\mathbf{L}} \mathbb{Z}/l^n\mathbb{Z})^D)^{(l)} \\& \cong (\varprojlim_n H^{d+2-r}_c(U,\tilde{\mathcal{G}} \otimes^{\mathbf{L}} \mathbb{Z}/l^n\mathbb{Z}))\{l\}^D \\& \cong H^{d+2-r}_c(U,\tilde{\mathcal{G}})^{(l)}\{l\}^D
\end{align*}
\end{proof}

\begin{proposition}\label{local gm}
Rappelons que nous avons supposé (H \ref{40}).
\begin{itemize}
\item[$(i)$] Supposons que le morphisme $\hat{T_2} \rightarrow \hat{T_1}$ est injectif. Alors les groupes $H^{a-1}(K_v,G)$ et $H^{a-1}(K_v^h,G)$ sont d'exposant fini.
\item[$(ii)$] Supposons que le morphisme $\check{T_1} \rightarrow \check{T_2}$ est injectif et que $K_a^M(L)_{tors}$ est d'exposant fini pour toute extension finie $L$ de $K_v$. Alors le groupe $H^{a-1}(K_v,G)_{tors}$ est d'exposant fini.
\item[$(iii)$] Supposons que $K_a^M(L)_{tors}$ est d'exposant fini pour toute extension finie $L$ de $K_v$. Alors le groupe $H^{a-1}(K_v,G)_{tors}$ est d'exposant fini.
\end{itemize}
\end{proposition} 

\begin{remarque}\label{Milnor}
L'hypothèse sur la $K$-théorie de Milnor est toujours vérifiée pour $a=0$, et aussi pour $a=1$ si $k_1$ est un corps $p$-adique. Je ne sais pas si elle est vérifiée dans d'autres cas.
\end{remarque}

\begin{proof}
\begin{itemize}
\item[(i)] L'hypothèse implique que le conoyau de $\check{T_1} \rightarrow \check{T_2}$ est fini. À l'aide du triangle distingué $\text{Ker}(T_1 \rightarrow T_2)[1] \rightarrow G \rightarrow \text{Coker}(T_1 \rightarrow T_2) \rightarrow \text{Ker}(T_1 \rightarrow T_2)[2]$, on obtient une suite exacte:
$$H^{a}(K_v,\text{Ker}(T_1 \rightarrow T_2)) \rightarrow H^{a-1}(K_v,G) \rightarrow H^{a-1}(K_v,\text{Coker}(T_1 \rightarrow T_2)).$$
Comme nous l'avons déjà vu plusieurs fois, le groupe $H^{a}(K_v,\text{Ker}(T_1 \rightarrow T_2)) \cong H^a(K_v,\text{Ker}(\check{T_1} \rightarrow \check{T_2}) \otimes^{\mathbf{L}} \mathbb{Z}(a)[1])$ est fini. De plus, un argument de restriction-corestriction prouve que le groupe $H^{a-1}(K_v,\text{Coker}(T_1 \rightarrow T_2))$ est d'exposant fini. On en déduit que $H^{a-1}(K_v,G)$ est d'exposant fini. On prouve exactement de la même manière que $H^{a-1}(K_v^h,G)$ est d'exposant fini.
\item[(ii)] Notons $Q$ le conoyau de $\check{T_1} \rightarrow \check{T_2}$, de sorte que l'on a un quasi-isomorphisme $G \cong Q \otimes^{\mathbf{L}} \mathbb{Z}(a)[1]$. On peut alors trouver un $\text{Gal}(K^s/K)$-module fini $F$ et un $\text{Gal}(K^s/K)$-module sans torsion de type fini $M$ s'insérant dans une suite exacte de $\text{Gal}(K^s/K)$-modules:
$$0 \rightarrow F \rightarrow Q \rightarrow M \rightarrow 0.$$
On en déduit une suite exacte:
$$H^a(K_v,F \otimes^{\mathbf{L}} \mathbb{Z}(a)) \rightarrow H^{a-1}(K_v,G) \rightarrow H^{a}(K_v,M \otimes^{\mathbf{L}} \mathbb{Z}(a)).$$
Le groupe $H^a(K_v,F \otimes^{\mathbf{L}} \mathbb{Z}(a))$ est bien sûr d'exposant fini. De plus, si $L$ est une extension finie de $K_v$ de degré $e$ telle que $\text{Gal}(K^s/L)$ agit trivialement sur $M$, on a, pour un certain entier naturel $n$, un diagramme commutatif:\\
\centerline{\xymatrix{
H^{a}(K_v,M \otimes \mathbb{Z}(a)) \ar[r]^{\text{Res}} \ar[rd]^{\cdot e} & H^a(L,\mathbb{Z}(a))^n \ar[d]^{\text{Cor}}\\
& H^{a}(K_v,M \otimes \mathbb{Z}(a))
}}
L'isomorphisme de Suslin-Nesterenko-Totaro s'écrit $H^a(L,\mathbb{Z}(a)) \cong K^M_a(L)$. Par conséquent, l'hypothèse impose que $H^a(L,\mathbb{Z}(a))_{tors}$ est d'exposant fini, et on en déduit immédiatement que $H^{a}(K_v,M \otimes \mathbb{Z}(a))_{tors}$ et $H^{a-1}(K_v,G)_{tors}$ le sont aussi.
\item[(iii)] En notant $\text{Im}(T_1 \rightarrow T_2) = \text{Im}(\check{T_1} \rightarrow \check{T_2}) \otimes \mathbb{Z}(a)$, on dispose d'un triangle distingué:
$$[T_1 \rightarrow \text{Im}(T_1 \rightarrow T_2)] \rightarrow G \rightarrow [\text{Im}(T_1 \rightarrow T_2) \rightarrow T_2] \rightarrow [T_1 \rightarrow \text{Im}(T_1 \rightarrow T_2)][1],$$
d'où une suite exacte:
$$H^{a-1}(K_v,[T_1 \rightarrow \text{Im}(T_1 \rightarrow T_2)]) \rightarrow H^{a-1}(K_v,G) \rightarrow H^{a-1}(K_v,[\text{Im}(T_1 \rightarrow T_2) \rightarrow T_2]).$$
D'après (i), le groupe $H^{a-1}(K_v,[T_1 \rightarrow \text{Im}(T_1 \rightarrow T_2)])$ est d'exposant fini, et d'après (ii), $H^{a-1}(K_v,[\text{Im}(T_1 \rightarrow T_2) \rightarrow T_2])_{tors}$ l'est aussi. On en déduit que $H^{a-1}(K_v,G)_{tors}$ est d'exposant fini.

\end{itemize}
\end{proof}

\begin{proposition}\textbf{(Dualité locale pour les groupes de type multiplicatif)}\label{local gm 2}
Rappelons que nous avons supposé (H \ref{40}).
\begin{itemize}
\item[(i)] Supposons que le morphisme $\hat{T_2} \rightarrow \hat{T_1}$ est injectif.
On a alors un accouplement parfait:
$$H^{a-1}(K_v,G)^{\wedge} \times H^{d+2-a}(K_v,\tilde{G}) \rightarrow \mathbb{Q}/\mathbb{Z}.$$
\item[(ii)] Le morphisme naturel $H^{a-1}(K_v,G)^{\wedge} \rightarrow (H^{d+2-a}(K_v,\tilde{G}))^D$ est injectif.
\end{itemize}
\end{proposition}

\begin{proof}
\begin{itemize}
\item[(i)] Les triangles distingués $T_1 \rightarrow T_2 \rightarrow G \rightarrow T_1[1]$ et $\tilde{T_2} \rightarrow \tilde{T_1} \rightarrow \tilde{G} \rightarrow \tilde{T_2}[1]$ fournissent des suites exactes:
\begin{gather}
\begin{split}
H^{a-1}(K_v,T_1) \rightarrow H^{a-1}(K_v,T_2) \rightarrow H^{a-1}(K_v,G) \rightarrow\\ \rightarrow H^a(K_v,T_1) \rightarrow H^a(K_v,T_2),
\end{split}\\
\begin{split}
H^{d+2-a}(K_v,\tilde{T_2}) \rightarrow H^{d+2-a}(K_v,\tilde{T_1}) \rightarrow H^{d+2-a}(K_v,\tilde{G}) \rightarrow \\ \rightarrow H^{d+3-a}(K_v,\tilde{T_2}) \rightarrow H^{d+3-a}(K_v,\tilde{T_1}).
\end{split}
\end{gather}
Exploitons ces deux suites:
\begin{itemize}
\item[$\bullet$] Comme $H^a(K_v,T_1)$ et $H^a(K_v,T_2)$ sont finis et $H^{a-1}(K_v,G)$ est d'exposant fini d'après \ref{local gm}, la suite (9) impose que la suite $$ H^{a-1}(K_v,T_2)^{\wedge} \rightarrow H^{a-1}(K_v,G)^{\wedge} \rightarrow H^a(K_v,T_1) \rightarrow H^a(K_v,T_2)$$
est exacte. De plus, $H^{a-1}(K_v,T_1)^{\wedge} \rightarrow H^{a-1}(K_v,T_2)^{\wedge} \rightarrow H^{a-1}(K_v,G)^{\wedge}$ est un complexe.
\item[$\bullet$] Les groupes de la suite (10) étant discrets de torsion, on obtient une suite exacte 
\begin{align*}
H^{d+3-a}(K_v,\tilde{T_1})^D \rightarrow H^{d+3-a}(K_v,\tilde{T_2})^D &\rightarrow H^{d+2-a}(K_v,\tilde{G})^D \\
&\rightarrow H^{d+2-a}(K_v,\tilde{T_1})^D  \rightarrow  H^{d+2-a}(K_v,\tilde{T_2})^D.
\end{align*}
\end{itemize}
On obtient alors un diagramme commutatif:
\begin{multline*}
\xymatrix{
H^{a-1}(K_v,T_1)^{\wedge}\ar[r]\ar[d]^{\cong} & H^{a-1}(K_v,T_2)^{\wedge} \ar[r]\ar[d]^{\cong} &  H^{a-1}(K_v,G)^{\wedge} \ar[r]\ar[d] & ...\\
H^{d+3-a}(K_v,\tilde{T_1})^D \ar[r] &  H^{d+3-a}(K_v,\tilde{T_2})^D \ar[r] & H^{d+2-a}(K_v,\tilde{G})^D \ar[r] & ...
}\\
\xymatrix{
... \ar[r] &  H^a(K_v,T_1) \ar[r]\ar[d]^{\cong} &  H^a(K_v,T_2)\ar[d]^{\cong}\\
... \ar[r] & H^{d+2-a}(K_v,\tilde{T_1})^D \ar[r] &  H^{d+2-a}(K_v,\tilde{T_2})^D
}
\end{multline*}
où la première ligne est un complexe dont les quatre derniers termes forment une suite exacte, la deuxième ligne est exacte, et les isomorphismes proviennent de \ref{local tore}. On en déduit que le morphisme vertical central est un isomorphisme. On en déduit que l'accouplement:
$$H^{a-1}(K_v,G)^{\wedge} \times H^{d+2-a}(K_v,\tilde{G}) \rightarrow \mathbb{Q}/\mathbb{Z}$$
est parfait.
\item[(ii)] Soit $n>0$ un entier. On dispose d'un diagramme commutatif à lignes exactes:
\scriptsize
\begin{multline*}
\xymatrix{
H^{a-1}(K_v, T_1 \otimes \mathbb{Z}/n\mathbb{Z}) \ar[r]\ar[d]^{\cong} & H^{a-1}(K_v, T_2 \otimes \mathbb{Z}/n\mathbb{Z}) \ar[r]\ar[d]^{\cong} &H^{a-1}(K_v, G \otimes \mathbb{Z}/n\mathbb{Z}) \ar[r]\ar[d] & ...\\
H^{d+2-a}(K_v,\tilde{T_1} \otimes \mathbb{Z}/n\mathbb{Z})^D \ar[r] &H^{d+2-a}(K_v,\tilde{T_2} \otimes \mathbb{Z}/n\mathbb{Z})^D \ar[r] &H^{d+1-a}(K_v,\tilde{G} \otimes \mathbb{Z}/n\mathbb{Z})^D \ar[r] & ...
}\\
\xymatrix{
... \ar[r] & H^{a}(K_v, T_1 \otimes \mathbb{Z}/n\mathbb{Z}) \ar[r]\ar[d]^{\cong} & H^{a}(K_v, T_2 \otimes \mathbb{Z}/n\mathbb{Z}) \ar[d]^{\cong}\\
... \ar[r] &H^{d+1-a}(K_v,\tilde{T_1} \otimes \mathbb{Z}/n\mathbb{Z})^D \ar[r] &H^{d+1-a}(K_v,\tilde{T_2} \otimes \mathbb{Z}/n\mathbb{Z})^D.
}
\end{multline*}
\normalsize
Le lemme des cinq fournit alors une dualité parfaite de groupes finis:
$$H^{a-1}(K_v, G \otimes \mathbb{Z}/n\mathbb{Z})  \times H^{d+1-a}(K_v,\tilde{G} \otimes \mathbb{Z}/n\mathbb{Z}) \rightarrow \mathbb{Q}/\mathbb{Z}.$$
On a alors un diagramme commutatif à lignes exactes:
\begin{multline*}
\xymatrix{
0 \ar[r] & H^{a-1}(K_v,G)/n \ar[r]\ar[d] &  H^{a-1}(K_v,G \otimes \mathbb{Z}/n\mathbb{Z}) \ar[r]\ar[d]^{\cong} & ...  \\
0 \ar[r] & ({_n}H^{d+2-a}(K_v,\tilde{G}))^D \ar[r] & H^{d+1-a}(K_v,\tilde{G} \otimes \mathbb{Z}/n\mathbb{Z})^D\ar[r] & ...
}\\
\xymatrix{
... \ar[r] &  {_n}H^{a}(K_v,G) \ar[r]\ar[d] & 0\\
... \ar[r] & (H^{d+1-a}(K_v,\tilde{G})/n)^D \ar[r] & 0.
}
\end{multline*}
Par conséquent, la flèche $H^{a-1}(K_v,G)/n \rightarrow ({_n}H^{d+2-a}(K_v,\tilde{G}))^D$ est injective. Il suffit alors de passer à la limite projective sur $n$.
\end{itemize}
\end{proof}

On définit:
\begin{itemize}
\item[$\bullet$] $\mathcal{D}^{d+3-a}(U,\tilde{\mathcal{G}}) = \text{Im} (H^{d+3-a}_c(U,\tilde{\mathcal{G}}) \rightarrow H^{d+3-a}(K,\tilde{G}))$,
\item[$\bullet$] $\mathcal{D}^{d+2-a}(U,\tilde{\mathcal{G}}_t) = \text{Im} (H^{d+2-a}_c(U,\tilde{\mathcal{G}}_t) \rightarrow H^{d+2-a}(K,\tilde{G}_t))$,
\item[$\bullet$] $\mathcal{D}^{d+2-a}(U,\tilde{\mathcal{G}}\otimes \mathbb{Z}/n\mathbb{Z}) = \text{Im} (H^{d+2-a}_c(U,\tilde{\mathcal{G}}\otimes \mathbb{Z}/n\mathbb{Z}) \rightarrow H^{d+2-a}(K,\tilde{G}\otimes \mathbb{Z}/n\mathbb{Z}))$.
\item[$\bullet$] $\mathcal{D}^a(U,\mathcal{G}) = \text{Im} (H^a_c(U,\mathcal{G}) \rightarrow H^a(K,G))$,
\item[$\bullet$] $\mathcal{D}^{a-1}(U,\mathcal{G}_t) = \text{Im} (H^{a-1}_c(U,\mathcal{G}_t) \rightarrow H^{a-1}(K,G_t))$,
\item[$\bullet$] $\mathcal{D}^{a-1}(U,\mathcal{G} \otimes \mathbb{Z}/n\mathbb{Z}) = \text{Im} (H^{a-1}_c(U,\mathcal{G} \otimes \mathbb{Z}/n\mathbb{Z}) \rightarrow H^{a-1}(K,G \otimes \mathbb{Z}/n\mathbb{Z}))$.
\end{itemize}

\begin{lemma}\label{nature gm}
Rappelons que nous avons supposé (H \ref{40}).
\begin{itemize}
\item[(i)] Le groupe $H^{d+3-a}(U,\tilde{\mathcal{G}})$ est de torsion de type cofini.
\item[(ii)] Le groupe $H^{d+2-a}(K_v,\tilde{G})$ est de torsion de type cofini pour $v \in X^{(1)}$.
\item[(iii)] Le groupe $H^{d+2-a}(K_v^h,\tilde{G})$ est de torsion de type cofini pour $v \in X^{(1)}$.
\end{itemize}
Si $\hat{T_2} \rightarrow \hat{T_1}$ est injectif, alors:
\begin{itemize}
\item[(iv)] Le groupe $H^{a}(U,\mathcal{G})$ est de torsion de type cofini.
\item[(v)] Le groupe $H^{a-1}(K_v,G)$ est de torsion de type cofini pour $v \in X^{(1)}$.
\item[(vi)] Le groupe $H^{a-1}(K_v^h,G)$ est de torsion de type cofini pour $v \in X^{(1)}$.
\end{itemize}
\end{lemma}

\begin{proof}
\begin{itemize}
\item[(i)] La proposition \ref{nature tore} montre que $H^{d+3-a}(U,\tilde{\mathcal{T}_1})$ et $H^{d+4-a}(U,\tilde{\mathcal{T}_2})$ sont de torsion de type cofini. La suite exacte $H^{d+3-a}(U,\tilde{\mathcal{T}_1}) \rightarrow H^{d+3-a}(U,\tilde{\mathcal{G}}) \rightarrow H^{d+4-a}(U,\tilde{\mathcal{T}_2})$ permet alors de conclure.
\item[(ii)] Les groupes $H^{d+2-a}(K_v,\hat{T_1} \otimes \mathbb{Q}(d+1-a))$ et $H^{d+3-a}(K,\hat{T_2} \otimes \mathbb{Q}(d+1-a))$ sont nuls. Par conséquent, les suites exactes:
$$H^{d+1-a}(K_v,\tilde{T_{1}}_t) \rightarrow H^{d+2-a}(K_v,\tilde{T_1}) \rightarrow H^{d+2-a}(K_v,\hat{T_1} \otimes \mathbb{Q}(d+1-a)),$$
$$H^{d+2-a}(K_v,\tilde{T_{2}}_t) \rightarrow H^{d+3-a}(K_v,\tilde{T_2}) \rightarrow H^{d+3-a}(K_v,\hat{T_2} \otimes \mathbb{Q}(d+1-a))$$
fournissent des surjections $H^{d+1-a}(K_v,\tilde{T_{1}}_t) \rightarrow H^{d+2-a}(K_v,\tilde{T_1})$ et $H^{d+2-a}(K_v,\tilde{T_{2}}_t) \rightarrow H^{d+3-a}(K_v,\tilde{T_2})$.
On en déduit que $H^{d+2-a}(K_v,\tilde{T_1})$ et $H^{d+3-a}(K_v,\tilde{T_2})$ sont de torsion de type cofini. La suite exacte $H^{d+2-a}(K_v,\tilde{T_1}) \rightarrow H^{d+2-a}(K_v,\tilde{G}) \rightarrow H^{d+3-a}(K_v,\tilde{T_2})$ permet alors de conclure.
\item[(iii)] La preuve est identique à celle de (ii).
\item[(iv)] Le triangle distingué $\text{Ker}(\mathcal{T}_1 \rightarrow \mathcal{T}_2)[1] \rightarrow \mathcal{G} \rightarrow \text{Coker}(\mathcal{T}_1 \rightarrow \mathcal{T}_2) \rightarrow \text{Ker}(\mathcal{T}_1 \rightarrow \mathcal{T}_2)[2]$ fournit une suite exacte:
$$H^{a+1}(U,\text{Ker}(\mathcal{T}_1 \rightarrow \mathcal{T}_2)) \rightarrow H^a(U,\mathcal{G}) \rightarrow H^a(U,\text{Coker}(\mathcal{T}_1 \rightarrow \mathcal{T}_2)).$$
Comme $\hat{\mathcal{T}_2} \rightarrow \hat{\mathcal{T}_1}$ est injectif, le conoyau de $\check{\mathcal{T}_1} \rightarrow \check{\mathcal{T}_2}$ est fini, et donc $H^a(U,\text{Coker}(\mathcal{T}_1 \rightarrow \mathcal{T}_2))$ est d'exposant fini. De plus, comme $H^{a+1}(U,\mathbb{Q}/\mathbb{Z}(a))  \rightarrow H^{a+2}(U,\mathbb{Z}(a))$ est surjectif, un argument de restriction-corestriction montre que $H^{a+1}(U,\text{Ker}(\mathcal{T}_1 \rightarrow \mathcal{T}_2))$ est de torsion. Par conséquent, $H^a(U,\mathcal{G})$ est de torsion. Reste à montrer qu'il est de type cofini. Mais cela découle immédiatement de la suite exacte:
$$0 \rightarrow H^{a-1}(U,\mathcal{G})/n \rightarrow H^{a-1}(U,\mathcal{G} \otimes \mathbb{Z}/n\mathbb{Z}) \rightarrow {_n}H^a(U,\mathcal{G}) \rightarrow 0.$$
\item[(v)] Nous avons déjà montré dans \ref{local gm} que $H^{a-1}(K_v,G)$ est de torsion. Le même argument qu'en (iv) montre qu'il est de type cofini.
\item[(vi)] La preuve est identique à celle de (v).
\end{itemize}
\end{proof}

\begin{corollary} \label{tcof gm}
On rappelle que l'on a supposé (H \ref{40}).
\begin{itemize}
\item[(i)] Les groupes $H^{d+3-a}_c(U,\tilde{\mathcal{G}})$ et $\mathcal{D}^{d+3-a}(U,\tilde{\mathcal{G}})$ sont de torsion de type cofini.
\item[(ii)] Si $\hat{T_2} \rightarrow \hat{T_1}$ est injectif, les groupes $H^{a}_c(U,\mathcal{G})$ et $\mathcal{D}^{a}(U,\mathcal{G})$ sont de torsion de type cofini.
\end{itemize}
\end{corollary}

\begin{proof}
\begin{itemize}
\item[(i)] La suite exacte $\bigoplus_{v \in X \setminus U} H^{d+2-a}(K_v^h,\tilde{G}) \rightarrow H^{d+3-a}_c(U,\tilde{\mathcal{G}}) \rightarrow H^{d+3-a}(U,\tilde{\mathcal{G}})$ prouve que $H^{d+3-a}_c(U,\tilde{\mathcal{G}})$ est de torsion de type cofini. On en déduit que $\mathcal{D}^{d+3-a}(U,\tilde{\mathcal{G}})$ est aussi de torsion de type cofini.
\item[(ii)] La preuve est identique à celle de (i).
\end{itemize}
\end{proof}

\begin{lemma}\label{iso}
On rappelle que l'on a supposé (H \ref{40}).
Soit $L = K$, $K_v$ ou $K_v^h$ pour un certain $v \in X^{(1)}$.
\begin{itemize}
\item[(i)] Le morphisme $H^{d+2-a}(L,\tilde{G}_t) \rightarrow H^{d+3-a}(L,\tilde{G})$ est un isomorphisme. En particulier, $\Sha^{d+2-a}(\tilde{G}_t) \cong \Sha^{d+3-a}(\tilde{G})$.
\item[(ii)] Supposons que $\hat{T_2} \rightarrow \hat{T_1}$ est injectif. Le morphisme $H^{a-1}(L,G_t) \rightarrow H^{a}(L,G)$ est un isomorphisme. En particulier, $\Sha^{a-1}(G_t) \cong \Sha^{a}(G)$.
\end{itemize}
\end{lemma}

\begin{proof}
\begin{itemize}
\item[(i)] Cela découle immédiatement de la nullité de $H^{d+2-a}(L,\tilde{G} \otimes \mathbb{Q})$ et $H^{d+3-a}(L,\tilde{G} \otimes \mathbb{Q})$.
\item[(ii)] Il suffit de vérifier que $H^{a-1}(L,G \otimes \mathbb{Q})$ et $H^a(L,G \otimes \mathbb{Q})$ sont nuls. La nullité de $H^a(L,G \otimes \mathbb{Q})$ est évidente (puisque $H^a(L,T_2 \otimes \mathbb{Q})$ et $H^{a+1}(L,T_1 \otimes \mathbb{Q})$ sont nuls). Pour montrer la nullité de $H^{a-1}(L,G \otimes \mathbb{Q})$, on considère la suite exacte:
$$H^a(L, \text{Ker}(T_1 \otimes \mathbb{Q} \rightarrow T_2 \otimes \mathbb{Q})) \rightarrow H^{a-1}(L,G\otimes \mathbb{Q}) \rightarrow H^{a-1}(L,\text{Coker}(T_1 \otimes \mathbb{Q} \rightarrow T_2 \otimes \mathbb{Q})).$$
Le groupe $H^a(L, \text{Ker}(T_1 \otimes \mathbb{Q} \rightarrow T_2 \otimes \mathbb{Q}))$ est nul car $H^{a+1}(L,\mathbb{Q}(a))$ l'est, et le groupe $H^{a-1}(L,\text{Coker}(T_1 \otimes \mathbb{Q} \rightarrow T_2 \otimes \mathbb{Q}))$ est nul car $\hat{T_2} \rightarrow \hat{T_1}$ est injectif. Cela entraîne la nullité de $H^{a-1}(L,G\otimes \mathbb{Q})$.
\end{itemize}
\end{proof}

\begin{proposition} \label{Sha gm}
On rappelle que l'on a supposé (H \ref{40}).
\begin{itemize}
\item[(i)] Soit $l$ un nombre premier. Il existe un ouvert $U_1$ de $U_0$ tel que, pour tout ouvert non vide $U$ de $U_1$, on a $$\mathcal{D}^{d+3-a}(U,\tilde{\mathcal{G}})\{l\} = \mathcal{D}^{d+3-a}(U_1,\tilde{\mathcal{G}})\{l\} = \Sha^{d+3-a}(K,\tilde{G})\{l\}$$
et $$\mathcal{D}^{d+2-a}(U,\tilde{\mathcal{G}}_t)\{l\} = \mathcal{D}^{d+2-a}(U_1,\tilde{\mathcal{G}}_t)\{l\} = \Sha^{d+2-a}(K,\tilde{G}_t)\{l\}.$$
\item[(ii)] Supposons que $\hat{T_2} \rightarrow \hat{T_1}$ est injectif. Soit $l$ un nombre premier. Il existe un ouvert $U_2$ de $U_0$ tel que, pour tout ouvert non vide $U$ de $U_2$, on a $$\mathcal{D}^{a}(U_2,\mathcal{G})\{l\} = \mathcal{D}^{a}(U_2,\mathcal{G})\{l\} = \Sha^{a}(K,G)\{l\}$$
et $$\mathcal{D}^{a-1}(U,\mathcal{G}_t)\{l\} = \mathcal{D}^{a-1}(U_2,\mathcal{G}_t)\{l\} = \Sha^{a-1}(K,G_t)\{l\}.$$
\end{itemize}
\end{proposition}

\begin{proof}
\begin{itemize}
\item[(i)] On commence par remarquer que, pour $v \in X^{(1)}$, on a un diagramme commutatif:\\
\centerline{\xymatrix{
H^{d+3-a}(K_v^h,\tilde{G}) \ar[r]  & H^{d+3-a}(K_v,\tilde{G})\\
H^{d+2-a}(K_v^h,\tilde{G}_t) \ar[r] \ar[u]& H^{d+2-a}(K_v,\tilde{G}_t)\ar[u]
}}
D'après le lemme \ref{iso}, les flèches verticales sont des isomorphismes. De plus, la flèche $H^{d+2-a}(K_v^h,\tilde{G}_t) \rightarrow H^{d+2-a}(K_v,\tilde{G}_t)$ est aussi un isomorphisme puisqu'elle s'insère dans le diagramme commutatif à lignes exactes:
\begin{multline*}
\xymatrix{
H^{d+2-a}(K_v^h,\tilde{T_2}_t) \ar[r]\ar[d]^{\cong} & H^{d+2-a}(K_v^h,\tilde{T_1}_t) \ar[r]\ar[d]^{\cong} & H^{d+2-a}(K_v^h,\tilde{G}_t) \ar[r]\ar[d] & ... \\
H^{d+2-a}(K_v,\tilde{T_2}_t) \ar[r] & H^{d+2-a}(K_v,\tilde{T_1}_t) \ar[r] & H^{d+2-a}(K_v,\tilde{G}_t) \ar[r] & ...
}\\
\xymatrix{
... \ar[r] & H^{d+3-a}(K_v^h,\tilde{T_2}_t) \ar[r]\ar[d]^{\cong} & H^{d+3-a}(K_v^h,\tilde{T_1}_t) \ar[d]^{\cong} \\
... \ar[r] & H^{d+3-a}(K_v,\tilde{T_2}_t) \ar[r] & H^{d+3-a}(K_v,\tilde{T_1}_t) .
}
\end{multline*}
On en déduit que la restriction $H^{d+3-a}(K_v^h,\tilde{G}) \rightarrow H^{d+3-a}(K_v,\tilde{G})$ est un isomorphisme.\\
Dans la suite de cette preuve, les commutativités des diagrammes que nous utilisons implicitement sont contenues dans la proposition 4.3 de \cite{CTH}. \\
Fixons maintenant un nombre premier $l$. Pour des ouverts $V \subseteq V'$ de $U_0$, on a $\mathcal{D}^{d+2-a}(V,\tilde{\mathcal{G}}_t) \subseteq \mathcal{D}^{d+2-a}(V',\tilde{\mathcal{G}}_t)$ et $\mathcal{D}^{d+3-a}(V,\tilde{\mathcal{G}}) \subseteq \mathcal{D}^{d+3-a}(V',\tilde{\mathcal{G}})$. Or, grâce au corollaire \ref{tcof gm}, on sait que $\mathcal{D}^{d+2-a}(U_0,\tilde{\mathcal{G}}_t)$ et $\mathcal{D}^{d+3-a}(U_0,\tilde{\mathcal{G}})$ sont de torsion de type cofini. Par conséquent, il existe un ouvert $U_1$ de $U_0$ tel que, pour tout ouvert $V \subseteq U_1$, on a $\mathcal{D}^{d+2-a}(V,\tilde{\mathcal{G}}_t)\{l\} = \mathcal{D}^{d+2-a}(U_1,\tilde{\mathcal{G}}_t)\{l\}$ et $\mathcal{D}^{d+3-a}(V,\tilde{\mathcal{G}})\{l\} = \mathcal{D}^{d+3-a}(U_1,\tilde{\mathcal{G}})\{l\}$. Une chasse au diagramme permet alors d'établir que $\mathcal{D}^{d+2-a}(U_1,\tilde{\mathcal{G}}_t)\{l\} = \Sha^{d+2-a}(\tilde{G}_t)\{ l \}$ et $\mathcal{D}^{d+3-a}(U_1,\tilde{\mathcal{G}})\{l\} = \Sha^{d+3-a}(\tilde{G})\{ l \}$.
\item[(ii)] La preuve est tout à fait analogue.
\end{itemize}
\end{proof}

\begin{lemma} \label{exacte gm fini}
On rappelle que l'on a supposé (H \ref{40}).
\begin{itemize}
\item[(i)] On a une suite exacte:
$$\bigoplus_{v \in X^{(1)}} H^{d+1-a}(K_v,\tilde{G}\otimes \mathbb{Z}/n\mathbb{Z})  \rightarrow H^{d+2-a}_c(U,\tilde{\mathcal{G}}\otimes \mathbb{Z}/n\mathbb{Z}) \rightarrow \mathcal{D}^{d+2-a}(U,\tilde{\mathcal{G}}\otimes \mathbb{Z}/n\mathbb{Z}) \rightarrow 0.$$
\item[(ii)] On a une suite exacte:
$$\bigoplus_{v \in X^{(1)}} H^{a-2}(K_v,G\otimes \mathbb{Z}/n\mathbb{Z})  \rightarrow H^{a-1}_c(U,\mathcal{G}\otimes \mathbb{Z}/n\mathbb{Z}) \rightarrow \mathcal{D}^{a-1}(U,\mathcal{G}\otimes \mathbb{Z}/n\mathbb{Z}) \rightarrow 0.$$
\end{itemize}
\end{lemma}

\begin{proof}
\begin{itemize}
\item[(i)] En procédant exactement de la même manière que pour les modules finis, il suffit de montrer que le morphisme $H^{d+2-a}(\mathcal{O}_v^h,\tilde{\mathcal{G}}\otimes \mathbb{Z}/n\mathbb{Z}) \rightarrow H^{d+2-a}(K_v^h,\tilde{G}\otimes \mathbb{Z}/n\mathbb{Z})$ est injectif. On dispose d'un diagramme commutatif à lignes exactes:
\begin{multline*}
\xymatrix{
H^{d+2-a}(\mathcal{O}_v^h,\tilde{\mathcal{T}_2}\otimes \mathbb{Z}/n\mathbb{Z}) \ar[r] \ar@{^{(}->}[d] & H^{d+2-a}(\mathcal{O}_v^h,\tilde{\mathcal{T}_1}\otimes \mathbb{Z}/n\mathbb{Z}) \ar[r] \ar@{^{(}->}[d] & ... \\
 H^{d+2-a}(K_v^h,\tilde{T_2}\otimes \mathbb{Z}/n\mathbb{Z}) \ar[r] & H^{d+2-a}(K_v^h,\tilde{T_1}\otimes \mathbb{Z}/n\mathbb{Z}) \ar[r] &  ...
}\\
\xymatrix{
... \ar[r] &H^{d+2-a}(\mathcal{O}_v^h,\tilde{\mathcal{G}}\otimes \mathbb{Z}/n\mathbb{Z}) \ar[r] \ar[d] &H^{d+3-a}(\mathcal{O}_v^h,\tilde{\mathcal{T}_2}\otimes \mathbb{Z}/n\mathbb{Z}) \ar@{^{(}->}[d]\\
... \ar[r] & H^{d+2-a}(K_v^h,\tilde{G}\otimes \mathbb{Z}/n\mathbb{Z}) \ar[r] & H^{d+3-a}(K_v^h,\tilde{T_2}\otimes \mathbb{Z}/n\mathbb{Z}).
}
\end{multline*}
De plus, les deux flèches verticales de gauche admettent des sections qui font commuter le diagramme:\\
\centerline{\xymatrix{
H^{d+2-a}(\mathcal{O}_v^h,\tilde{\mathcal{T}_2}\otimes \mathbb{Z}/n\mathbb{Z}) \ar[r] & H^{d+2-a}(\mathcal{O}_v^h,\tilde{\mathcal{T}_1}\otimes \mathbb{Z}/n\mathbb{Z}) \\
 H^{d+2-a}(K_v^h,\tilde{T_2}\otimes \mathbb{Z}/n\mathbb{Z}) \ar[r]\ar[u] & H^{d+2-a}(K_v^h,\tilde{T_1}\otimes \mathbb{Z}/n\mathbb{Z})\ar[u]
}}
Une chasse au diagramme permet alors de conclure que $H^{d+2-a}(\mathcal{O}_v^h,\tilde{\mathcal{G}}\otimes \mathbb{Z}/n\mathbb{Z}) \rightarrow H^{d+2-a}(K_v^h,\tilde{G}\otimes \mathbb{Z}/n\mathbb{Z})$ est injectif.
\item[(ii)] La preuve est tout à fait analogue.
\end{itemize}
\end{proof}

\begin{lemma} \label{exacte gmt}
On rappelle que l'on a supposé (H \ref{40}).
\begin{itemize}
\item[(i)] On a une suite exacte:
$$\bigoplus_{v \in X^{(1)}} H^{d+1-a}(K_v,\tilde{G}_t)  \rightarrow H^{d+2-a}_c(U,\tilde{\mathcal{G}}_t) \rightarrow \mathcal{D}^{d+2-a}(U,\tilde{\mathcal{G}}_t) \rightarrow 0.$$
\item[(ii)] On a une suite exacte:
$$\bigoplus_{v \in X^{(1)}} H^{a-2}(K_v,G_t)  \rightarrow H^{a-1}_c(U,\mathcal{G}_t) \rightarrow \mathcal{D}^{a-1}(U,\mathcal{G}_t) \rightarrow 0.$$
\end{itemize}
\end{lemma}

\begin{proof}
\begin{itemize}
\item[(i)] On a un diagramme commutatif à lignes exactes:\\
\scriptsize
\xymatrix{
\varinjlim_n H^{d+1-a}(K_v,\tilde{T_2} \otimes \mathbb{Z}/n\mathbb{Z}) \ar[r]\ar[d]^{\cong} & \varinjlim_n H^{d+1-a}(K_v,\tilde{T_1} \otimes \mathbb{Z}/n\mathbb{Z}) \ar[r]\ar[d]^{\cong} & \varinjlim_n H^{d+1-a}(K_v,\tilde{G} \otimes \mathbb{Z}/n\mathbb{Z}) \ar[r]\ar[d] &... \\
H^{d+1-a}(K_v,\tilde{T_2}_t) \ar[r] &H^{d+1-a}(K_v,\tilde{T_1}_t) \ar[r] &H^{d+1-a}(K_v,\tilde{G}_t) \ar[r] &...
}\\
\xymatrix{
\hspace{130pt}...\ar[r] &\varinjlim_n H^{d+2-a}(K_v,\tilde{T_2} \otimes \mathbb{Z}/n\mathbb{Z}) \ar[r]\ar[d]^{\cong} & \varinjlim_n H^{d+2-a}(K_v,\tilde{T_1} \otimes \mathbb{Z}/n\mathbb{Z})\ar[d]^{\cong} \\
\hspace{130pt}...\ar[r]&H^{d+2-a}(K_v,\tilde{T_2}_t) \ar[r] &H^{d+2-a}(K_v,\tilde{T_1}_t).
}\\
\normalsize
Le lemme des cinq permet alors d'établir que $ \varinjlim_n H^{d+1-a}(K_v,\tilde{G} \otimes \mathbb{Z}/n\mathbb{Z}) \cong H^{d+1-a}(K_v,\tilde{G}_t) $. De la même manière, on prouve que $\varinjlim_n H^{d+2-a}_c(U,\tilde{\mathcal{G}}\otimes \mathbb{Z}/n\mathbb{Z}) \cong H^{d+2-a}_c(U,\tilde{\mathcal{G}}_t)$ et que $\varinjlim_n H^{d+2-a}(K,\tilde{G}\otimes \mathbb{Z}/n\mathbb{Z}) \cong H^{d+2-a}(K,\tilde{G}_t)$. Par conséquent, en passant à la limite inductive dans les suites exactes:
$$\bigoplus_{v \in X^{(1)}} H^{d+1-a}(K_v,\tilde{G}\otimes \mathbb{Z}/n\mathbb{Z})  \rightarrow H^{d+2-a}_c(U,\tilde{\mathcal{G}}\otimes \mathbb{Z}/n\mathbb{Z}) \rightarrow \mathcal{D}^{d+2-a}(U,\tilde{\mathcal{G}}\otimes \mathbb{Z}/n\mathbb{Z}) \rightarrow 0$$
on obtient la suite exacte:
$$\bigoplus_{v \in X^{(1)}} H^{d+1-a}(K_v,\tilde{G}_t)  \rightarrow H^{d+2-a}_c(U,\tilde{\mathcal{G}}_t) \rightarrow \mathcal{D}^{d+2-a}(U,\tilde{\mathcal{G}}_t) \rightarrow 0.$$
\item[(ii)] La preuve est tout à fait analogue.
\end{itemize}
\end{proof}

\begin{proposition}\label{exacte gm}
On rappelle que l'on a supposé (H \ref{40}).
\begin{itemize}
\item[(i)] Soit $l$ un nombre premier. Soit $U$ un ouvert de $U_1$. On a une suite exacte:
$$\bigoplus_{v \in X^{(1)}} \overline{H^{d+2-a}(K_v,\tilde{G})\{l\} } \rightarrow \overline{H^{d+3-a}_c(U,\tilde{\mathcal{G}})\{l\}} \rightarrow \overline{\mathcal{D}^{d+3-a}(U,\tilde{\mathcal{G}})\{l\}} \rightarrow 0.$$
\item[(ii)] Supposons $\hat{T_2} \rightarrow \hat{T_1}$ injectif. Soit $l$ un nombre premier. Soit $U$ un ouvert de $U_2$. On a une suite exacte:
$$\bigoplus_{v \in X^{(1)}} \overline{H^{a-1}(K_v,G)\{l\} } \rightarrow \overline{H^{a}_c(U,\mathcal{G})\{l\}} \rightarrow \overline{\mathcal{D}^{a}(U,\mathcal{G})\{l\}} \rightarrow 0.$$
\end{itemize}
\end{proposition}

\begin{proof}
\begin{itemize}
\item[(i)] On dispose d'un diagramme commutatif où, d'après le lemme \ref{exacte gmt}, la première ligne est exacte:\\
\centerline{\xymatrix{
\bigoplus_{v \in X^{(1)}} \overline{H^{d+1-a}(K_v,\tilde{G}_t)}\{l\}  \ar[r] \ar[d] & \overline{H^{d+2-a}_c(U,\tilde{\mathcal{G}}_t)}\{l\} \ar[r] \ar[d] & \overline{\mathcal{D}^{d+2-a}(U,\tilde{\mathcal{G}}_t)}\{l\} \ar[r] \ar[d] & 0 
\\
\bigoplus_{v \in X^{(1)}} \overline{H^{d+2-a}(K_v,\tilde{G}) }\{l\} \ar[r]  & \overline{H^{d+3-a}_c(U,\tilde{\mathcal{G}})}\{l\} \ar[r]  & \overline{\mathcal{D}^{d+3-a}(U,\tilde{\mathcal{G}})}\{l\} \ar[r]  & 0
}}
Il suffit donc de montrer que les morphismes verticaux sont des isomorphismes. 
\begin{itemize}
\item[$\bullet$] On dispose de la suite exacte:
$$H^{d+1-a}(K_v,\tilde{G} \otimes \mathbb{Q}) \rightarrow H^{d+1-a}(K_v,\tilde{G}_t) \rightarrow H^{d+2-a}(K_v,\tilde{G}) \rightarrow H^{d+2-a}(K_v,\tilde{G} \otimes \mathbb{Q}).$$
Comme $H^{d+2-a}(K_v,\tilde{T_1} \otimes \mathbb{Q}) = H^{d+3-a}(K_v,\tilde{T_2} \otimes \mathbb{Q})=0$, on a $ H^{d+2-a}(K_v,\tilde{G} \otimes \mathbb{Q})=0$. Par conséquent, comme $H^{d+1-a}(K_v,\tilde{G} \otimes \mathbb{Q})$ est divisible, on obtient un isomorphisme $\overline{H^{d+1-a}(K_v,\tilde{G}_t)} \cong  \overline{H^{d+2-a}(K_v,\tilde{G}) }$.
\item[$\bullet$] On dispose de la suite exacte:
$$H^{d+2-a}_c(U,\tilde{\mathcal{G}} \otimes \mathbb{Q}) \rightarrow H^{d+2-a}_c(U,\tilde{\mathcal{G}}_t) \rightarrow H^{d+3-a}_c(U,\tilde{\mathcal{G}}) \rightarrow H^{d+3-a}_c(U,\tilde{\mathcal{G}} \otimes \mathbb{Q}).$$
Comme $H^{d+3-a}(U,\tilde{\mathcal{T}_1} \otimes \mathbb{Q}) = H^{d+4-a}(U,\tilde{\mathcal{T}_2}\otimes \mathbb{Q}) = 0$, on a $H^{d+3-a}(U,\tilde{\mathcal{G}} \otimes \mathbb{Q})=0$. De plus, $H^{d+2-a}(K_v,\tilde{G} \otimes \mathbb{Q}) = 0$ pour tout $v\in X^{(1)}$. Donc $H^{d+3-a}_c(U,\tilde{\mathcal{G}} \otimes \mathbb{Q})=0$. Par conséquent, comme $H^{d+2-a}_c(U,\tilde{\mathcal{G}} \otimes \mathbb{Q})$ est divisible, on obtient un isomorphisme $\overline{H^{d+2-a}_c(U,\tilde{\mathcal{G}}_t)} \cong  \overline{H^{d+3-a}_c(U,\tilde{\mathcal{G}})}$.
\item[$\bullet$] D'après le lemme \ref{iso}, on a un isomorphisme $\Sha^{d+2-a}(\tilde{G}_t) \cong \Sha^{d+3-a}(\tilde{G})$, d'où un isomorphisme $\overline{\mathcal{D}^{d+2-a}(U,\tilde{\mathcal{G}}_t)}\{l\} \cong \overline{\mathcal{D}^{d+3-a}(U,\tilde{\mathcal{G}})}\{l\}$ d'après \ref{Sha gm} puisque $U \subseteq U_1$.
\end{itemize}
\item[(ii)] En procédant exactement de la même manière, il s'agit de montrer que l'on a des isomorphismes $\overline{H^{a-2}(K_v,G_t)} \cong  \overline{H^{a-1}(K_v,G) }$, $\overline{H^{a-1}_c(U,\mathcal{G}_t)} \cong  \overline{H^{a}_c(U,\mathcal{G})}$ et $\overline{\mathcal{D}^{a-1}(U,\mathcal{G}_t)}\{l\} \cong \overline{\mathcal{D}^{a}(U,\mathcal{G})}\{l\}$. On les établit de la même manière que dans (i) en remarquant que, comme $H^{a-1}(K_v,G)$ et $H^{a}_c(U,\mathcal{G})$ sont de torsion, les morphismes $H^{a-1}(K_v,G) \rightarrow H^{a-1}(K_v,G \otimes \mathbb{Q})$ et $H^{a}_c(U,\mathcal{G}) \rightarrow H^{a}_c(U,\mathcal{G} \otimes \mathbb{Q})$ sont nuls.
\end{itemize}
\end{proof}

\begin{theorem} \label{PT gm} \textbf{(Dualité de Poitou-Tate pour les groupes de type multiplicatif)}\\
On rappelle que l'on a supposé (H \ref{40}), c'est-à-dire que $X$ est une courbe, et que $a \in \{0,1,...,d+1\}$. 
\begin{itemize}
\item[(i)] On fait l'une des deux hypothèses suivantes:
\begin{hypoth}\label{1}
\begin{minipage}[t]{11.9cm}
le morphisme $\hat{T_2} \rightarrow \hat{T_1}$ est injectif,
\end{minipage}
\end{hypoth}
\begin{hypoth}\label{2}
\begin{minipage}[t]{11.9cm}
le groupe $K_a^M(L)_{tors}$ est d'exposant fini pour tout corps $(d+1)$-local $L$ dont le corps résiduel est une extension finie de $k$.
\end{minipage}
\end{hypoth}
On a alors un accouplement parfait:
$$ \overline{\Sha^{a-1}(G)_{tors}} \times \overline{\Sha^{d+3-a}(\tilde{G})} \rightarrow \mathbb{Q}/\mathbb{Z}.$$
\item[(ii)] Sous l'hypothèse (H \ref{1}), on a un accouplement parfait:
$$ \overline{\Sha^{a}(G)} \times \Sha^{d+2-a}(\tilde{G}) \rightarrow \mathbb{Q}/\mathbb{Z}.$$
\end{itemize}
\end{theorem}

\begin{proof}
\begin{itemize}
\item[(i)] Fixons un nombre premier $l$. Pour $U$ un ouvert de $U_1$, posons $D^{a-1}_{sh}(U,\mathcal{G}) = \text{Ker}(H^{a-1}(U,\mathcal{G}) \rightarrow \prod_{v \in X^{(1)}} H^{a-1}(K_v,G))$. D'après \ref{exacte gm} et la proposition 4.3 de \cite{CTH}, on dispose d'un diagramme commutatif à lignes exactes:\\
\small
\centerline{\xymatrix{
0 \ar[r] & D^{a-1}_{sh}(U,\mathcal{G})\{l\} \ar[r]\ar[d] & H^{a-1}(U,\mathcal{G})\{l\} \ar[r]\ar[d]& \prod_{v \in X^{(1)}} H^{a-1}(K_v,G)\{l\}\ar[d]\\
0 \ar[r] & (\overline{\mathcal{D}^{d+3-a}(U,\tilde{\mathcal{G}})}\{l\})^D \ar[r]& (\overline{H^{d+3-a}_c(U,\tilde{\mathcal{G}})}\{l\})^D \ar[r]& (\bigoplus_{v \in X^{(1)}} \overline{H^{d+2-a}(K_v,\tilde{G})}\{l\})^D
}}
\normalsize
La proposition \ref{AV gm} impose que le morphisme vertical central est surjectif et que son noyau est divisible. De plus, les hypothèses (H \ref{1}) et (H \ref{2}) imposent que $H^{a-1}(K_v,G)\{l\}$ est d'exposant fini pour tout $v \in X^{(1)}$ d'après \ref{local gm}. On en déduit que le morphisme naturel $H^{a-1}(K_v,G)\{l\} \rightarrow H^{a-1}(K_v,G)^{\wedge}$ est injectif. La proposition \ref{local gm 2} montre alors que le morphisme $\prod_{v \in X^{(1)}} H^{a-1}(K_v,G)\{l\} \rightarrow (\bigoplus_{v \in X^{(1)}} \overline{H^{d+2-a}(K_v,\tilde{G})}\{l\})^D$ est injectif. Par conséquent le morphisme $D^{a-1}_{sh}(U,\mathcal{G})\{l\} \rightarrow (\overline{\mathcal{D}^{d+3-a}(U,\tilde{\mathcal{G}})}\{l\})^D$ est surjectif et son noyau est divisible. En passant à la limite inductive sur $U$ et en utilisant \ref{Sha gm}, on obtient un isomorphisme:
$$\Sha^{a-1}(G)\{l\}/D_l \cong \overline{\Sha^{d+3-a}(\tilde{G})\{l\}}^D,$$
où $D_l$ est contenu dans le sous-groupe divisible de $\Sha^{a-1}(G)\{l\}$. Or $\overline{\Sha^{d+3-a}(\tilde{G})\{l\}}^D$ n'a aucun élément divisible. Donc $D_l$ est le sous-groupe divisible de $\Sha^{a-1}(G)\{l\}$ et on a un isomorphisme:
$$\overline{\Sha^{a-1}(G)\{l\}} \cong \overline{\Sha^{d+3-a}(\tilde{G})\{l\}}^D.$$
\item[(ii)] La méthode est similaire. Fixons un nombre premier $l$. Pour $U$ un ouvert de $U_2$, posons $D^{d+2-a}_{sh}(U,\tilde{\mathcal{G}}) = \text{Ker}(H^{d+2-a}(U,\tilde{\mathcal{G}}) \rightarrow \prod_{v \in X^{(1)}} H^{d+2-a}(K_v,\tilde{G}))$. D'après \ref{exacte gm} et la proposition 4.3 de \cite{CTH}, on dispose d'un diagramme commutatif à lignes exactes:\\
\small
\centerline{\xymatrix{
0 \ar[r] & D^{d+2-a}_{sh}(U,\tilde{\mathcal{G}})\{l\} \ar[r]\ar[d] & H^{d+2-a}(U,\tilde{\mathcal{G}})\{l\} \ar[r]\ar[d]& \prod_{v \in X^{(1)}} H^{d+2-a}(K_v,\tilde{G}))\{l\}\ar[d]\\
0 \ar[r] & (\overline{\mathcal{D}^{a}(U,\mathcal{G})}\{l\})^D \ar[r]& (\overline{H^{a}_c(U,\mathcal{G})}\{l\})^D \ar[r]& (\bigoplus_{v \in X^{(1)}} \overline{H^{a-1}(K_v,G)}\{l\})^D
}}
\normalsize
La proposition \ref{AV gm} impose que le morphisme vertical central est surjectif et que son noyau est divisible. De plus, l'hypothèse (H \ref{1}) impose que le morphisme vertical de droite est un isomorphisme d'après \ref{local gm} et \ref{local gm 2}. Par conséquent le morphisme $D^{d+2-a}_{sh}(U,\tilde{\mathcal{G}})\{l\} \rightarrow (\overline{\mathcal{D}^{a}(U,\mathcal{G})}\{l\})^D$ est surjectif et son noyau est divisible. En passant à la limite inductive sur $U$ et en utilisant \ref{Sha gm}, on obtient un isomorphisme:
$$\Sha^{d+2-a}(\tilde{G})\{l\}/D_l \cong \overline{\Sha^{a}(G)\{l\}}^D,$$
où $D_l$ est contenu dans le sous-groupe divisible de $\Sha^{d+2-a}(\tilde{G})\{l\}$. Or $\overline{\Sha^{a}(G)\{l\}}^D$ n'a aucun élément divisible. Donc $D_l$ est le sous-groupe divisible de $\Sha^{d+2-a}(\tilde{G})\{l\}$ et on a un isomorphisme:
$$\overline{\Sha^{d+2-a}(\tilde{G})\{l\}} \cong \overline{\Sha^{a}(G)\{l\}}^D.$$
Comme $\hat{T_2} \rightarrow \hat{T_1}$ est injectif, si l'on note $Q$ son conoyau, $\tilde{G}$ est quasi-isomorphe à $Q \otimes \mathbb{Z}(d+1-a)$. Par conséquent, la conjecture de Beilinson-Lichtenbaum (avec un argument de restriction-corestriction) permet d'établir que le groupe $\Sha^{d+2-a}(\tilde{G})$ est d'exposant fini, ce qui achève la preuve. 
\end{itemize}
\end{proof}

\begin{corollary}\label{cor}
On rappelle que l'on a supposé (H \ref{40}), c'est-à-dire que $X$ est une courbe, et que $a \in \{0,1,...,d+1\}$.
\begin{itemize}
\item[(i)] On fait l'une des deux hypothèses suivantes:
\begin{hypoth}\label{1'}
\begin{minipage}[t]{11.9cm}
le morphisme $\check{T_1} \rightarrow \check{T_2}$ est injectif,
\end{minipage}
\end{hypoth}
\begin{hypoth}\label{2'}
\begin{minipage}[t]{11.9cm}
le groupe $K_{d+1-a}^M(L)_{tors}$ est d'exposant fini pour tout corps $(d+1)$-local $L$ dont le corps résiduel est une extension finie de $k$.
\end{minipage}
\end{hypoth}
On a alors un accouplement parfait:
$$ \overline{\Sha^{a+1}(G)} \times \overline{\Sha^{d+1-a}(\tilde{G})_{tors}} \rightarrow \mathbb{Q}/\mathbb{Z}.$$
\item[(ii)] Sous l'hypothèse (H \ref{1'}), on a un accouplement parfait:
$$ \Sha^{a}(G) \times \overline{\Sha^{d+2-a}(\tilde{G})} \rightarrow \mathbb{Q}/\mathbb{Z}.$$
\end{itemize}
\end{corollary}

\begin{proof}
Notons $H = [\tilde{T_2} \rightarrow \tilde{T_1}][1]$ et $\tilde{H} = [T_1 \rightarrow T_2][-1]$. En appliquant le théorème précédent à $H$ au lieu de $G$, on obtient:
\begin{itemize}
\item[$\bullet$] sous (H \ref{1'}) ou (H \ref{2'}), un accouplement parfait:
$$\overline{\Sha^{d-a}(H)_{tors}} \times \overline{\Sha^{a+2}(\tilde{H})} \rightarrow \mathbb{Q}/\mathbb{Z},$$
\item[$\bullet$] sous (H \ref{1'}), un accouplement parfait: $$\overline{\Sha^{d+1-a}(H)} \times \Sha^{a+1}(\tilde{H}) \rightarrow \mathbb{Q}/\mathbb{Z},$$
\end{itemize}
d'où le résultat.
\end{proof}

\begin{remarque}\label{rq gm}
\begin{itemize}
\item[(i)] Comme indiqué dans la remarque \ref{Milnor}, l'hypothèse (H \ref{2}) est toujours vérifiée pour $a=0$, et elle est aussi vérifiée pour $a=1$ lorsque $k_1$ est $p$-adique. Par contre, je ne sais pas si elle est vérifiée dans d'autres cas.
\item[(ii)] Soit $G_{tm}$ un groupe de type multiplicatif. Soient $T_1$ et $T_2$ deux tores s'insérant dans une suite exacte $0 \rightarrow G_{tm} \rightarrow T_1 \rightarrow T_2 \rightarrow 0$, de sorte que $G_{tm}$ s'identifie à $G[-1]$. On remarque alors que $\Sha^1(G_{tm})$ est de torsion d'exposant fini et que $\Sha^3(G_{tm})$ est de torsion. On obtient donc des accouplements parfaits:
$$\Sha^1(G_{tm}) \times \overline{\Sha^{d+2}(\tilde{G})} \rightarrow \mathbb{Q}/\mathbb{Z} \hspace{15pt} \text{et} \hspace{15pt} \overline{\Sha^2(G_{tm})} \times \Sha^{d+1}(\tilde{G}) \rightarrow \mathbb{Q}/\mathbb{Z},$$
et si l'hypothèse (H \ref{2}) est vérifiée  pour $a = d$, on a aussi un accouplement parfait:
$$ \overline{\Sha^{3}(G_{tm})} \times \overline{\Sha^{d}(\tilde{G})_{tors}} \rightarrow \mathbb{Q}/\mathbb{Z}.$$
\item[(iii)] Supposons que $a=1$ et que le morphisme $\check{T_1} \rightarrow \check{T_2}$ est injectif. On dispose alors de trois accouplements parfaits:
$$ \overline{\Sha^{0}(G)_{tors}} \times \overline{\Sha^{d+2}(\tilde{G})} \rightarrow \mathbb{Q}/\mathbb{Z}, \hspace{20pt} \Sha^{1}(G) \times \overline{\Sha^{d+1}(\tilde{G})} \rightarrow \mathbb{Q}/\mathbb{Z},$$
$$ \overline{\Sha^{2}(G)} \times \overline{\Sha^{d}(\tilde{G})_{tors}} \rightarrow \mathbb{Q}/\mathbb{Z}.$$
\item[(iv)] Si $T_2 = 0$, on retrouve la dualité parfaite $ \Sha^a(T_1) \times \overline{\Sha^{d+3-a}(\tilde{T_1})} \rightarrow \mathbb{Q}/\mathbb{Z}$.
\item[(v)] Si l'hypothèse (H \ref{2}) est vérifiée et $T_1=0$, on obtient une dualité:
$$ \overline{\Sha^{a-1}(T_2)_{tors}} \times \overline{\Sha^{d+4-a}(\tilde{T_2})} \rightarrow \mathbb{Q}/\mathbb{Z}.$$
Lorsque $k_1$ est $p$-adique et $a=1$, cela prouve en particulier que le groupe $\Sha^{d+3}(\tilde{T_2})$ est divisible. Si l'on admet que la torsion dans les groupes de K-théorie de Milnor des corps locaux supérieurs est d'exposant fini, en prenant $T_2 = \mathbb{Z}(a)[1]$ et à l'aide de la conjecture de Beilinson-Lichtenbaum et de l'isomorphisme de Nesterenko-Suslin-Totaro, on obtient une dualité parfaite:
$$ \overline{\text{Ker}\left( K^M_a(K) \rightarrow \prod_{v \in X^{(1)}} K^M_a(K_v)\right)_{tors}}  \times \overline{\Sha^{d+4-a}(\mathbb{Z}(d+1-a))} \rightarrow \mathbb{Q}/\mathbb{Z}.$$
\item[(vi)] Les points (iv) et (v) permettent de comprendre mieux les groupes de Tate-Shafarevich de $\mathbb{Z}(d)$ lorsque $k_1$ est $p$-adique: 
\begin{itemize}
\item[$\bullet$] le groupe $\Sha^{d+1}(\mathbb{Z}(d))$ est nul d'après la conjecture de Beilinson-Lichtenbaum.
\item[$\bullet$] les groupes $\Sha^{d+2}(\mathbb{Z}(d))$ et $\Sha^{d+3}(\mathbb{Z}(d))$ sont divisibles d'après (iv) et (v).
\item[$\bullet$] le groupe $\Sha^{r}(\mathbb{Z}(d))$ pour $r \geq d+4$ est nul par dimension cohomologique.
\end{itemize}
\end{itemize}
\end{remarque}

\begin{remarque}
Supposons que $k_1$ est de caractéristique $p>0$. Dans ce cas:
\begin{itemize}
\item[$\bullet$] le lemme \ref{acc gm} reste vrai.
\item[$\bullet$] la proposition \ref{AV gm fini} reste valable pour $n$ non divisible par $p$ et le théorème \ref{AV gm} reste vrai pour $l \neq p$.
\item[$\bullet$] dans la proposition \ref{local gm}, l'assertion (i) reste vraie; dans (ii) et (iii), il faut remplacer $K_a^M(L)_{tors}$ par $K_a^M(L)_{\text{non}-p}$.
\item[$\bullet$] la proposition \ref{local gm 2} fournit un accouplement parfait:
$$\varprojlim_{p \nmid n} H^{a-1}(K_v,G)/n \times H^{d+2-a}(K_v,\tilde{G})_{\text{non}-p} \rightarrow \mathbb{Q}/\mathbb{Z}$$
lorsque $\hat{T_2} \rightarrow \hat{T_1}$ est injectif, et une injection $$\varprojlim_{p \nmid n} H^{a-1}(K_v,G)/n \rightarrow (H^{d+2-a}(K_v,\tilde{G})_{\text{non}-p})^D$$
dans le cas général.
\item[$\bullet$] concernant la proposition \ref{nature gm} et le corollaire \ref{tcof gm}, les groupes $H^{d+3-a}(U,\tilde{\mathcal{G}})_{\text{non}-p}$, $H^{d+2-a}(K_v,\tilde{G})_{\text{non}-p}$, $H^{d+2-a}(K_v^h,\tilde{G})_{\text{non}-p}$, $H^{d+3-a}_c(U,\tilde{\mathcal{G}})_{\text{non}-p}$ et $\mathcal{D}^{d+3-a}(U,\tilde{\mathcal{G}})_{\text{non}-p}$ sont de torsion de type cofini, et si $\hat{T_2} \rightarrow \hat{T_1}$ est injectif, les groupes $H^a(U,\mathcal{G})_{\text{non}-p}$, $H^{a-1}(K_v,G)_{\text{non}-p}$, $H^{a-1}(K_v^h,G)_{\text{non}-p}$, $H^a_c(U,\mathcal{G})_{\text{non}-p}$ et $\mathcal{D}^a(U,\mathcal{G})_{\text{non}-p}$ sont de torsion de type cofini.
\item[$\bullet$] le lemme \ref{iso} reste vrai.
\item[$\bullet$] la proposition \ref{Sha gm} reste vraie pour $l \neq p$.
\item[$\bullet$] le lemme \ref{exacte gm fini} reste vraie pour $n$ non divisible par $p$.
\item[$\bullet$] concernant le lemme \ref{exacte gmt}, on a une suite exacte:
$$\bigoplus_{v \in X^{(1)}} H^{d+1-a}(K_v,\tilde{G}_t)_{\text{non}-p}  \rightarrow H^{d+2-a}_c(U,\tilde{\mathcal{G}}_t)_{\text{non}-p} \rightarrow \mathcal{D}^{d+2-a}(U,\tilde{\mathcal{G}}_t)_{\text{non}-p} \rightarrow 0.$$
\item[$\bullet$] la proposition \ref{exacte gm} reste vraie pour $l \neq p$.
\item[$\bullet$] le théorème \ref{PT gm} fournit un accouplement parfait:
$$ \overline{\Sha^{a-1}(G)_{\text{non}-p}} \times \overline{\Sha^{d+3-a}(\tilde{G})_{\text{non}-p}} \rightarrow \mathbb{Q}/\mathbb{Z}$$
sous l'hypothèse (H \ref{1}) ou sous l'hypothèse:
\begin{hypoth}\label{3}
\begin{minipage}[t]{11.9cm}
 le groupe $K_a^M(L)_{\text{non}-p}$ est d'exposant fini pour tout corps $(d+1)$-local $L$ dont le corps résiduel est une extension finie de $k$,
\end{minipage}
\end{hypoth}
et son corollaire fournit un accouplement parfait:
$$ \overline{\Sha^{a+1}(G)_{\text{non}-p}} \times \overline{\Sha^{d+1-a}(\tilde{G})_{\text{non}-p}} \rightarrow \mathbb{Q}/\mathbb{Z}$$
sous les hypothèses (H \ref{1'}) ou sous l'hypothèse:
\begin{hypoth}\label{3'}
\begin{minipage}[t]{11.9cm}
le groupe $K_{d+1-a}^M(L)_{\text{non}-p}$ est d'exposant fini pour tout corps $(d+1)$-local $L$ dont le corps résiduel est une extension finie de $k$.
\end{minipage}
\end{hypoth}
\item[$\bullet$] le théorème \ref{PT gm} fournit un accouplement parfait:
$$\overline{\Sha^a(G)_{\text{non}-p}} \times \Sha^{d+2-a}(\tilde{G})_{\text{non}-p} \rightarrow \mathbb{Q}/\mathbb{Z}$$
sous l'hypothèse (H \ref{1}), et son corollaire fournit un accouplement parfait:
$$\Sha^a(G)_{\text{non}-p} \times \overline{\Sha^{d+2-a}(\tilde{G})_{\text{non}-p}} \rightarrow \mathbb{Q}/\mathbb{Z}$$
sous l'hypothèse (H \ref{1'}).
\end{itemize}

\end{remarque}

\section{\scshape Finitude de certains groupes de Tate-Shafarevich}\label{finitude}

Dans toute cette section, nous supposons (H \ref{40}), c'est-à-dire que $X$ est une courbe. 

\subsection{Finitude du deuxième groupe de Tate-Shafarevich d'un tore}\label{1.1}

Soit $T$ un tore sur $K$, déployé par une extension finie $L$. Notons $\hat{T}$ (resp. $\check{T}$) son module des caractères (resp. cocaractères), et posons $\tilde{T} = \hat{T} \otimes^{\mathbf{L}} \mathbb{Z}(d)$. En remarquant que $T$ est quasi-isomorphe à $\check{T} \otimes^{\mathbf{L}} \mathbb{Z}(1)[1]$, le théorème \ref{PT tore} fournit des accouplements parfaits de groupes finis:
$$\Sha^1(T) \times \overline{\Sha^{d+2}( \tilde{T})} \rightarrow \mathbb{Q}/\mathbb{Z},$$
$$\overline{\Sha^2(T)} \times \Sha^{d+1}( \tilde{T}) \rightarrow \mathbb{Q}/\mathbb{Z}.$$
Il est alors naturel de se demander si les groupes $\Sha^{d+2}( \tilde{T})$ et $\Sha^2(T)$ sont finis de sorte que $\overline{\Sha^{d+2}( \tilde{T})}=\Sha^{d+2}( \tilde{T})$ et $\overline{\Sha^2(T)} = \Sha^2(T)$. Comme $\Sha^{d+2}( \tilde{T})$ et $\Sha^2(T)$ sont de torsion de type cofini, un argument de restriction-corestriction montre que les nullités de $\Sha^{d+2}(L,\mathbb{Z}(d))$ et $\Sha^2(L,\mathbb{G}_m)$ impliquent les finitudes de $\Sha^{d+2}( \tilde{T})$ et $\Sha^2(T)$. Il est donc intéressant d'étudier la nullité des groupes $\Sha^{d+2}(\mathbb{Z}(d))$ et $\Sha^2(\mathbb{G}_m)$, ce qui fait l'objet de cette section. Pour ce faire, commençons par rappeler le lemme \ref{nul dual}:

\begin{lemma}
On rappelle que l'on a supposé (H \ref{40}). Le groupe $\Sha^{d+2}(\mathbb{Z}(d))$ est nul si, et seulement si, le groupe $\Sha^2(\mathbb{G}_m)$ l'est aussi.
\end{lemma}

Par conséquent, à partir de maintenant, nous pouvons nous intéresser exclusivement au groupe $\Sha^2(\mathbb{G}_m)$ qui semble plus simple à manipuler. Un premier résultat dans ce sens est exprimé dans le théorème suivant:

\begin{theorem} \label{nul 1}
On rappelle que l'on a supposé (H \ref{40}), c'est-à-dire que $X$ est une courbe. Notons $\overline{X} = X \times_k k^s$ et supposons que $\overline{X} \cong \mathbb{P}^1_{\overline{k}}$ (c'est-à-dire que $X$ est la droite projective ou une conique dans $\mathbb{P}^2_k$). Alors $\Sha^2(\mathbb{G}_m)$ est nul.
\end{theorem}

La remarque suivante sera utile dans la preuve:

\begin{remarque}
Soit $x \in \Sha^2(\mathbb{G}_m)$. On voit alors que le résidu de $x$ dans $H^1(k(v),\mathbb{Q}/\mathbb{Z})$ est nul pour chaque $v \in X^{(1)}$. Par conséquent, $x \in \text{Br}(X)$. Comme pour chaque $v \in X^{(1)} $ le groupe $\text{Br}(\mathcal{O}_v)$ est isomorphe à $\text{Br}(k(v))$ et s'injecte dans $\text{Br}(K_v)$, on en déduit que $x$ est dans le noyau de $\text{Br}(X) \rightarrow \prod_{v\in X^{(1)}} \text{Br}(k(v))$. Réciproquement, on voit aisément que $\text{Ker}\left( \text{Br}(X) \rightarrow \prod_{v\in X^{(1)}} \text{Br}(k(v))\right) $ est contenu dans $\Sha^2(\mathbb{G}_m)$, d'où l'égalité: $$\text{Ker}\left( \text{Br}(X) \rightarrow \prod_{v\in X^{(1)}} \text{Br}(k(v))\right)  = \Sha^2(\mathbb{G}_m).$$
\end{remarque}

\begin{proof}[Démonstration du théorème \ref{nul 1}]
Notons $(\text{Pic} (X))^{\perp}$ le noyau de $\text{Br}(X) \rightarrow \prod_{v\in X^{(1)}} \text{Br}(k(v))$. On a un isomorphisme  $(\text{Pic} (X))^{\perp} \cong \Sha^2(\mathbb{G}_m)$. Soit $x \in (\text{Pic} (X))^{\perp}$. On sait que $(\text{Pic} (X))^{\perp}$ est divisible. On en déduit que pour chaque entier $n>0$ il existe $x_n \in (\text{Pic} (X))^{\perp}$ tel que $x = nx_n$. Par ailleurs, remarquons que $\text{Br} (\overline{X}) = 0$, puisque $\text{Br} (\overline{X})$ s'injecte dans $\text{Br} (\overline{k}(X))$ et $\text{Br} (\overline{k}(X))=0$ d'après le théorème de Tsen. On obtient donc une suite exacte:
$$ \text{Br} (k) \rightarrow \text{Br} (X) \rightarrow H^1(k,\text{Pic} (\overline{X})).$$
Or, comme $\overline{X} \cong \mathbb{P}^1_{\overline{k}}$, le groupe $H^1(k,\text{Pic} (\overline{X}))$ est nul. On en déduit que $x$ et $x_n$ pour chaque $n$ sont dans l'image de $\text{Br}(k)$. Notons $\tilde{x}$ (resp. $\tilde{x_n}$) un élément de $\text{Br}(k)$ d'image $x$ (resp. $x_n$) dans $\text{Br}(X)$. Fixons maintenant $v_0 \in X^{(1)}$, et notons $n_0 = [k(v_0):k]$. Comme $x_{n_0} \in (\text{Pic} (X))^{\perp}$, on déduit que l'image de $\tilde{x_{n_0}}$ par la composée:
$$\text{Br}(k) \rightarrow \text{Br}(X) \rightarrow \text{Br}(k(v)) \rightarrow \text{Br}(k)$$
est 0. Mais un argument de restriction-corestriction impose que cette image est aussi $n_0\tilde{x_{n_0}}$. On en déduit que $n_0\tilde{x_{n_0}} = 0$, et donc $x = n_0x_{n_0} = 0$. Par conséquent, $(\text{Pic} (X))^{\perp}$ est nul.
\end{proof}

\begin{remarque}
\begin{itemize}
\item[$\bullet$] En fait, dans la preuve précédente, on pourrait supposer que $H^1(k,\text{Pic}(\overline{X}))$ est d'exposant fini $e$ au lieu de $\overline{X} \cong \mathbb{P}^1_{\overline{k}}$. En effet, dans ce cas, on prend pour $\tilde{x}$ (resp. $\tilde{x_n}$) un élément de $\text{Br}(k)$ d'image $ex$ (resp. $ex_n$) dans $\text{Br}(X)$, et on montre exactement de la même manière que $ex = n_0(ex_{n_0})$ est nul. On en déduit que le groupe divisible $(\text{Pic} (X))^{\perp}$ est d'exposant $e$, donc nul.
\item[$\bullet$] Supposons que $k$ est $p$-adique. En notant $J$ la jacobienne de $X$, le groupe $H^1(k,\text{Pic}(\overline{X}))$ est d'exposant fini si, et seulement si, $H^1(k,J)$ est d'exposant fini. Or $H^1(k,J)$ est isomorphe au dual de $J(k)$ d'après le théorème de dualité pour les variétés abéliennes sur un corps $p$-adique (corollaire 3.4 de \cite{MilADT}), et $J(k)$ est trivial si, et seulement si, $J$ est triviale d'après le théorème de structure de Mattuck (\cite{Mat}). Par conséquent, $H^1(k,\text{Pic}(\overline{X}))$ est d'exposant fini si, et seulement si, il est nul. Je ne sais pas si cette équivalence reste vraie lorsque $k$ n'est pas $p$-adique.
\end{itemize}
\end{remarque}

Notons maintenant $\mathcal{O}_k$ l'anneau des entiers de $k$, $\pi$ une uniformisante de $\mathcal{O}_k$ et $\kappa$ le corps résiduel de $\mathcal{O}_k$. Pour obtenir la nullité de $\Sha^2(\mathbb{G}_m)$ dans des situations plus générales (théorème \ref{0courbes}(ii)(iii) ou corollaires \ref{nul 3} et \ref{nul 4}), nous allons procéder par récurrence sur l'entier $d \geq 0$. Pour ce faire, nous allons commencer par établir la propriété d'hérédité. Dans le cas où $k_0$ est un corps fini, l'initialisation sera donnée par le cas où $d=1$, c'est-à-dire le cas où $k$ est $p$-adique, et elle découlera aisément des articles \cite{Kat} et \cite{HS1}. Dans le cas où $k_0= \mathbb{C}((t))$, l'initialisation sera donnée par le cas où $d=0$, c'est-à-dire le cas où $k=\mathbb{C}((t))$, et elle découlera aisément de l'article \cite{Dou}. Nous allons donc établir l'hérédité sous l'hypothèse suivante sur le corps $k$: 
\begin{hypo}\label{41}
\begin{minipage}[t]{12.72cm}
$d>1$ si $k_0$ est fini et $d>0$ si $k_0 = \mathbb{C}((t))$, c'est-à-dire $k$ n'est ni un corps fini ni un corps $p$-adique ni $\mathbb{C}((t))$.
\end{minipage}
\end{hypo}
Énonçons maintenant l'hypothèse de récurrence:
\begin{hypo}\label{42}
\begin{minipage}[t]{12.72cm}
\begin{itemize}
\item[(i)] il existe un schéma intègre, projectif, lisse $\mathcal{X}$ de dimension 2 sur $\text{Spec} \; \mathcal{O}_k$ dont la fibre générique est $X$ et dont la fibre spéciale, que nous notons $X_0$, est intègre de point générique $\eta_0$,
\item[(ii)] il existe un entier naturel non nul $e$ vérifiant la propriété suivante: pour tout $w \in X_0^{(1)}$, il existe un ouvert affine $\mathcal{U}_w = \text{Spec} \; \mathcal{A}_w$ de $\mathcal{X}$ contenant $w$ et un point fermé $v_w$ de $U_w=\mathcal{U}_w \times_{\mathcal{O}_k} k$ tels que l'adhérence $\overline{\{v_w\}}$ de $v_w$ dans $\mathcal{U}_w$, munie de sa structure réduite, est régulière, contient $w$ et $\pi$ est de valuation au plus $e$ dans l'anneau de valuation discrète $\mathcal{O}_{\overline{\{v_w\}},w}$,
\item[(iii)] les groupes $\Sha^2(\kappa(\eta_0),\mathbb{Z})$ et $\Sha^3(\kappa(\eta_0),\mathbb{Z}(1))$ sont nuls.
\end{itemize}
\end{minipage}
\end{hypo}

Sous de telles hypothèses, on notera $U_{w,0}$ la fibre spéciale de $\mathcal{U}_w $ pour chaque $w \in X_0$.

\begin{lemma}
On rappelle que l'on a supposé (H \ref{40}). On suppose aussi (H \ref{41}) et (H \ref{42}). Soient $r \in \{0,1\}$ et $n \geq 1$. On dispose d'un diagramme commutatif:\\
\centerline{\xymatrix{
H^{r+1}(X,\mathbb{Z}/n\mathbb{Z}(r)) \ar[r] \ar[d] & H^{r+1}(K,\mathbb{Z}/n\mathbb{Z}(r)) \ar[d]\\
H^r(X_0,\mathbb{Z}/n\mathbb{Z}(r-1)) \ar@{^{(}->}[r] & H^r(\kappa (\eta_0), \mathbb{Z}/n\mathbb{Z}(r-1)),
}}
où le morphisme vertical de droite est le résidu en $\eta_0$.
\end{lemma}

\begin{proof}
Le cas $r=0$ est évident puisque $H^r(X_0,\mathbb{Z}/n\mathbb{Z}(r-1)) \rightarrow H^r(\kappa (\eta_0), \mathbb{Z}/n\mathbb{Z}(r-1))$ est un isomorphisme.\\
Concernant le cas $r=1$, il suffit de montrer que l'image de la composée $H^2(X,\mu_n) \rightarrow H^2(K,\mu_n) \rightarrow H^1(\kappa (\eta_0), \mathbb{Z}/n\mathbb{Z})$ est contenue dans $H^1(X_0,\mathbb{Z}/n\mathbb{Z})$. Soit donc $x \in H^2(X,\mu_n)$. Écrivons le complexe de Bloch-Ogus (proposition 1.7 de \cite{Kat}):
$$H^2(K,\mu_n) \rightarrow \bigoplus_{v \in X^{(1)}} H^1(k(v),\mathbb{Z}/n\mathbb{Z}) \oplus H^1(\kappa(\eta_0), \mathbb{Z}/n\mathbb{Z}) \rightarrow \bigoplus_{w \in X_0^{(1)}} H^0(\kappa(w),\mathbb{Z}/n\mathbb{Z}(-1)).$$
Comme $ x \in H^2(X,\mu_n)$, l'image de $x$ dans $\bigoplus_{v \in X^{(1)}} H^1(k(v),\mathbb{Z}/n\mathbb{Z}) $ est nulle. On en déduit que l'image de $x$ dans $H^1(\kappa(\eta_0), \mathbb{Z}/n\mathbb{Z})$ est contenue dans:
$$\text{Ker} \left(  H^1(\kappa(\eta_0), \mathbb{Z}/n\mathbb{Z}) \rightarrow \bigoplus_{w \in X_0^{(1)}} H^0(\kappa(w),\mathbb{Z}/n\mathbb{Z}(-1))\right) = H^1(X_0,\mathbb{Z}/n\mathbb{Z}) .$$
\end{proof}

\begin{remarque}
On rappelle que l'on a supposé (H \ref{40}) et on suppose aussi (H \ref{41}) et (H \ref{42}). Sous de telles hypothèses, on rappelle que $U_{w,0}$ désigne la fibre spéciale $ \mathcal{U}_w \times_{\mathcal{O}_k} \kappa$ de $\mathcal{U}_w$. On remarque alors qu'une preuve tout à fait identique à celle qui précède fournit un diagramme commutatif:\\
\centerline{\xymatrix{
H^{r+1}(U_w,\mathbb{Z}/n\mathbb{Z}(r)) \ar@{^{(}->}[r] \ar[d] & H^{r+1}(K,\mathbb{Z}/n\mathbb{Z}(r)) \ar[d]\\
H^r(U_{w,0},\mathbb{Z}/n\mathbb{Z}(r-1)) \ar@{^{(}->}[r] & H^r(\kappa (\eta_0), \mathbb{Z}/n\mathbb{Z}(r-1)).
}}
\end{remarque}

\begin{lemma}
On rappelle que l'on a supposé (H \ref{40}). On suppose aussi (H \ref{41}) et (H \ref{42}). Soient $r \in \{0,1\}$ et $n \geq 1$. Soient $w \in X_0^{(1)}$ et $e_w$ la valuation de $\pi$ dans $\mathcal{O}_{\overline{\{v_w\}},w}$. Le diagramme suivant:\\
\centerline{\xymatrix{
H^{r+1}(X,\mathbb{Z}/n\mathbb{Z}(r)) \ar[r]^{\text{Res}_{k(v_w)}}\ar[d]^{\delta_{\eta_0}} & H^{r+1}(k(v_w),\mathbb{Z}/n\mathbb{Z}(r))\ar[d]^{\delta_w}\\
H^r(X_0,\mathbb{Z}/n\mathbb{Z}(r-1)) \ar[r]^{e_w\cdot\text{Res}_{\kappa(w)}} & H^r(\kappa(w),\mathbb{Z}/n\mathbb{Z}(r-1)),
}}
dont les morphismes verticaux sont des résidus, est commutatif.
\end{lemma}

\begin{proof}
Le diagramme:\\
\centerline{\xymatrix{
H^{r+1}(X,\mathbb{Z}/n\mathbb{Z}(r)) \ar[r] \ar[d] & H^{r+1}(U_w,\mathbb{Z}/n\mathbb{Z}(r)) \ar[d] \\
H^r(X_0,\mathbb{Z}/n\mathbb{Z}(r-1)) \ar[r] & H^r(U_{w,0},\mathbb{Z}/n\mathbb{Z}(r-1)) 
}}
est évidemment commutatif. Il suffit donc de montrer la commutativité du diagramme:\\
\centerline{\xymatrix{
H^{r+1}(U_w,\mathbb{Z}/n\mathbb{Z}(r)) \ar[r]^{\text{Res}_{k(v_w)}}\ar[d]^{\delta_{\eta_0}} & H^{r+1}(k(v_w),\mathbb{Z}/n\mathbb{Z}(r))\ar[d]^{\delta_w}\\
H^r(U_{w,0},\mathbb{Z}/n\mathbb{Z}(r-1)) \ar[r]^{e_w\cdot\text{Res}_{\kappa(w)}} & H^r(\kappa(w),\mathbb{Z}/n\mathbb{Z}(r-1)).
}}
Notons $\hat{\mathcal{A}}_w$ le complété de $\mathcal{A}_w$ pour la topologie $\pi$-adique. On remarque que, comme $\mathcal{O}_{\overline{\{v\}},w}$ est complet pour la topologie $\pi$-adique, le morphisme $\text{Spec} \; \mathcal{O}_{\overline{\{v\}},w} \rightarrow \mathcal{U}_w$ s'étend en un morphisme $\text{Spec}  \; \mathcal{O}_{\overline{\{v\}},w} \rightarrow \text{Spec} \; \hat{\mathcal{A}}_w$. En tenant compte de la compatibilité des résidus avec la complétion et en remplaçant $v$ et $w$ par leurs images à travers le morphisme $\text{Spec}  \; \mathcal{O}_{\overline{\{v\}},w} \rightarrow \text{Spec} \; \hat{\mathcal{A}}_w$, on peut supposer que $\mathcal{A}_w$ est complet pour la topologie $\pi$-adique, et donc que le morphisme $H^r(\mathcal{U}_w,\mathbb{Z}/n\mathbb{Z}(r-1)) \rightarrow H^r(U_w,\mathbb{Z}/n\mathbb{Z}(r-1))$ est surjectif.\\
Soit $x \in H^{r+1}(U_w,\mathbb{Z}/n\mathbb{Z}(r))$. Notons $y_{0} = \delta_{\eta_0}(x) \in H^r(U_{w,0},\mathbb{Z}/n\mathbb{Z}(r-1))$. D'après ce qui précède, $y_0$ se relève en un élément $y \in H^r(\mathcal{U}_w,\mathbb{Z}/n\mathbb{Z}(r-1))$. En voyant $y$ dans $H^r(K,\mathbb{Z}/n\mathbb{Z}(r-1))$ et $\pi$ dans $H^1(K,\mu_n) = K^{\times}/{K^{\times}}^n$, on pose $z = y \cup \pi \in H^{r+1}(K,\mathbb{Z}/n\mathbb{Z}(r))$. Pour $v \in U_w^{(1)}$, on remarque que $\delta_v(z) = v(\pi) \text{Res}_{k(v)}(y)=0$ et donc $z \in H^{r+1}(U_w,\mathbb{Z}/n\mathbb{Z}(r))$. De plus, comme $\pi$ est une uniformisante de $\mathcal{O}_{\eta_0}$, on a $\delta_{\eta_0}(z) = y_0$ et donc, d'après le théorème de pureté cohomologique absolue de Gabber (théorème 3.1.1 de l'exposé XVI de \cite{Gab}), $x-z$ provient de $H^{r+1}(\mathcal{U}_w,\mathbb{Z}/n\mathbb{Z}(r))$. Comme le morphisme $H^{r+1}(\mathcal{U}_w,\mathbb{Z}/n\mathbb{Z}(r)) \rightarrow H^{r+1}(k(v_w),\mathbb{Z}/n\mathbb{Z}(r))$ se factorise par $H^{r+1}(\overline{\{v_w\}},\mathbb{Z}/n\mathbb{Z}(r))$, on déduit que $\delta_w(\text{Res}_{k(v_w)}(x-z)) = e_w\text{Res}_{\kappa(w)}(\delta_{\eta_0}(x-z)) = 0$. Il reste donc à montrer que $\delta_w(\text{Res}_{k(v_w)}(z)) = e_w\text{Res}_{\kappa(w)}(\delta_{\eta_0}(z))$, ce qui découle immédiatement des calculs: 
$$\delta_w(\text{Res}_{k(v_w)}(z)) = \delta_w(\text{Res}_{k(v_w)}(y) \cup \text{Res}_{k(v_w)}(\pi)) = e_w\text{Res}_{\kappa(w)}(y)= e_w\text{Res}_{\kappa(w)}(\delta_{\eta_0}(x)).$$
\end{proof}

\begin{corollary}
On rappelle que l'on a supposé (H \ref{40}). On suppose aussi (H \ref{41}) et (H \ref{42}). Soit $r \in \{0,1\}$. On a $e!\cdot\Sha^{r+2}(\mathbb{Z}(r)) \subseteq H^{r+2}(\mathcal{X}, \mathbb{Z}(r))$.
\end{corollary}

\begin{proof}
Soit $x \in \Sha^{r+2}(\mathbb{Z}(r))$. Soit $n >0$ tel que $x$ est de $n$-torsion. Alors $x \in \Sha^{r+1}(\mathbb{Z}/n\mathbb{Z}(r))$. Pour $v \in X^{(1)}$, comme l'image de $x$ dans $H^{r+1}(K_v,\mathbb{Z}/n\mathbb{Z}(r))$ est nulle, l'image de $x$ dans $H^r(k(v),\mathbb{Z}/n\mathbb{Z}(r-1))$ l'est aussi, ce qui prouve que $x$ provient de $\tilde{x} \in H^{r+1}(X,\mathbb{Z}/n\mathbb{Z}(r))$. Étant donné que $H^{r+1}(\mathcal{O}_v,\mathbb{Z}/n\mathbb{Z}(r))$ s'injecte dans $H^{r+1}(K_v,\mathbb{Z}/n\mathbb{Z}(r))$ et que $H^{r+1}(\mathcal{O}_v,\mathbb{Z}/n\mathbb{Z}(r)) \rightarrow H^{r+1}(k(v),\mathbb{Z}/n\mathbb{Z}(r))$ est un isomorphisme, on déduit que  $$\tilde{x} \in \text{Ker} \left( H^{r+1}(X,\mathbb{Z}/n\mathbb{Z}(r)) \rightarrow \prod_{v \in X^{(1)}} H^{r+1}(k(v),\mathbb{Z}/n\mathbb{Z}(r)) \right).$$
Notons $y$ l'image de $\tilde{x}$ dans $H^r(X_0,\mathbb{Z}/n\mathbb{Z}(r-1))$. À l'aide du lemme précédent et de l'hypothèse (H \ref{42})(ii), on déduit que $e!y \in \text{Ker}(H^r(X_0,\mathbb{Z}/n\mathbb{Z}(r-1)) \rightarrow \prod_{w \in X^{(1)}_0} H^r(\kappa(w),\mathbb{Z}/n\mathbb{Z}(r-1)))$. Montrons que $e!y=0$:
\begin{itemize}
\item[$\bullet$] Si $r=0$, alors $H^0(X_0,\mathbb{Z}/n\mathbb{Z}(-1)) \rightarrow \prod_{w \in X^{(1)}_0} H^0(\kappa(w),\mathbb{Z}/n\mathbb{Z}(-1))$ est injectif et donc $e!y=0$.
\item[$\bullet$] Si $r=1$, comme $H^1(\mathcal{O}_{X_0,w},\mathbb{Z}/n\mathbb{Z}) \rightarrow H^1(\kappa(w),\mathbb{Z}/n\mathbb{Z})$ est un isomorphisme et $H^1(\mathcal{O}_{X_0,w},\mathbb{Z}/n\mathbb{Z}) \rightarrow H^1(\kappa(\eta_0)_w,\mathbb{Z}/n\mathbb{Z})$ est injectif, on déduit que $e!y \in \Sha^1(\kappa(\eta_0), \mathbb{Z}/n\mathbb{Z})$. Or $\Sha^1(\kappa(\eta_0), \mathbb{Z}/n\mathbb{Z})= {_n}\Sha^2(\kappa(\eta_0), \mathbb{Z})= 0$ d'après l'hypothèse (H \ref{42})(iii), et donc $e!y=0$. 
\end{itemize}
Par conséquent, $e!\tilde{x}  \in \text{Ker}\left( H^{r+1}(X,\mathbb{Z}/n\mathbb{Z}(r)) \rightarrow H^{r}(X_0,\mathbb{Z}/n\mathbb{Z}(r-1))\right)$. Le théorème de pureté cohomologique absolue de Gabber permet alors de conclure que $e!\tilde{x}$ provient de $H^{r+1}(\mathcal{X},\mathbb{Z}/n\mathbb{Z}(r))$, ce qui prouve que tout élément de $e!\cdot\Sha^{r+2}(\mathbb{Z}(r))$ provient de $H^{r+2}(\mathcal{X}, \mathbb{Z}(r))$. Reste donc à montrer que le morphisme $  H^{r+2}(\mathcal{X}, \mathbb{Z}(r))\rightarrow H^{r+2}(K,\mathbb{Z}(r))$ est injectif:
\begin{itemize}
\item[$\bullet$] Si $r=0$, le morphisme $H^2(\mathcal{X},\mathbb{Z}) \rightarrow H^2(K,\mathbb{Z})$ s'identifie au morphisme $H^1(\mathcal{X},\mathbb{Q}/\mathbb{Z}) \rightarrow H^1(K,\mathbb{Q}/\mathbb{Z})$. Ce dernier est la composée de  $H^1(\mathcal{X},\mathbb{Q}/\mathbb{Z}) \rightarrow H^1(X,\mathbb{Q}/\mathbb{Z})$ suivie de $H^1(X,\mathbb{Q}/\mathbb{Z}) \rightarrow H^1(K,\mathbb{Q}/\mathbb{Z})$, et ces deux morphismes sont injectifs d'après le théorème de pureté cohomologique absolue de Gabber. On en déduit l'injectivité de $H^2(\mathcal{X},\mathbb{Z}) \rightarrow H^2(K,\mathbb{Z})$ .
\item[$\bullet$] Si $r=1$, c'est évident puisque $\mathcal{X}$ est intègre régulier.
\end{itemize}
\end{proof}

\begin{theorem} \label{nul 2}
Rappelons que nous avons supposé (H \ref{40}), c'est-à-dire que $X$ est une courbe. En particulier, $k$ n'est ni un corps fini ni un corps $p$-adique ni $\mathbb{C}((t))$. Soit $r \in \{0,1\}$. On a $\Sha^{r+2}(\mathbb{Z}(r))=0$.
\end{theorem}

\begin{proof}
Soit $x \in e!\cdot\Sha^{r+2}(\mathbb{Z}(r))$. D'après le corollaire précédent, on a $x \in H^{r+2}(\mathcal{X},\mathbb{Z}(r))$. De plus, pour chaque $v \in X^{(1)}$, l'image de $x$ dans $H^{r+2}(k(v),\mathbb{Z}(r))$ est nulle.\\
Soit $w \in X_0^{(1)}$. D'après l'hypothèse (H \ref{42})(ii), la restriction $H^{r+2}(\mathcal{X},\mathbb{Z}(r)) \rightarrow H^{r+2}(\kappa(w),\mathbb{Z}(r))$ se factorise sous la forme $H^{r+2}(\mathcal{X},\mathbb{Z}(r)) \rightarrow H^{r+2}(\overline{\{v_w\}},\mathbb{Z}(r)) \rightarrow H^{r+2}(\kappa(w),\mathbb{Z}(r))$. L'image de $x$ dans $H^{r+2}(k(v_w),\mathbb{Z}(r))$ est nulle. Comme $\overline{\{v_w\}}$ est régulier, la flèche $H^{r+2}(\overline{\{v_w\}},\mathbb{Z}(r)) \rightarrow H^{r+2}(k(v_w),\mathbb{Z}(r))$ est injective et donc l'image de $x$ dans $H^{r+2}(\overline{\{v_w\}},\mathbb{Z}(r))$ est nulle. Par conséquent, l'image de $x$ dans $H^{r+2}(k(w),\mathbb{Z}(r))$ est nulle. Cela impose que l'image de $x$ dans $H^{r+2}(\kappa(\eta_0),\mathbb{Z}(r))$ est en fait dans $\Sha^{r+2}(\kappa(\eta_0),\mathbb{Z}(r))$, qui est nul d'après l'hypothèse (H \ref{42})(iii). On en déduit que $$x \in \text{Ker}\left( H^{r+2}(\mathcal{X},\mathbb{Z}(r)) \rightarrow H^{r+2}(\kappa(\eta_0),\mathbb{Z}(r))\right) .$$
\begin{itemize}
\item[$\bullet$] Si $r=0$, comme $H^{2}(X_0,\mathbb{Z}) \rightarrow H^{2}(\kappa(\eta_0),\mathbb{Z})$ est injectif, $x $ est dans le noyau de $\text{Ker}\left( H^{2}(\mathcal{X},\mathbb{Z}) \rightarrow H^{2}(X_0,\mathbb{Z})\right)$. Ce dernier morphisme est injectif par pureté cohomologique absolue et donc $x=0$. Par conséquent, le groupe $\Sha^{2}(\mathbb{Z})$ est de $e!$-torsion et divisible, donc nul.
\item[$\bullet$] Si $r=1$, comme on a un isomorphisme $\text{Br}(\mathcal{O}_{\mathcal{X},\eta_0}) \cong \text{Br}(\kappa(\eta_0))$ et une injection $\text{Br}(\mathcal{O}_{\mathcal{X},\eta_0}) \rightarrow \text{Br}(K_{\eta_0})$, on déduit que: $$x \in \text{Ker}\left( \text{Br}(K) \rightarrow \prod_{v \in X^{(1)}} \text{Br}(K_v) \times \text{Br}(K_{\eta_0})\right) .$$
Soit $n \geq 1$ tel que $x$ est de $n$-torsion. Alors:
$$x \in \text{Ker}\left( H^2(K,\mu_n) \rightarrow \prod_{v \in X^{(1)}} H^2(K_v,\mu_n) \times H^2(K_{\eta_0},\mu_n)\right) .$$
D'après le théorème 3.3.6 de \cite{HHK}, cela impose que $x=0$. On en déduit que le groupe $\Sha^{3}(\mathbb{Z}(1))$ est de $e!$-torsion, donc nul.
\end{itemize}
\end{proof}

Le théorème précédent nous permet à présent de passer à la récurrence:

\begin{corollary} \textbf{(Cas où $k_1$ est $p$-adique)} \label{nul 3}\\
Rappelons que nous avons supposé (H \ref{40}), c'est-à-dire que $X$ est une courbe. Supposons que $d\geq 1$ et que le corps $k_1$ est $p$-adique. Pour $i \in \{1,2,3,...,d\}$, notons $\mathcal{O}_{k_i}$ l'anneau des entiers de $k_i$ et $\pi_i$ une uniformisante de $\mathcal{O}_{k_i}$. Supposons que, pour chaque $i \in \{1,2,3,...,d \}$, il existe un schéma intègre, projectif, lisse $\mathcal{X}_i$ de dimension 2 sur $\text{Spec} \; \mathcal{O}_{k_i}$ vérifiant les conditions suivantes:
\begin{itemize}
\item[$\bullet$] pour $1 \leq i \leq d$, la fibre générique $X_i$ et la fibre spéciale $X_{i,0}$ de $\mathcal{X}_i$ sont intègres.
\item[$\bullet$] la fibre générique $X_d$ de $\mathcal{X}_d$ est isomorphe à $X$.
\item[$\bullet$] pour $1\leq i \leq d-1$, la fibre générique $X_i$ de $\mathcal{X}_i$ est isomorphe à la fibre spéciale $X_{i+1,0}$ de $\mathcal{X}_{i+1}$.
\item[$\bullet$] la fibre spéciale $X_{1,0}$ de $\mathcal{X}_1$ est géométriquement intègre. 
\item[$\bullet$] il existe un entier naturel $e$ vérifiant la propriété suivante: pour $1 \leq i \leq d$, pour $w \in X_{i,0}^{(1)}$, il existe un ouvert affine $\mathcal{U}_w$ de $\mathcal{X}_i$ contenant $w$ et un point fermé $v_w$ de $U_w=\mathcal{U}_w \times_{\mathcal{O}_{k_i}} k_i$ tels que l'adhérence $\overline{\{v_w\}}$ de $v_w$ dans $\mathcal{U}_w$, munie de sa structure réduite, est régulière, contient $w$ et $\pi_i$ est de valuation au plus $e$ dans l'anneau de valuation discrète $\mathcal{O}_{\overline{\{v_w\}},w}$.
\end{itemize}
Alors $\Sha^{2}(\mathbb{Z}) = \Sha^{2}(\mathbb{G}_m)=0$.
\end{corollary}

\begin{proof}
Par récurrence, il suffit de montrer que, si $K_1$ est le corps des fonctions de $\mathcal{X}_1$, alors $\Sha^{2}(K_1,\mathbb{G}_m)=\Sha^{2}(K_1,\mathbb{Z})=0$. La nullité de $\Sha^{2}(K_1,\mathbb{G}_m)$ est prouvée dans la proposition 3.4 de \cite{HS1}. Il reste donc à vérifier que $\Sha^2(K_1,\mathbb{Z})$ est nul, ou, ce qui revient au même, vérifier que $\Sha^1(K_1,\mathbb{Z}/n\mathbb{Z})$ est nul pour tout $n>0$. Par dualité, cela équivaut à montrer que $\Sha^3(K_1,\mathbb{Z}/n\mathbb{Z}(2))$ est nul pour tout $n>0$, ou, ce qui revient au même, montrer que $\Sha^3(K_1,\mathbb{Q}/\mathbb{Z}(2))$ est nul. Mais la fibre spéciale de $\mathcal{X}_1$ étant géométriquement intègre, si l'on note $K_0$ son corps des fonctions, la proposition 5.2 de \cite{Kat} impose que le groupe $\Sha^3(K_1,\mathbb{Q}/\mathbb{Z}(2))$ est isomorphe au groupe  $\Sha^2(K_0,\mathbb{Q}/\mathbb{Z}(1))$, qui est nul d'après le théorème de Brauer-Hasse-Noether. Cela achève la preuve.
\end{proof}

\begin{corollary} \textbf{(Cas où $k_0$ est $\mathbb{C}((t))$)} \label{nul 4}\\
Rappelons que nous avons supposé (H \ref{40}), c'est-à-dire que $X$ est une courbe. Supposons que $d \geq 0$ et que $k_0 = \mathbb{C}((t))$. Pour $i \in \{1,2,3,...,d\}$, notons $\mathcal{O}_{k_i}$ l'anneau des entiers de $k_i$ et $\pi_i$ une uniformisante de $\mathcal{O}_{k_i}$. Supposons que, pour chaque $i \in \{1,2,...,d \}$, il existe un schéma intègre, projectif, lisse $\mathcal{X}_i$ de dimension 2 sur $\text{Spec} \; \mathcal{O}_{k_i}$ vérifiant les conditions suivantes:
\begin{itemize}
\item[$\bullet$] pour $1 \leq i \leq d$, la fibre générique $X_i$ et la fibre spéciale $X_{i,0}$ de $\mathcal{X}_i$ sont intègres.
\item[$\bullet$] la fibre générique $X_d$ de $\mathcal{X}_d$ est isomorphe à $X$.
\item[$\bullet$] pour $1\leq i \leq d-1$, la fibre générique $X_i$ de $\mathcal{X}_i$ est isomorphe à la fibre spéciale $X_{i+1,0}$ de $\mathcal{X}_{i+1}$.
\item[$\bullet$] la jacobienne de la fibre spéciale $X_{1,0}$ a très mauvaise réduction.
\item[$\bullet$] il existe un entier naturel $e$ vérifiant la propriété suivante: pour $1 \leq i \leq d$, pour $w \in X_{i,0}^{(1)}$, il existe un ouvert affine $\mathcal{U}_w$ de $\mathcal{X}_i$ contenant $w$ et un point fermé $v_w$ de $U_w=\mathcal{U}_w \times_{\mathcal{O}_{k_i}} k_i$ tels que l'adhérence $\overline{\{v_w\}}$ de $v_w$ dans $\mathcal{U}_w$, munie de sa structure réduite, est régulière, contient $w$ et $\pi_i$ est de valuation au plus $e$ dans l'anneau de valuation discrète $\mathcal{O}_{\overline{\{v_w\}},w}$.
\end{itemize}
Alors $\Sha^{2}(\mathbb{Z})=\Sha^{2}(\mathbb{G}_m)=0$.
\end{corollary}

\begin{proof}
Par récurrence, il suffit de montrer que, si $K_0$ est le corps des fonctions de $X_{1,0}$, alors $\Sha^{2}(K_0,\mathbb{G}_m)=\Sha^{2}(K_0,\mathbb{Z})=0$. La nullité de $\Sha^{2}(K_0,\mathbb{G}_m)$ provient de l'article \cite{Dou} (on remarquera que ce résultat reste vrai même si l'errata \cite{DouErrata} montre que le théorème 1 de \cite{Dou} est faux). D'après le lemme \ref{nul dual}, cela entraîne aussi que $\Sha^2(K_0,\mathbb{Z})$ est nul, ce qui achève la preuve.
\end{proof}

Je ne sais pas à quel point la dernière hypothèse des corollaires précédents est contraignante, mais elle permet au moins de traiter les exemples qui suivent.

\begin{theorem} \textbf{(Courbes constantes)} \label{constante}\\
Rappelons que nous avons supposé (H \ref{40}), c'est-à-dire que $X$ est une courbe.
\begin{itemize}
\item[(i)] Supposons que $d \geq 1$, que $k_1$ est un corps $p$-adique et que $X$ est une courbe de la forme $\text{Proj}(k[x,y,z]/(P(x,y,z))) $ où $P \in \mathcal{O}_{k_1}[x,y,z]$ est un polynôme homogène. Supposons aussi que $\text{Proj}(k_0[x,y,z]/(\overline{P}(x,y,z))) $ est une courbe lisse et géométriquement intègre. Alors $\Sha^{2}(\mathbb{Z})=\Sha^2(\mathbb{G}_m)=0$.
\item[(ii)] Supposons que $d\geq 0$, que $k_0 = \mathbb{C}((t))$ et que $X$ est la courbe elliptique sur $k$ d'équation $y^2=x^3+Ax+B$ avec $A, B \in k_0$. Supposons de plus que la courbe elliptique sur $k_0$ d'équation $y^2=x^3+Ax+B$ admet une réduction modulo $t$ de type additif. Alors $\Sha^{2}(\mathbb{Z})=\Sha^2(\mathbb{G}_m)=0$.
\end{itemize}
\end{theorem}

\begin{proof}
\begin{itemize}
\item[(i)] Pour $i \in \{1,2,...,d\}$, soit $\mathcal{X}_i = \text{Proj}(\mathcal{O}_{k_i}[x,y,z]/(P(x,y,z)))$. Pour chaque $i$, le schéma $\mathcal{X}_i$ étant dominant sur $\mathcal{O}_{k_i}$, il est plat. De plus, le critère jacobien permet de vérifier immédiatement que $X_{i,0}$ est lisse sur le corps résiduel de $k_i$. Par conséquent, pour chaque $i$, $\mathcal{X}_i$ est lisse sur $\mathcal{O}_{k_i}$.\\
Toutes les hypothèses du corollaire \ref{nul 3} sont évidemment vérifiées sauf peut-être la dernière. Fixons donc un certain $i \in \{1,2,...,d\}$ et soit $w$ un point fermé de la fibre spéciale $X_{i,0}$. Supposons sans perte de généralité que $w$ est dans l'ouvert $\text{Spec}(\mathcal{O}_{k_i}[x,y]/(P(x,y,1)))$ de $\mathcal{X}_i$, et choisissons $\mathcal{U}_w = \text{Spec}(\mathcal{O}_{k_i}[x,y]/(P(x,y,1)))$. Le point $w$ est alors le noyau d'un morphisme $\text{ev}_w: k_{i-1}[x,y]/(P(x,y,1)) \rightarrow k_{i-1}^s$. Notons $b$ et $c$ les images respectives de $x$ et $y$ dans $k_{i-1}^s$. Soient $\lambda$ une extension finie de $k_{i-1}$ contenant $b$ et $c$ et $l = \lambda((\pi_i))$. Le noyau $v$ du morphisme $k_i[x,y]/(P(x,y,1)) \rightarrow k_i^s$ qui envoie $x$ et $y$ sur $b$ et $c$ respectivement est un point fermé de la fibre générique $U_w$ de $\mathcal{U}_w$ qui contient $w$ dans son adhérence. On vérifie aisément que l'idéal premier définissant $w$ dans $\mathcal{U}_w$ est l'idéal engendré par l'idéal premier définissant $v$ dans $\mathcal{U}_w$ et par $\pi_i$. Par conséquent, l'anneau des fonctions de $\overline{\{v\}}$ est un anneau intègre ayant un unique idéal premier non nul, engendré par $\pi_i$: c'est donc un anneau de valuation discrète d'uniformisante $\pi_i$, ce qui prouve la dernière hypothèse du corollaire \ref{nul 3} avec $e=1$.
\item[(ii)] La démonstration est analogue à celle qui précède en remarquant qu'une courbe elliptique ayant réduction de type additif a très mauvaise réduction.
\end{itemize}
\end{proof}

\begin{example}
La partie (ii) précédente s'applique par exemple à la courbe elliptique $y^2=x^3+t$ sur $k = \mathbb{C}((t))((\pi_1))...((\pi_d))$ pour $d \geq 0$.
\end{example}

Nous fournissons maintenant des contre-exemples à $\Sha^2(\mathbb{G}_m)=0$.

\begin{example}\label{ex0}
Soit $p$ un nombre premier impair, de sorte que $1-p$ soit un carré dans $\mathbb{Q}_p$. Prenons $k= \mathbb{Q}_p((t))$ et considérons $X$ la courbe elliptique d'équation $y^2=x(1-x)(x-p)$. Nous allons montrer que, dans ce contexte, $\Sha^2(\mathbb{G}_m)\neq 0$.
\begin{itemize}
\item[$\bullet$] D'après l'appendice de \cite{CPS}, dans $K_{\eta_0}$, $1-x$ est de valuation nulle mais n'est pas un carré.
\item[$\bullet$] Soit $v \in X^{(1)}$. Soit $\pi$ une uniformisante de $\mathcal{O}_v$. Montrons que $1-x$ est un carré dans $K_v$. En suivant les idées de \cite{CPS}, plusieurs cas se présentent.
\begin{itemize}
\item[(1)] Supposons que $v(1-x) <0$. Alors $v(x)$, $v(1-x)$ et $v(x-p)$ sont égaux et pairs. Écrivons $x=u/\pi^{2n}$ avec $u \in \mathcal{O}_v^{\times}$ et $n>0$. L'équation de $X$ impose que $-u$ est un carré dans $\mathcal{O}_v$. Par conséquent, $1-x = \frac{\pi^{2n}-u}{\pi^{2n}}$ est un carré dans $K_v$.
\item[(2)] Supposons que $v(1-x)>0$. Alors $v(x) = v(x-p)=0$ et $v(1-x)$ est pair. En écrivant $1-x = u\pi^{2n}$, on remarque que $x = 1-u\pi^{2n}$ et $x-p = 1-p-u\pi^{2n}$ sont des carrés dans $K_v$. Il en est donc de même pour $1-x$.
\item[(3)] Supposons que $v(1-x)=0$. Si $v(x)>0$, alors $1-x$ est bien sûr un carré dans $K_v$. Si $v(x-p)>0$, alors $1-x = 1-p-(x-p)$ est aussi un carré. Il reste donc à étudier le cas $v(x)=v(x-p)=0$. Dans ce cas, les réductions de $x$ et $y$ modulo l'idéal maximal de $\mathcal{O}_v$ donnent lieu à des éléments $b$ et $c$ dans $k(v)$ tels que $b^2 = c(1-c)(c-p) \neq 0$. Pour montrer que $1-x$ est un carré dans $K_v$, il suffit de montrer que $1-c$ est un carré dans $k(v)$. Notons $w$ la valuation de $k(v)$, $B$ son anneau des entiers, $k_B$ son corps résiduel et $\pi_B$ une uniformisante de $B$. Plusieurs cas se présentent alors.
\begin{itemize}
\item[(a)] Supposons que $w(1-c) <0$. Alors $w(c)$, $w(1-c)$ et $w(c-p)$ sont égaux et pairs. Écrivons $c=u/\pi_B^{2n}$ avec $u \in B^{\times}$ et $n>0$. L'équation vérifiée par $b$ et $c$ impose que $-u$ est un carré dans $B$. Par conséquent, $1-c = \frac{\pi_B^{2n}-u}{\pi_B^{2n}}$ est un carré dans $k(v)$.
\item[(b)] Supposons que $w(1-c)>0$. Alors $w(c) = w(c-p)=0$ et $w(1-c)$ est pair. En écrivant $1-c = u\pi_B^{2n}$, on remarque que $c = 1-u\pi_B^{2n}$ et $c-p = 1-p-u\pi_B^{2n}$ sont des carrés dans $k(v)$. Il en est donc de même pour $1-c$.
\item[(c)] Supposons que $w(1-c)=0$. Si $w(c)>0$, alors $1-c$ est bien sûr un carré dans $k(v)$. Si $w(c-p)>0$, alors $1-c = 1-p-(c-p)$ est aussi un carré. Il reste donc à étudier le cas $w(c)=w(c-p)=0$. Dans ce cas, les réductions de $b$ et $c$ modulo l'idéal maximal de $B$ donnent lieu à des éléments $\beta$ et $\gamma$ dans $k_B$ tels que $\beta^2 = \gamma(1-\gamma)(\gamma-p) \neq 0$. Pour montrer que $1-c$ est un carré dans $k(v)$, il suffit de montrer que $1-\gamma$ est un carré dans le corps $p$-adique $k_B$. Mais cela est montré dans l'appendice de \cite{CPS}. 
\end{itemize}
\end{itemize}
Par conséquent, $1-x \in \Sha^1(\mathbb{Z}/2\mathbb{Z})$.
\item[$\bullet$] Soit $z = (1-x) \cup t \in H^2(K,\mathbb{Z}/2\mathbb{Z})$. Pour chaque $v \in X^{(1)}$, $1-x$ est un carré dans $K_v$ et donc $z \in \Sha^2(\mathbb{Z}/2\mathbb{Z})$. De plus, le résidu de $z$ en $\eta_0$ est $\text{Res}(1-x) \in \kappa(\eta_0)^{\times}/{\kappa(\eta_0)^{\times}}^2$. Ce résidu ne peut pas être nul puisque $1-x$ n'est pas un carré dans $K_{\eta_0}$. Par conséquent, $z \neq 0$ et ${_2}\Sha^2(\mathbb{G}_m) = \Sha^2(\mathbb{Z}/2\mathbb{Z})\neq 0$.
\end{itemize}
Plus généralement, si $k= \mathbb{Q}_p((t_1))...((t_{d-1}))$ avec $d>1$, la courbe elliptique $X$ d'équation $y^2=x(1-x)(x-\lambda)$ où $\lambda \in \{p,t_1,...,t_{d-2}\}$ vérifie $\Sha^2(\mathbb{G}_m)\neq 0$. Bien sûr, on peut remplacer $\mathbb{Q}_p$ par n'importe quel corps $p$-adique.
\end{example}

\begin{remarque}
Dans le cas où $k$ est un corps $p$-adique, on a $\Sha^2(\mathbb{G}_m)=0$ quelle que soit la courbe $X$, même si elle a mauvaise réduction (voir la proposition 3.4 de \cite{HS1}).
\end{remarque}

\begin{example}\label{ex}
\begin{itemize}
\item[$\bullet$] Pour $k=\mathbb{C}((t))$ et $X$ une courbe elliptique ayant bonne réduction (resp. mauvaise réduction de type multiplicatif), on a $\Sha^2(\mathbb{G}_m)\neq 0$: en effet, d'après \cite{Dou}, si $J$ désigne la jacobienne de $X$, le groupe $\Sha^2(\mathbb{G}_m)\cong H^1(k,\text{Pic}\; \overline{X})$ s'identifie au conoyau de $\mathbb{Z} \rightarrow H^1(k,J)$ et $H^1(k,J)$ est isomorphe à $(\mathbb{Q}/\mathbb{Z})^2$ (resp. à $\mathbb{Q}/\mathbb{Z}$). Ces résultats de \cite{Dou} restent vrais malgré l'errata \cite{DouErrata}.
\item[$\bullet$] Pour $k=\mathbb{C}((t_1))...((t_{d+1}))$ avec $d>0$, un raisonnement analogue à l'exemple précédent montre que la courbe elliptique $X$ d'équation $y^2=x(1-x)(x-\lambda)$ avec $\lambda \in \{ t_1,...,t_d\}$ vérifie $\Sha^2(\mathbb{G}_m)\neq 0$.
\end{itemize}
\end{example}

\begin{remarque}
Le groupe $\Sha^2(\mathbb{G}_m)$ est-il une puissance de $\mathbb{Q}/\mathbb{Z}$? Je ne sais pas répondre en toute généralité à cette question de Jean-Louis Colliot-Thélène, mais la réponse est affirmative lorsque $k$ est un corps $p$-adique ou $\mathbb{C}((t))$:
\begin{itemize}
\item[$\bullet$] lorsque $k$ est un corps $p$-adique, on a toujours $\Sha^2(\mathbb{G}_m)=0$ d'après \cite{HS1};
\item[$\bullet$] lorsque $k$ est le corps $\mathbb{C}((t))$, les résultats de \cite{Dou} montrent que $\Sha^2(\mathbb{G}_m)$ est le conoyau de $\mathbb{Z} \rightarrow H^1(k,J)$ et que $H^1(k,J)$ est isomorphe à $(\mathbb{Q}/\mathbb{Z})^r$ pour un certain $r \in \{ 0,1,...,2g\}$ où $g$ est le genre de $X$, ce qui impose que $\Sha^2(\mathbb{G}_m) \cong (\mathbb{Q}/\mathbb{Z})^r$.
\end{itemize}
\end{remarque}

\subsection{Finitude du $(d+3)$-ième groupe de Tate-Shafarevich du dual d'un tore}\label{1.2}

Soit $T$ un tore sur $K$, déployé par une extension finie $L$. Notons $\hat{T}$ (resp. $\check{T}$) son module des caractères (resp. cocaractères), et posons $\tilde{T} = \hat{T} \otimes^{\mathbf{L}} \mathbb{Z}(d)$.\\
En tenant compte du théorème \ref{0}(ii), il est intéressant d'étudier le groupe $\Sha^{d+3}(\tilde{T})$: c'est le but de cette section.\\

\begin{proposition} Soit $Y$ une courbe projective lisse géométriquement intègre sur une extension finie $l$ de $k$ de corps de fonctions $L$.
\begin{itemize}
\item[$(i)$] Le groupe $\Sha^{d+3}(\tilde{T})$ est de torsion de type cofini.
\item[$(ii)$] On a un isomorphisme $\Sha^{d+3}(\tilde{T}) \cong (\varprojlim_n \Sha^1({_n}T))^D$. En particulier, $\Sha^{d+3}(\mathbb{Z}(d)) \cong (\varprojlim_n \Sha^1(\mu_n))^D$, et ce groupe est nul pour $X = \mathbb{P}^1_k$.
\item[$(iii)$] Supposons que $k_1$ soit un corps $p$-adique. Alors $\Sha^{d+3}(\tilde{T})$ est divisible. Il est nul dès que $\varprojlim_n \Sha^1(L,\mu_n)=0$ ou dès que $H^{d+1}(l,H^2(\overline{l}(Y),\mathbb{Z}(d)))=0$. Cette deuxième condition est automatiquement satisfaite lorque $d=1$.
\item[$(iv)$] Supposons que $k_0 = \mathbb{C}((t))$. Alors $\Sha^{d+3}(\tilde{T})$ est fini dès que $\Sha^{2}(L,\mathbb{Z})=0$ ou dès que $H^{d+1}(l,H^2(\overline{l}(Y),\mathbb{Z}(d)))=0$. Cette deuxième condition est automatiquement satisfaite lorque $d=1$.
\end{itemize}
\end{proposition}

\begin{proof}
\begin{itemize}
\item[(i)] Cela découle immédiatement du corollaire \ref{tcof gm}, du lemme \ref{iso} et de la proposition \ref{Sha gm}.
\item[(ii)] D'après la preuve de la proposition \ref{bouts tore}, on a un isomorphisme $\Sha^{d+3}(\tilde{T}) \cong \varinjlim_n \Sha^{d+2}(\hat{T} \otimes \mathbb{Z}/n\mathbb{Z}(d))$. Il suffit alors d'appliquer le théorème \ref{PT fini} pour établir $\Sha^{d+3}(\tilde{T}) \cong (\varprojlim_n \Sha^1({_n}T))^D$. Les autres affirmations en découlent aisément.
\item[(iii)] Le groupe $\Sha^{d+3}(\tilde{T})$ est divisible d'après la remarque \ref{rq gm}. Il est donc nul dès que l'une des conditions suivantes est vérifiéé:
\begin{itemize}
\item[$\bullet$] $\varprojlim_n \Sha^1(L,\mu_n)=0$ (d'après (ii)). 
\item[$\bullet$] $\Sha^{d+3}(L,\mathbb{Z}(d)) =0$ (par restriction-corestriction).
\end{itemize}
La deuxième condition est bien sûr satisfaite dès que $ H^{d+3}(L,\mathbb{Z}(d))=0$. Reste donc à montrer que $H^{d+1}(l,H^2(\overline{l}(Y),\mathbb{Z}(d))) \cong H^{d+3}(L,\mathbb{Z}(d))$.\\
Soit $U$ un ouvert affine de $Y$ et écrivons la suite spectrale:
$$H^r(l,H^s(\overline{U},\mathbb{Z}(d))) \Rightarrow H^{r+s}(U,\mathbb{Z}(d)).$$
On remarque que:
\begin{itemize}
\item[$\bullet$] comme $\text{scd}(l) = d+1$, on a $H^r(l,H^s(\overline{U},\mathbb{Z}(d)))=0$ dès que $r \geq d+2$. 
\item[$\bullet$] comme $\text{cd}(\overline{U}) \leq 1$, le groupe $H^s(\overline{U},\mathbb{Z}(d))$ est uniquement divisible pour $s > 2$, et donc $H^r(l,H^s(\overline{U},\mathbb{Z}(d)))=0$ dès que $s >2$ et $r>0$.
\item[$\bullet$] pour $s \geq d+2$, le groupe $H^{s}(\overline{U},\mathbb{Z}(d))$ est uniquement divisible, et, comme $H^s(\overline{U},\mathbb{Q}(d))=0$, on a une surjection $H^{s-1}(\overline{U},\mathbb{Q}/\mathbb{Z}(d)) \rightarrow H^s(\overline{U},\mathbb{Z}(d))$; on en déduit que le groupe $H^s(\overline{U},\mathbb{Z}(d))$ est de torsion, donc nul.
\end{itemize}
Par conséquent, la suite spectrale fournit un isomorphisme $H^{d+1}(l,H^2(\overline{U},\mathbb{Z}(d))) \cong H^{d+3}(U,\mathbb{Z}(d))$. En prenant la limite inductive sur les ouverts affines $U$ de $X$, on obtient un isomorphisme:
$$H^{d+1}(l,H^2(\overline{l}(X),\mathbb{Z}(d))) \cong H^{d+3}(K,\mathbb{Z}(d)).$$
Lorsque $d=1$, ces groupes sont nuls d'après me théorème de Hilbert 90.
\item[(iv)] Le cas où $H^{d+1}(l,H^2(\overline{l}(Y),\mathbb{Z}(d)))=0$ se démontre de la même manière que dans (iii). On se place donc dans le cas où $\Sha^{2}(L,\mathbb{Z})=0$. Comme $k_0 = \mathbb{C}((t))$, on a $\varprojlim_n \Sha^1(L,\mu_n) \cong \varprojlim_n \Sha^1(L,\mathbb{Z}/n\mathbb{Z})= \varprojlim_n {_n}\Sha^2(L,\mathbb{Z})=0$. Un argument de restriction-corestriction permet alors de conclure.
\end{itemize}
\end{proof}

\begin{remarque}
La nullité de $\Sha^2(L,\mathbb{Z})$ a été étudiée au paragraphe \ref{1.1}.
\end{remarque}

\section{\scshape Approximation faible pour les tores}

Soit $T$ un tore sur $K$. Notons $\hat{T}$ (resp. $\check{T}$) son module des caractères (resp. cocaractères), et posons $\tilde{T} = \hat{T} \otimes^{\mathbf{L}} \mathbb{Z}(d)$. Nous voulons ici étudier les propriétés d'approximation faible pour le tore $T$:

\begin{definition} \label{def}
On dit que $T$ vérifie \textbf{l'approximation faible} si $T(K)$ est dense dans $\prod_{v \in X^{(1)}} T(K_v)$, où le groupe $T(K_v)$ est muni de la topologie $v$-adique pour chaque $v$. On dit que $T$ vérifie \textbf{l'approximation faible faible} (resp. \textbf{l'approximation faible faible dénombrable}) s'il existe une partie finie (resp. dénombrable) $S_0$ de $X^{(1)}$ telle que, pour toute partie finie $S$ de $X^{(1)}$ n'intersectant pas $S_0$, le groupe $T(K)$ est dense dans $\prod_{v \in S} T(K_v)$. 
\end{definition}

Nous allons voir que ces propriétés se lisent dans la ``taille'' du groupe $\Sha^{d+2}_{\omega}(\tilde{T})$: plus le groupe $\Sha^{d+2}_{\omega}(\tilde{T})$ est ``gros'', plus on s'éloigne de la propriété d'approximation faible. Pour ce faire, rappelons que l'approximation faible est vérifiée pour les tores quasi-triviaux. En effet, ces derniers étant lisses et $K$-rationnels, cela découle du lemme d'approximation d'Artin-Whaples (théorème XII.1.2 de \cite{LanAlg}) et du théorème des fonctions implicites pour les valuations ultramétriques.\\

Fixons maintenant une partie finie $S$ de $X^{(1)}$.

\begin{lemma}
Soit $F$ un $\text{Gal}(K^s/K)$-module discret fini. Pour $r \in \{1,2,...,d+1\}$, on a une suite exacte:
$$H^r(K,F) \rightarrow \prod_{v \in S} H^r(K_v,F) \rightarrow \Sha^{d+2-r}_S(F')^D \rightarrow \Sha^{d+2-r}(F')^D \rightarrow 0,$$
où le morphisme $\prod_{v \in S} H^r(K_v,F) \rightarrow \Sha^{d+2-r}_S(F')^D$ est donné par $$(f_v) \mapsto (f' \mapsto \sum_{v \in S} (f_v,f'_v)_v).$$
\end{lemma}

\begin{proof}
D'après la proposition \ref{bouts suite fini}, on a une suite exacte:\\
\centerline{\xymatrix{H^{d+2-r}(K,F') \ar[r]& \mathbb{P}^{d+2-r}(F') \ar[r] & H^{r}(K,F)^D}.}
On en déduit la suite exacte:\\
\centerline{\xymatrix{\Sha^{d+2-r}_S(F') \ar[r]& \prod_{v\in S} H^{d+2-r}(K_v,F') \ar[r] & H^{r}(K,F)^D}.}
En dualisant, on obtient l'exactitude de:
$$H^r(K,F) \rightarrow \prod_{v \in S} H^r(K_v,F) \rightarrow \Sha^{d+2-r}_S(F')^D.$$
Pour achever la preuve, il suffit de dualiser la suite:
$$0 \rightarrow \Sha^{d+2-r}(F') \rightarrow \Sha^{d+2-r}_S(F') \rightarrow \bigoplus_{v \in S} H^{d+2-r}(K_v,F) .$$
\end{proof}

\begin{lemma}
Le groupe $\Sha^{d+2}_S(\mathbb{Z}(d))$ coïncide avec le groupe $\Sha^{d+2}(\mathbb{Z}(d))$ et est donc de torsion de type cofini divisible.
\end{lemma}

\begin{proof}
Soit $n>0$. D'après la proposition précédente, on a une suite exacte:
$$H^1(K,\mu_n) \rightarrow \prod_{v \in S} H^1(K_v,\mu_n) \rightarrow \Sha^{d+1}_S(\mu_n^{\otimes d})^D \rightarrow \Sha^{d+1}(\mu_n^{\otimes d})^D \rightarrow 0.$$
Or le morphisme $H^1(K,\mu_n) \rightarrow \prod_{v \in S} H^1(K_v,\mu_n)$ est surjectif puisque $\mathbb{G}_m$ vérifie l'approximation faible. Donc $\Sha^{d+1}_S(\mu_n^{\otimes d}) = \Sha^{d+1}(\mu_n^{\otimes d})$. Comme la conjecture de Beilinson-Lichtenbaum impose que $\Sha^{d+1}(\mu_n^{\otimes d}) = {_n}\Sha^{d+2}(\mathbb{Z}(d))$ et que $\Sha^{d+1}_S(\mu_n^{\otimes d}) = {_n}\Sha^{d+2}(\mathbb{Z}(d))$, cela prouve que $\Sha^{d+2}_S(\mathbb{Z}(d)) = \Sha^{d+2}(\mathbb{Z}(d))$.
\end{proof}

\begin{theorem} \label{AF tore}
Rappelons que $T$ est un tore et $\tilde{T} = \hat{T} \otimes^{\mathbf{L}} \mathbb{Z}(d)$.
\begin{itemize}
\item[(i)] On a la suite exacte: $$0 \rightarrow \overline{T(K)}_S \rightarrow \prod_{v \in S} T(K_v) \rightarrow (\Sha^{d+2}_S(\tilde{T}))^D \rightarrow (\Sha^{d+2}(\tilde{T}))^D \rightarrow 0,$$
où $\overline{T(K)}_S$ désigne l'adhérence de $T(K)$ dans $\prod_{v \in S} T(K_v)$.
\item[(ii)] On a la suite exacte:
$$0 \rightarrow \overline{T(K)} \rightarrow \prod_{v \in X^{(1)}} T(K_v) \rightarrow (\Sha^{d+2}_{\omega}(\tilde{T}))^D \rightarrow (\Sha^{d+2}(\tilde{T}))^D \rightarrow 0,$$
où $ \overline{T(K)}$ désigne l'adhérence de $T(K)$ dans $\prod_{v \in X^{(1)}} T(K_v)$.
\end{itemize}
\end{theorem}

\begin{proof}
\begin{itemize}
\item[(i)] Remarquons que, si $T$ est quasi-trivial:
\begin{itemize}
\item[$\bullet$] la flèche $\overline{T(K)}_S \rightarrow \prod_{v \in S} T(K_v)$ est surjective puisque $T$ vérifie l'approximation faible,
\item[$\bullet$] le morphisme $\prod_{v \in S} T(K_v) \rightarrow (\Sha^{d+2}_S(\tilde{T}))^D$ est nul d'après le lemme précédent et le lemme de Shapiro.
\item[$\bullet$] le morphisme $(\Sha^{d+2}_S(\tilde{T}))^D \rightarrow (\Sha^{d+2}(\tilde{T}))^D$ est un isomorphisme d'après le lemme précédent et le lemme de Shapiro.
\end{itemize}
On en déduit que la suite $$0 \rightarrow \overline{T(K)}_S \rightarrow \prod_{v \in S} T(K_v) \rightarrow (\Sha^{d+2}_S(\tilde{T}))^D \rightarrow (\Sha^{d+2}(\tilde{T}))^D \rightarrow 0$$
est exacte dès que $T$ est quasi-trivial. Par conséquent, pour prouver la proposition, nous pouvons remplacer $T$ par un tore de la forme $T^m \times_K T_0$, où $m>0$ et $T_0$ est un tore quasi-trivial.\\
D'après le lemme 1.10 de \cite{San}, on peut donc supposer qu'il existe une suite exacte:
$$0 \rightarrow F \rightarrow R \rightarrow T \rightarrow 0$$
où $F$ est un schéma en groupes fini commutatif sur $K$ et $R$ un tore quasi-trivial sur $K$. On en déduit une suite exacte de $\text{Gal}(K^s/K)$-modules:
$$0 \rightarrow \hat{T} \rightarrow \hat{R} \rightarrow \hat{F} \rightarrow 0.$$
En tensorisant par $\mathbb{Z}(d)$, on obtient une suite exacte de complexes:
$$0 \rightarrow \tilde{T} \rightarrow \tilde{R} \rightarrow F' \rightarrow 0.$$
D'après le lemme de Shapiro et la conjecture de Beilinson-Lichtenbaum, on remarque que $H^{d+1}(K,\tilde{R})=0$ et $H^{d+1}(K_v,\tilde{R})=0$ pour $v \in X^{(1)}$. On en déduit une suite exacte:
$$0 \rightarrow \Sha^{d+1}_S(F') \rightarrow \Sha^{d+2}_S(\tilde{T}) \rightarrow \Sha^{d+2}_S(\tilde{R}).$$
Montrons que le morphisme $ \Sha^{d+2}_S(\tilde{T}) \rightarrow \Sha^{d+2}_S(\tilde{R})$ est surjectif. Soit $n$ l'ordre de $F$. Soit $x \in \Sha^{d+2}_S(\tilde{R})$. Comme $\Sha^{d+2}_S(\tilde{R})$ est divisible et $H^{d+2}(K,F')$ est d'exposant fini, il existe $x' \in \Sha^{d+2}_S(\tilde{R})$ tel que $x = nx'$ et l'image de $x'$ dans $H^{d+2}(K,F')$ est nulle. Par conséquent, $x'$ provient d'un élément $y\in H^{d+2}(K,\tilde{T})$. Soit $z$ l'image de $y$ dans $\prod_{v \not\in S} H^{d+2}(K_v,\tilde{T})$. Son image dans $\prod_{v \not\in S} H^{d+2}(K_v,\tilde{R})$ est bien sûr nulle, et donc $z$ provient d'un élément de $\prod_{v \not\in S} H^{d+1}(K_v,F')$, qui est de $n$-torsion. Donc $nz=0$, et $ny\in \Sha^{d+2}_S(\tilde{T})$. Comme $ny$ s'envoie sur $x$ dans $\Sha^{d+2}_S(\tilde{R})$, on déduit que le morphisme $ \Sha^{d+2}_S(\tilde{T}) \rightarrow \Sha^{d+2}_S(\tilde{R})$ est surjectif.\\
Nous disposons donc d'une suite exacte:
$$0 \rightarrow \Sha^{d+1}_S(F') \rightarrow \Sha^{d+2}_S(\tilde{T}) \rightarrow \Sha^{d+2}_S(\tilde{R}) \rightarrow 0.$$
En dualisant, on obtient la suite exacte:
$$0 \rightarrow \Sha^{d+2}_S(\tilde{R})^D \rightarrow \Sha^{d+2}_S(\tilde{T})^D \rightarrow \Sha^{d+1}_S(F')^D \rightarrow 0.$$
On a donc un diagramme (D) commutatif à lignes exactes:\\
\centerline{\xymatrix{
&R(K)\ar[d] \ar[r] & T(K)\ar[d] \ar[r] & H^1(K,F)\ar[d] \ar[r] & 0\\
&\prod_{v \in S} R(K_v)\ar[d] \ar[r] & \prod_{v \in S} T(K_v)\ar[d] \ar[r] & \prod_{v \in S} H^1(K_v,F)\ar[d] \ar[r] & 0\\
0 \ar[r] & \Sha^{d+2}_S(\tilde{R})^D \ar[r] & \Sha^{d+2}_S(\tilde{T})^D \ar[r] & \Sha^{d+1}_S(F')^D \ar[r] & 0
}}
De plus, la colonne concernant le module fini $F$ est exacte. Comme $R$ vérifie l'approximation faible, une simple chasse au diagramme montre que $\text{Ker}(\prod_{v \in S} T(K_v) \rightarrow \Sha^{d+2}_S(\tilde{T}))$ est contenu dans l'adhérence de $T(K)$ dans $\prod_{v \in S} T(K_v)$. Montrons que $0 \rightarrow T(K) \rightarrow \prod_{v \in S} T(K_v) \rightarrow \Sha^{d+2}_S(\tilde{T})^D$ est un complexe.\\
Soient $x \in T(K)$, $(x_v)_{v\in S}$ son image dans $\prod_{v \in S} T(K_v)$ et $y$ son image dans $\Sha^{d+2}_S(\tilde{T})^D$. On remarque immédiatement que l'image de $y$ dans $\Sha^{d+1}_S(F')^D$ est nulle, et donc $y$ provient d'un élément $z \in \Sha^{d+2}_S(\tilde{R})^D$. Comme $F$ est d'ordre $n$, la famille $(nx_v)_{v \in S}$ provient d'un élément $(r_v)_{v \in S} \in \prod_{v \in S} R(K_v)$. Par injectivité de $\Sha^{d+2}_S(\tilde{R})^D \rightarrow \Sha^{d+2}_S(\tilde{T})^D$ et nullité de $\prod_{v \in S} R(K_v) \rightarrow \Sha^{d+2}_S(\tilde{R})^D$, on déduit que $nz=0$. Mais $\Sha^{d+2}_S(\tilde{R})$ étant de torsion de type cofini divisible, son dual est sans torsion. Donc $z=0$ et $y=0$. \\
Cela permet de conclure à l'exactitude de:
$$0 \rightarrow \overline{T(K)}_S \rightarrow \prod_{v \in S} T(K_v) \rightarrow (\Sha^{d+2}_S(\tilde{T}))^D.$$
Reste donc à montrer l'exactitude de:
$$\prod_{v \in S} T(K_v) \rightarrow (\Sha^{d+2}_S(\tilde{T}))^D \rightarrow (\Sha^{d+2}(\tilde{T}))^D \rightarrow 0.$$
Pour ce faire, remarquons que nous disposons d'une suite exacte:
$$0 \rightarrow \Sha^{d+2}(\tilde{T}) \rightarrow \Sha^{d+2}_S(\tilde{T}) \rightarrow \bigoplus_{v \in S} H^{d+2}(K_v,\tilde{T}),$$
d'où une suite exacte duale:
$$\prod_{v \in S} T(K_v)^{\wedge} \rightarrow \Sha^{d+2}_S(\tilde{T})^D \rightarrow \Sha^{d+2}(\tilde{T})^D \rightarrow 0.$$
Or une chasse au diagramme dans le diagramme commutatif à lignes exactes (D) montre que le conoyau de $\overline{T(K)}_S \rightarrow \prod_{v \in S} T(K_v)$ est de $n$-torsion. Par conséquent, l'image de $\prod_{v \in S} T(K_v) \rightarrow (\Sha^{d+2}_S(\tilde{T}))^D$ est aussi de $n$-torsion et elle coïncide avec l'image de $\prod_{v \in S} T(K_v)^{\wedge} \rightarrow \Sha^{d+2}_S(\tilde{T})^D$. On en déduit l'exactitude de
$$\prod_{v \in S} T(K_v) \rightarrow (\Sha^{d+2}_S(\tilde{T}))^D \rightarrow (\Sha^{d+2}(\tilde{T}))^D \rightarrow 0,$$
ce qui achève la preuve.
\item[(ii)] En passant à la limite projective sur S dans la suite exacte de (i), on obtient immédiatement l'exactitude de:
$$0 \rightarrow \overline{T(K)} \rightarrow \prod_{v \in X^{(1)}} T(K_v) \rightarrow \Sha^2_{\omega}(\tilde{T})^D.$$
Pour établir les derniers termes de la suite exacte, on procède comme dans (i). En effet, on vérifie exactement de la même manière que l'on peut supposer que $T$ s'insère dans une suite exacte:
$$0 \rightarrow F \rightarrow R \rightarrow T \rightarrow 0,$$
où $F$ est un groupe fini et $R$ est un tore quasi-trivial. On remarque alors que l'on dispose d'une suite exacte:
$$0 \rightarrow \Sha^{d+2}(\tilde{T}) \rightarrow \Sha^{d+2}_{\omega}(\tilde{T}) \rightarrow \bigoplus_{v \in X^{(1)}} H^{d+2}(K_v,\tilde{T}),$$
d'où une suite exacte duale:
$$\prod_{v \in X^{(1)}} T(K_v)^{\wedge} \rightarrow \Sha^{d+2}_{\omega}(\tilde{T})^D \rightarrow \Sha^{d+2}(\tilde{T})^D \rightarrow 0.$$
Or, comme $R$ vérifie l'approximation faible, une chasse au diagramme dans:\\
\centerline{\xymatrix{
R(K)\ar[d] \ar[r] & T(K)\ar[d] \ar[r] & H^1(K,F)\ar[d] \ar[r] & 0\\
\prod_{v \in X^{(1)}} R(K_v) \ar[r] & \prod_{v \in X^{(1)}} T(K_v) \ar[r] & \prod_{v \in X^{(1)}} H^1(K_v,F) \ar[r] & 0\\
}}
permet de conclure que le groupe topologique $\left( \prod_{v \in X^{(1)}} T(K_v)\right) /\overline{T(K)}$ s'identifie à $\left( \prod_{v \in X^{(1)}} H^1(K_v,F)\right) /\overline{H^1(K,F)}$. C'est donc un groupe compact et l'image de $\prod_{v \in X^{(1)}} T(K_v)$ dans $\Sha^2_{\omega}(\tilde{T})^D$ est fermée. On en déduit que cette image coïncide avec celle de $\prod_{v \in X^{(1)}} T(K_v)^{\wedge}$, ce qui impose l'exactitude de:
$$\prod_{v \in X^{(1)}} T(K_v) \rightarrow \Sha^{d+2}_{\omega}(\tilde{T})^D \rightarrow \Sha^{d+2}(\tilde{T})^D \rightarrow 0.$$
\end{itemize}
\end{proof}

\begin{remarque}
\begin{itemize}
\item[(i)] La preuve précédente montre aussi que $\left( \prod_{v \in S} T(K_v)\right) /\overline{T(K)}_S$ et $\left( \prod_{v \in X^{(1)}} T(K_v)\right) /\overline{T(K)}$ sont d'exposant fini, mais pas forcément finis.
\item[(ii)] Pour montrer que $0 \rightarrow \overline{T(K)}_S \rightarrow \prod_{v \in S} T(K_v) \rightarrow (\Sha^{d+2}_S(\tilde{T}))^D$ est un complexe, nous aurions aussi pu faire appel à un argument topologique comme dans le lemme 9.5 et la remarque 9.8 de \cite{CTH}.
\end{itemize}
\end{remarque}

Nous pouvons récrire la suite exacte (i) du théorème sous une forme plus agréable:

\begin{corollary}\label{AF corollaire 1}
Rappelons que $T$ est un tore et $\tilde{T} = \hat{T} \otimes^{\mathbf{L}} \mathbb{Z}(d)$. On a une suite exacte:
$$0 \rightarrow \overline{T(K)}_S \rightarrow \prod_{v \in S} T(K_v) \rightarrow \overline{\Sha^{d+2}_S(\tilde{T})}^D \rightarrow \Sha^{1}(T) \rightarrow 0.$$
\end{corollary}

\begin{proof}
Le corollaire découle immédiatement du théorème et de la remarque précédents, du fait que les groupes $\Sha^{d+2}_S(\tilde{T})$ et $\Sha^{d+2}(\tilde{T})$ sont de torsion de type cofini, et du théorème de dualité \ref{0}.
\end{proof}

Le théorème \ref{AF tore} nous permet finalement d'étudier les propriétés d'approximation faible, d'approximation faible faible et d'approximation faible faible dénombrable pour le tore $T$:

\begin{corollary} \textbf{(Propriétés d'approximation faible des tores)} \label{AFF tore}\\
Rappelons que $T$ est un tore et $\tilde{T} = \hat{T} \otimes^{\mathbf{L}} \mathbb{Z}(d)$.
\begin{itemize}
\item[(i)] Le tore $T$ vérifie l'approximation faible si, et seulement si, $\Sha^{d+2}(\tilde{T}) = \Sha^{d+2}_{\omega}(\tilde{T})$.
\item[(ii)] Le tore $T$ vérifie l'approximation faible faible si, et seulement si, le groupe de torsion $\Sha^{d+2}_{\omega}(\tilde{T})$ est de type cofini. En particulier, lorsque $\Sha^{d+2}(L,\mathbb{Z}(d)) = 0$ pour une extension finie $L$ de $K$ déployant $T$, le tore $T$ vérifie l'approximation faible faible si, et seulement si, $\Sha^{d+2}_{\omega}(\tilde{T})$ est fini.
\item[(iii)] Le tore $T$ vérifie l'approximation faible faible dénombrable si, et seulement si, le groupe de torsion $\Sha^{d+2}_{\omega}(\tilde{T})$ est dénombrable.
\end{itemize}
\end{corollary}

\begin{proof}
\begin{itemize}
\item[(i)] Cette propriété découle immédiatement du théorème \ref{AF tore}(ii).
\item[(ii)] Supposons que $T$ vérifie l'approximation faible faible, et notons $S_0$ une partie finie de $X^{(1)}$ telle que, pour chaque partie finie $S$ de $X^{(1)}$ n'intersectant pas $S_0$, le groupe $T(K)$ est dense dans $\prod_{v \in S} T(K_v)$. D'après le théorème \ref{AF tore}(i), cela impose que, pour une telle partie $S$, on a $\Sha^{d+2}_S(\tilde{T}) = \Sha^{d+2}(\tilde{T})$, d'où une suite exacte:
$$0 \rightarrow \Sha^{d+2}(\tilde{T}) \rightarrow \Sha^{d+2}_{\omega}(\tilde{T}) \rightarrow \bigoplus_{v \in S_0} H^{d+2}(K_v,\tilde{T}).$$
On en déduit que $\Sha^{d+2}_{\omega}(\tilde{T})$ est de torsion de type cofini.\\
Réciproquement, supposons que le groupe de torsion $\Sha^{d+2}_{\omega}(\tilde{T})$ est de type cofini. Montrons d'abord que $\Sha^{d+2}(\tilde{T})_{\text{div}} = \Sha^{d+2}_{\omega}(\tilde{T})_{\text{div}}$. Comme $\left( \prod_{v \in X^{(1)}} T(K_v)\right) /\overline{T(K)}$ est d'exposant fini, le théorème \ref{AF tore}(ii) fournit une suite exacte:
$$ 0 \rightarrow A \rightarrow \Sha^{d+2}_{\omega}(\tilde{T})^D \rightarrow  \Sha^{d+2}(\tilde{T})^D \rightarrow 0,$$
où $A$ est un groupe abélien de torsion. On en déduit que $\Sha^{d+2}_{\omega}(\tilde{T})^D/(\Sha^{d+2}_{\omega}(\tilde{T})^D)_{\text{tors}} \cong \Sha^{d+2}(\tilde{T})^D/(\Sha^{d+2}(\tilde{T})^D)_{\text{tors}}$, et donc que $\Sha^{d+2}_{\omega}(\tilde{T})_{\text{div}} \cong \Sha^{d+2}(\tilde{T})_{\text{div}}$. Rappelons maintenant qu'il existe $n>0$ tel que le groupe $\left( \prod_{v \in S} T(K_v) \right)  / \overline{T(K)}_S$ est de $n$-torsion pour toute partie finie S de $X^{(1)}$ et fixons un nombre premier $l$ divisant $n$. Comme $\Sha^{d+2}_{\omega}(\tilde{T})$ est de type cofini et $\Sha^{d+2}(\tilde{T})_{\text{div}} = \Sha^{d+2}_{\omega}(\tilde{T})_{\text{div}}$, il existe une partie finie $S_l \subseteq X^{(1)}$ telle que $\Sha^{d+2}_{\omega}(\tilde{T})\{l\} = \Sha^{d+2}_{S_{l}}(\tilde{T})\{l\}$. On en déduit que, pour chaque partie finie $S$ de $X^{(1)}$ n'intersectant pas $S_l$, on a $\Sha^{d+2}(\tilde{T})\{l\} = \Sha^{d+2}_S(\tilde{T})\{l\}$. Grâce au théorème \ref{AF tore}(i), on en déduit que $$\left( \left( \prod_{v \in S} T(K_v) \right)  / \overline{T(K)}_S \right) \{l\} = 0.$$
Cela prouve que $T(K)$ est dense dans $\prod_{v \in S} T(K_v)$ pour chaque partie finie $S$ de $X^{(1)}$ n'intersectant pas $\bigcup_{l|n} S_l$, et donc $T$ vérifie l'approximation faible faible.
\item[(iii)] Supposons que $\Sha^{d+2}_{\omega}(\tilde{T})$ est dénombrable. Soit $S_0$ l'ensemble des places $v \in X^{(1)}$ telles qu'il existe un élément de $\Sha^{d+2}_{\omega}(\tilde{T})$ dont l'image dans $H^{d+2}(K_v,\tilde{T})$ est non nulle. On sait que $\Sha^{d+2}_{\omega}(\tilde{T}) = \bigcup_{S \subseteq X^{(1)} \text{ fini}} \Sha^{d+2}_S(\tilde{T})$. Comme $\Sha^{d+2}_{\omega}(\tilde{T})$ est dénombrable, il existe une suite $(S_i)_{i \geq 1}$ de parties finies de $X^{(1)}$ telle que $\Sha^{d+2}_{\omega}(\tilde{T}) = \bigcup_{i \geq 1} \Sha^{d+2}_{S_i}(\tilde{T})$. On en déduit que $\Sha^{d+2}_{\omega}(\tilde{T}) \subseteq \Sha^{d+2}_{\bigcup_{i \geq 1} S_i}(\tilde{T})$ et que $S_0 \subseteq \bigcup_{i \geq 1} S_i$. Par conséquent, $S_0$ est dénombrable. De plus, pour chaque partie finie $S$ de $X^{(1)}$ n'intersectant par $S_0$, on a  $\Sha^{d+2}_S(\tilde{T}) = \Sha^{d+2}(\tilde{T})$. Le théorème \ref{AF tore}(i) impose alors que $T(K)$ est dense dans $\prod_{v \in S} T(K_v)$, et donc $T$ vérifie l'approximation faible faible dénombrable.\\
Réciproquement, supposons que $T$ vérifie l'approximation faible faible dénombrable. Soit $S_0$ une partie dénombrable de $X^{(1)}$ telle que, pour chaque partie finie $S$ de $X^{(1)}$ n'intersectant pas $S_0$, le groupe $T(K)$ est dense dans $\prod_{v \in S} T(K_v)$.  D'après le théorème \ref{AF tore}(i), cela impose que, pour une telle partie $S$, on a $\Sha^{d+2}_S(\tilde{T}) = \Sha^{d+2}(\tilde{T})$, d'où une suite exacte:
$$0 \rightarrow \Sha^{d+2}(\tilde{T}) \rightarrow \Sha^{d+2}_{\omega}(\tilde{T}) \rightarrow \bigoplus_{v \in S_0} H^{d+2}(K_v,\tilde{T}).$$
On en déduit que $\Sha^{d+2}_{\omega}(\tilde{T})$ est dénombrable.
\end{itemize}
\end{proof}

\begin{remarque}
\begin{itemize}
\item[(i)] Il existe des tores qui ne vérifient pas l'approximation faible faible dénombrable. C'est par exemple le cas du tore de la proposition 3.5 de \cite{HS2} (cela découle de la preuve de cette proposition et de l'indénombrabilité de $\mathbb{Q}_p$). 
\item[(ii)] Je ne connais pas la réponse à la question suivante mais elle me semble intéressante: existe-t-il des tores vérifiant l'approximation faible faible dénombrable mais ne vérifiant pas l'approximation faible faible?
\end{itemize}
\end{remarque}

\section{\scshape Applications au principe local-global}

Dans toute cette partie, on suppose (H \ref{40}), c'est-à-dire que $X$ est une courbe.

\subsection{Torseurs sous un tore}\label{torseur tore}

Dans cette section, on suppose que $d>0$, c'est-à-dire que $k$ (qui est de caractéristique 0) n'est pas $\mathbb{C}((t))$. Soient $T$ un tore sur $K$ et $Y$ un espace principal homogène sous $T$ tel que $Y(\mathbb{A}_K) \neq \emptyset$. Comme dans la partie 5 de \cite{HS1}, on peut définir:
\begin{align*}
H^{d+2}_{\text{lc}}(Y,\mathbb{Q}/\mathbb{Z}(d+1)) &= \text{Ker}(H^{d+2}(Y,\mathbb{Q}/\mathbb{Z}(d+1))/\text{Im}(H^{d+2}(K,\mathbb{Q}/\mathbb{Z}(d+1)))\\ & \rightarrow \prod_{v \in X^{(1)}} H^{d+2}(Y_{K_v},\mathbb{Q}/\mathbb{Z}(d+1))/\text{Im}(H^{d+2}(K_v,\mathbb{Q}/\mathbb{Z}(d+1)))
\end{align*}
où $Y_{K_v}$ désigne $Y \times_K K_v$, et on peut construire un morphisme $$\rho_Y: H^{d+2}_{\text{lc}}(Y,\mathbb{Q}/\mathbb{Z}(d+1)) \rightarrow \mathbb{Q}/\mathbb{Z}.$$
On prouve alors exactement de la même manière que dans la partie 5 de \cite{HS1} le théorème:

\begin{theorem}\label{local-global}
On rappelle que l'on a supposé que $X$ est une courbe et que $d>0$. Si $\rho_Y$ est trivial, alors $Y(K) \neq \emptyset$.
\end{theorem}

\begin{proof} \textit{(Esquisse)}\\
Comme $K$ est de dimension cohomologique $d+2$, la suite spectrale $H^p(K,H^q(\overline{Y},\mathbb{Q}/\mathbb{Z}(d+1))) \Rightarrow H^{p+q}(Y,\mathbb{Q}/\mathbb{Z}(d+1))$ induit un morphisme $H^{d+1}(K,H^1(\overline{Y},\mathbb{Q}/\mathbb{Z}(d+1))) \rightarrow H^{d+2}(Y,\mathbb{Q}/\mathbb{Z}(d+1))/\text{Im}(H^{d+2}(K,\mathbb{Q}/\mathbb{Z}(d+1)))$. De plus, on montre exactement comme dans le lemme 5.2 de \cite{HS1} que $H^1(\overline{Y},\mathbb{Q}/\mathbb{Z}(d+1)) \cong \tilde{T}_t$. La nullité de $H^{d+1}(K,\hat{T} \otimes \mathbb{Q}(d)) $ et $ H^{d+2}(K,\hat{T} \otimes \mathbb{Q}(d))$ permet alors de déduire des isomorphismes $H^{d+1}(K,H^1(\overline{Y},\mathbb{Q}/\mathbb{Z}(d+1))) \cong H^{d+1}(K,\tilde{T}_t) \cong H^{d+2}(K,\tilde{T})$, d'où un morphisme $ H^{d+2}(K,\tilde{T}) \rightarrow H^{d+2}(Y,\mathbb{Q}/\mathbb{Z}(d+1))/\text{Im}(H^{d+2}(K,\mathbb{Q}/\mathbb{Z}(d+1)))$. En passant aux éléments localement triviaux, on obtient un morphisme $\tau: \Sha^{d+2}(\tilde{T}) \rightarrow H^{d+2}_{\text{lc}}(Y,\mathbb{Q}/\mathbb{Z}(d+1))$. \\
En procédant exactement comme dans la proposition 5.3 de \cite{HS1}, on peut montrer que, si $\alpha$ est un élément de $\Sha^{d+2}(\tilde{T})$, $[Y]$ désigne la classe de $Y$ dans $\Sha^1(T)$ et $PT(.,.): \Sha^1(T) \times \Sha^{d+2}(\tilde{T}) \rightarrow \mathbb{Q}/\mathbb{Z}$ désigne l'accouplement du théorème \ref{PT tore}, alors $\rho_Y(\tau(\alpha)) = PT([Y],\alpha)$ au signe près. Ainsi, grâce au théorème \ref{PT tore}, on déduit que si $\rho_Y$ est trivial, alors $[Y]=0$ et donc $Y(K) \neq \emptyset$.
\end{proof}

\subsection{Torseurs sous un groupe réductif}

\subsubsection{Cas $d=0$}

Dans cette section, on suppose que $k=\mathbb{C}((t))$ (et donc que $d=0$). Nous allons suivre de près les méthodes développées par Borovoi dans l'article \cite{Bor} afin de montrer que, pour chaque $K$-groupe linéaire réductif connexe $H$, il existe un accouplement non dégénéré $\text{BM}: \Sha^1(H) \times \overline{\Brusse(H)} \rightarrow \mathbb{Q}/\mathbb{Z}$. En fait, un tel accouplement a déjà été étudié dans la partie 10 de l'article \cite{CTH}: les méthodes utilisées dans ce dernier, fondées sur les revêtements spéciaux, permettent d'établir la non-dégénérescence à gauche mais pas la non-dégénérescence à droite.\\

Considérons $H$ un groupe réductif connexe sur $K$, et notons $H^{ss}$ son sous-groupe dérivé et $H^{sc}$ son revêtement universel, qui est semi-simple simplement connexe. Soit $T$ un tore maximal de $H$. Notons $T^{(sc)}$ l'image réciproque de $T$ par le morphisme composé $\rho: H^{sc} \rightarrow H^{ss} \rightarrow H$ et $G = [T^{(sc)} \rightarrow T]$ le cône de $T^{(sc)} \rightarrow T$. Pour $L$ une extension de $K$, on définit la cohomologie galoisienne abélienne de $H$ par $H^r_{\text{ab}}(L,H) = H^r(L,G)$. On rappelle que, dans la section 3 de \cite{Bor}, Borovoi a construit, pour chaque extension $L$ de $K$, des morphismes d'abélianisation:
\begin{align*}
\text{ab}^0_L:& H(L) \rightarrow H^0_{\text{ab}}(L,H)\\
\text{ab}^1_L:& H^1(L,H) \rightarrow H^1_{\text{ab}}(L,H).
\end{align*}
Dans la suite, nous nous intéressons au morphisme $\text{ab}^1_L$ avec $L=K$ et avec $L= K_v$ pour un certain $v \in X^{(1)}$. Commençons par montrer qu'il est injectif.

\begin{theorem} \textbf{(Cas particulier de la conjecture de Serre II)}\\
Rappelons que nous avons supposé que $X$ est une courbe sur $k=\mathbb{C}((t))$. Supposons de plus que $H$ soit semisimple simplement connexe.
\begin{itemize}
\item[(i)] Pour chaque $v \in X^{(1)}$, l'ensemble pointé $H^1(K_v,H)$ est trivial.
\item[(ii)] L'ensemble pointé $H^1(K,H)$ est trivial.
\end{itemize}
\end{theorem}

\begin{proof}
\begin{itemize}
\item[(i)] On pourra aller voir le théorème 4.7 de \cite{BT}.
\item[(ii)] D'après le théorème 10 de \cite{Lan}, le corps $\mathbb{C}((t))$ est $C_1$. En utilisant en plus le théorème 6 de \cite{Lan} complété par \cite{Nag}, le corps $K$ est un corps $C_2$ de caractéristique 0 et son extension abélienne maximale est de dimension cohomologique au plus 1. La partie (v) du théorème 1.2 de \cite{CGR} permet alors de conclure.
\end{itemize}
\end{proof}

\begin{corollary}\label{inj}
On rappelle que l'on a supposé que $X$ est une courbe sur $k=\mathbb{C}((t))$.
\begin{itemize}
\item[(i)] Pour chaque $v \in X^{(1)}$, le morphisme d'abélianisation $\text{ab}^1_{K_v}$ est injectif.
\item[(ii)] Le morphisme d'abélianisation $\text{ab}^1_{K}$ est injectif.
\end{itemize}
\end{corollary}

\begin{proof}
La preuve découle du théorème précédent par un argument de torsion. Elle est tout à fait analogue à celle du corollaire 5.4.1 de \cite{Bor}.
\end{proof}

Quant à la surjectivité des morphismes d'abélianisation, elle découle des travaux de González-Avilés (\cite{GA}):

\begin{theorem} \label{surj}
On rappelle que l'on a supposé que $X$ est une courbe sur $k=\mathbb{C}((t))$.
\begin{itemize}
\item[(i)] Pour chaque $v \in X^{(1)}$, le morphisme d'abélianisation $\text{ab}^1_{K_v}$ est surjectif.
\item[(ii)] Le morphisme d'abélianisation $\text{ab}^1_{K}$ est surjectif.
\end{itemize}
\end{theorem}

\begin{proof}
Rappelons que le corps $K$ (resp. $K_v$ pour $v \in X^{(1)}$) est de type de Douai (voir définition 5.2 de \cite{GA}) d'après le théorème 2.1(a) de \cite{CGR}, puisque $K$ (resp. $K_v$) est un corps de caractéristique 0 de dimension cohomologique 2 dont l'extension abélienne maximale est de dimension cohomologique au plus 1 et l'exposant et l'indice coïncident pour les algèbres simples centrales sur des extensions finies de $K$ (resp. $K_v$) (voir p. 350 de \cite{Par} ou le théorème 5.5 de \cite{HHK2}). Cela permet de déduire la surjectivité des morphismes d'abélianisation du théorème 5.5(i) de \cite{GA}, puisque nous sommes en caractéristique 0.
\end{proof}

Nous allons voir que les propriétés que nous venons d'établir concernant les morphismes d'abélianisation permettent d'étudier l'obstruction au principe local-global.\\
Soit $Y$ un espace principal homogène sous $H$ tel que, pour chaque $v \in X^{(1)}$, on a $Y(K_v) \neq \emptyset$. En considérant un modèle géométriquement intègre de $Y$ sur un ouvert non vide de $X$, on peut définir un accouplement de type Brauer-Manin:
$$[.,.]: Y(\mathbb{A}_K) \times \Brusse(Y) \rightarrow \mathbb{Q}/\mathbb{Z},$$
 et on remarque que, pour $\alpha \in \Brusse(Y)$, la valeur de $[(P_v),\alpha]$ ne dépend pas du choix du point adélique $(P_v) \in Y(\mathbb{A}_K)$. Comme on dispose d'un isomorphisme canonique $\Brusse(Y) \rightarrow \Brusse(H)$ (voir le lemme 5.2(iii) de \cite{BVH}), cela permet de définir un accouplement:
\begin{align*}
\text{BM}: \Sha^1(H) \times \Brusse(H) & \rightarrow \mathbb{Q}/\mathbb{Z}\\
([Y],\alpha) &\mapsto [(P_v),\alpha]
\end{align*}
où $(P_v) \in Y(\mathbb{A}_K)$ est un point adélique quelconque.\\
D'autre part, le noyau du morphisme $T^{(sc)} \rightarrow T$ étant fini, le corollaire \ref{cor} fournit un accouplement parfait de type Poitou-Tate:
$$\text{PT}: \Sha^1(G) \times \overline{\Sha^1(\tilde{G})} \rightarrow \mathbb{Q}/\mathbb{Z}.$$
De plus, nous disposons d'isomorphismes permettant de comparer les deux accouplements précédents:
\begin{itemize}
\item[$\bullet$] d'après le corollaire \ref{inj} et le théorème \ref{surj}, les morphismes d'abélianisation induisent une bijection $\text{ab}^1: \Sha^1(H) \rightarrow \Sha^1(G)$,
\item[$\bullet$] d'après le corollaire 2.20 et le théorème 4.8 de \cite{BVH}, on dispose d'un isomorphisme $B: \Brusse(H) \rightarrow \Sha^1(\tilde{G})$.
\end{itemize}
Le lemme suivant montre que les accouplements de Brauer-Manin et de Poitou-Tate sont compatibles:

\begin{lemma}
On rappelle que l'on a supposé que $X$ est une courbe sur $k=\mathbb{C}((t))$.
Le diagramme:\\
\centerline{\xymatrix{
\text{BM}: & \Sha^1(H) \ar[d]^{\text{ab}^1} \ar@{}[r]|{\times}   & \overline{\Brusse(H)}\ar[d]^{B} \ar[r] & \mathbb{Q}/\mathbb{Z}\ar@{=}[d]\\
\text{PT}: &\Sha^1(G) \ar@{}[r]|{\times}  & \overline{\Sha^1(\tilde{G})} \ar[r] & \mathbb{Q}/\mathbb{Z}.
}}
est commutatif au signe près.
\end{lemma}

Avant de passer à la preuve du lemme précédent, introduisons les notations suivantes, calquées de l'article \cite{HS3}:

\begin{notation}
Lorsque $Y$ est une variété lisse géométriquement intègre sur $K$ et $\mathcal{Y}$ un modèle lisse géométriquement intègre de $Y$ sur un ouvert non vide $U_0$ de $X$:
\begin{itemize}
\item[$\bullet$] $\pi^Y: Y \rightarrow \text{Spec}\; K$ et $\pi^{\mathcal{Y}}:\mathcal{Y} \rightarrow U_0$ sont les morphismes structuraux,
\item[$\bullet$] $\overline{Y}$ désigne $Y \times_k k^s$,
\item[$\bullet$] $KD(Y)=[k^s(\overline{Y})^{\times} \rightarrow \text{Div}(\overline{Y})]=\tau_{\leq 1}\mathbb{R}\pi^Y_*\mathbb{G}_m[1]$,
\item[$\bullet$] $KD'(Y)=[k^s(\overline{Y})^{\times}/{k^s}^{\times} \rightarrow \text{Div}(\overline{Y})]=[\mathbb{G}_m \rightarrow \tau_{\leq 1}\mathbb{R}\pi^Y_*\mathbb{G}_m][1]$,
\item[$\bullet$] $\mathcal{KD}(\mathcal{Y})=\tau_{\leq 1}\mathbb{R}\pi^{\mathcal{Y}}_*\mathbb{G}_m[1]$,
\item[$\bullet$] $\mathcal{KD}'(\mathcal{Y})=[\mathbb{G}_m \rightarrow \tau_{\leq 1}\mathbb{R}\pi^{\mathcal{Y}}_*\mathbb{G}_m][1]$.
\end{itemize}
\end{notation}

\begin{proof}
Soient $[Y] \in \Sha^1(H)$ et $\alpha \in \Brusse(H)$. On notera aussi $\alpha$ l'image réciproque de $\alpha$ par l'isomorphisme $\Brusse(Y) \rightarrow \Brusse(H)$. Soit $U_0$ un ouvert non vide de $X$ tel qu'il existe:
\begin{itemize}
\item[$\bullet$] $\mathcal{Y}$ un modèle géométriquement intègre de $Y$ sur $U_0$,
\item[$\bullet$] $\mathcal{H}$ (resp. $\mathcal{H}^{(sc)}$) un groupe réductif sur un ouvert non vide de $U_0$ étendant $H$ (resp. $H^{sc}$),
\item[$\bullet$] un morphisme $\mathcal{H}^{(sc)} \rightarrow \mathcal{H}$ étendant $H^{sc} \rightarrow H$,
\item[$\bullet$] $\mathcal{T}$ (resp. $\mathcal{T}^{(sc)}$) un tore sur $U_0$ étendant $T$ (resp. $T^{(sc)}$),
\item[$\bullet$] un morphisme $\mathcal{T}^{(sc)} \rightarrow \mathcal{T}$ étendant $T^{(sc)} \rightarrow T$,
\item[$\bullet$] un morphisme $\mathcal{T}^{(sc)} \rightarrow \mathcal{H}^{(sc)}$ étendant $T^{(sc)} \hookrightarrow H^{sc}$,
\item[$\bullet$] un morphisme $\mathcal{T} \rightarrow \mathcal{H}$ étendant $T \hookrightarrow H$,
\item[$\bullet$] un morphisme $\mathcal{H} \times_{U_0} \mathcal{Y} \rightarrow \mathcal{Y}$ étendant l'action $H \times_K Y \rightarrow Y$ de $H$ sur $Y$,
\end{itemize}
et tel que le diagramme:\\
\centerline{\xymatrix{
\mathcal{T}^{(sc)} \ar[r]\ar[d] & \mathcal{T}\ar[d]\\
\mathcal{H}^{(sc)} \ar[r] & \mathcal{H}
}}
est commutatif. On adopte en outre les notations suivantes:
\begin{itemize}
\item[$\bullet$] $\mathcal{G} = [\mathcal{T}^{(sc)} \rightarrow \mathcal{T}]$,
\item[$\bullet$] $U$ désigne un ouvert de $U_0$ tel que $\alpha$ s'étend en un élément $\alpha_U \in H^1_c(U,\mathcal{KD}'(\mathcal{Y}))$ et tel que $\text{ab}^1([Y])$ s'étend en un élément de $H^1(U,\mathcal{G})$,
\item[$\bullet$] $\mathcal{E}_Y$ désigne la classe du morphisme naturel $\mathcal{KD}'(\mathcal{Y}) \rightarrow \mathbb{G}_m[2]$ dans le groupe $\text{Ext}^1_{U}(\mathcal{KD}'(\mathcal{Y})),\mathbb{G}_m[1])$, 
\item[$\bullet$] $C(\mathcal{H})=[\mathcal{KD}'(\mathcal{H}) \rightarrow \mathcal{KD}'(\mathcal{H}^{(sc)})][-1]$ et $C(\mathcal{T})=[\mathcal{KD}'(\mathcal{T}) \rightarrow \mathcal{KD}'(\mathcal{T}^{(sc)})][-1]$,
\item[$\bullet$]  $\text{Ext}^1_{U}(\mathcal{KD}'(\mathcal{Y} \oplus \mathcal{H}),\mathbb{G}_m[1]) = \text{Ext}^1_{U}(\mathcal{KD}'(\mathcal{Y}),\mathbb{G}_m[1]) \oplus \text{Ext}^1_{U}(\mathcal{KD}'(\mathcal{H}),\mathbb{G}_m[1])$ et $H^1_c(U,\mathcal{KD}'(\mathcal{Y} \oplus \mathcal{H})) = H^1_c(U,\mathcal{KD}'(\mathcal{Y})) \oplus H^1_c(U,\mathcal{KD}'( \mathcal{H}))$.
\end{itemize}
On dispose d'un diagramme commutatif d'accouplements:
\newpage

\centerline{\xymatrix{
\text{Ext}^1_{U}(\mathcal{KD}'(\mathcal{Y}),\mathbb{G}_m[1])  \ar@{}[r]|{\times}& H^1_c(U,\mathcal{KD}'(\mathcal{Y})) \ar[r]\ar[d] & H^3_c(U,\mathbb{G}_m) \ar[r]^{\cong}\ar@{=}[d] & \mathbb{Q}/\mathbb{Z}\\
\text{Ext}^1_{U}(\mathcal{KD}'(\mathcal{Y} \times_U \mathcal{H}),\mathbb{G}_m[1]) \ar[u] \ar[d] \ar@{}[r]|{\times}& H^1_c(U,\mathcal{KD}'(\mathcal{Y} \times_U \mathcal{H})) \ar[r] & H^3_c(U,\mathbb{G}_m) \ar[r]^{\cong}\ar@{=}[d] & \mathbb{Q}/\mathbb{Z}\\
\text{Ext}^1_{U}(\mathcal{KD}'(\mathcal{Y} \oplus \mathcal{H}),\mathbb{G}_m[1]) \ar@{}[r]|{\times}& H^1_c(U,\mathcal{KD}'(\mathcal{Y} \oplus \mathcal{H})) \ar[u]\ar[d]\ar[r] & H^3_c(U,\mathbb{G}_m) \ar[r]^{\cong}\ar@{=}[d] & \mathbb{Q}/\mathbb{Z}\\
\text{Ext}^1_{U}(\mathcal{KD}'(\mathcal{H}),\mathbb{G}_m[1])  \ar[u]\ar[d]\ar@{}[r]|{\times}& H^1_c(U,\mathcal{KD}'(\mathcal{H})) \ar[r] & H^3_c(U,\mathbb{G}_m) \ar[r]^{\cong}\ar@{=}[d] & \mathbb{Q}/\mathbb{Z}\\
\text{Ext}^1_{U}(C(\mathcal{H}),\mathbb{G}_m[1]) \ar@{}[r]|{\times}& H^1_c(U,C(\mathcal{H})) \ar[r]\ar[d]\ar[u] & H^3_c(U,\mathbb{G}_m) \ar[r]^{\cong}\ar@{=}[d] & \mathbb{Q}/\mathbb{Z}\\
\text{Ext}^1_{U}(C(\mathcal{T}),\mathbb{G}_m[1]) \ar[u]\ar[d] \ar@{}[r]|{\times}& H^1_c(U,C(\mathcal{T})) \ar[r] & H^3_c(U,\mathbb{G}_m) \ar[r]^{\cong}\ar@{=}[d] & \mathbb{Q}/\mathbb{Z}\\
\text{Ext}^1_{U}(\tilde{\mathcal{G}},\mathbb{G}_m[1])  \ar@{}[r]|{\times}& H^1_c(U,\tilde{\mathcal{G}}) \ar[u] \ar[r]\ar@{=}[d] & H^3_c(U,\mathbb{G}_m) \ar[r]^{\cong}\ar@{=}[d] & \mathbb{Q}/\mathbb{Z}\\
H^1(U,\mathcal{G})\ar[u]  \ar@{}[r]|{\times}& H^1_c(U,\tilde{\mathcal{G}}) \ar[r] & H^3_c(U,\mathbb{G}_m) \ar[r]^{\cong} & \mathbb{Q}/\mathbb{Z}
}}
où les morphismes verticaux sont induits (du haut vers le bas) par:
\begin{itemize}
\item[$\bullet$] le morphisme $\mathcal{KD}'(\mathcal{Y}) \rightarrow \mathcal{KD}'(\mathcal{Y} \times_U \mathcal{H})$ induit par le morphisme $\mathcal{Y} \times_U \mathcal{H} \rightarrow \mathcal{Y}$ qui étend l'action de $H$ sur $Y$,
\item[$\bullet$] le morphisme naturel $\mathcal{KD}'(\mathcal{Y}) \oplus \mathcal{KD}'(\mathcal{H}) \rightarrow  \mathcal{KD}'(\mathcal{Y} \times_U \mathcal{H})$ induit par les projections,
\item[$\bullet$] la projection $\mathcal{KD}'(\mathcal{Y}) \oplus \mathcal{KD}'(\mathcal{H}) \rightarrow \mathcal{KD}'(\mathcal{H})$,
\item[$\bullet$] le morphisme naturel $C(\mathcal{H}) \rightarrow \mathcal{KD}'(\mathcal{H})$,
\item[$\bullet$] le morphisme naturel $C(\mathcal{H}) \rightarrow C(\mathcal{T})$,
\item[$\bullet$] le morphisme $\tilde{\mathcal{G}} \rightarrow C(\mathcal{T})$ le morphisme induit par $\hat{\mathcal{T}} \cong [\mathbb{G}_m \rightarrow \pi^{\mathcal{T}}_*\mathbb{G}_m] \rightarrow \mathcal{KD}'(\mathcal{T})[-1]$ et par $\hat{\mathcal{T}^{(sc)}} \cong [\mathbb{G}_m \rightarrow \pi^{\mathcal{T}^{(sc)}}_*\mathbb{G}_m] \rightarrow \mathcal{KD}'(\mathcal{T}^{(sc)})[-1]$,
\item[$\bullet$] l'accouplement $\mathcal{G} \otimes^{\mathbf{L}} \tilde{\mathcal{G}} \rightarrow \mathbb{G}_m[1]$.
\end{itemize}
De manière analogue à la proposition 3.3 de \cite{HS3}, on peut montrer que:
$$\text{BM}([Y],(\alpha)) =\mathcal{E}_{Y} \cup \alpha_U.$$
De plus, l'article \cite{BVH} impose que l'isomorphisme $H^1(K,G) \rightarrow \text{Ext}^1_K(KD'(Y),\mathbb{G}_m[1])$ envoie $\text{ab}^1([Y])$ sur $-E_Y$, où $E_Y$ désigne l'image de $\mathcal{E}_Y$ dans $\text{Ext}^1_K(KD'(Y),\mathbb{G}_m[1])$ (théorème 5.5), et que les morphismes $\mathcal{KD}'(\mathcal{Y}) \oplus \mathcal{KD}'(\mathcal{H}) \rightarrow  \mathcal{KD}'(\mathcal{Y} \times_U \mathcal{H})$, $C(\mathcal{H}) \rightarrow \mathcal{KD}'(\mathcal{H})$ et $\tilde{\mathcal{G}} \rightarrow C(\mathcal{T})$ deviennent des isomorphismes sur la fibre générique (lemmes 5.2, 4.3 et 4.2). Donc, quitte à diminuer $U$, on en déduit que:
$$\text{PT}(\text{ab}^1([Y]), B(\alpha))=-\mathcal{E}_{Y} \cup \alpha_U = -\text{BM}([Y],(\alpha)).$$
\end{proof}

\begin{remarque}
Pour établir la proposition 3.3 de \cite{HS3}, on a besoin de donner une autre construction du morphisme $\text{BM}([Y],.)$ à l'aide du lemme du serpent (lemme 3.1). Dans notre situation, dans la preuve précédente, pour montrer que $\text{BM}([Y],(\alpha)) =\mathcal{E}_{Y} \cup \alpha_U$, il convient de remarquer que la même construction marche même si le morphisme $\text{Br}(K) \rightarrow \bigoplus_{v\in X^{(1)}} \text{Br}(K_v)$ n'est pas injectif et son conoyau n'est pas isomorphe à $\mathbb{Q}/\mathbb{Z}$: en fait, la loi de réciprocité de Weil fournit un morphisme $\text{Coker}( \text{Br}(K) \rightarrow \bigoplus_{v\in X^{(1)}} \text{Br}(K_v)) \rightarrow \mathbb{Q}/\mathbb{Z}$, et cela nous suffit.
\end{remarque}

\begin{theorem} \textbf{(Obstruction au principe local-global)} \label{BM}\\
On rappelle que l'on a supposé que $X$ est une courbe sur $k=\mathbb{C}((t))$. L'accouplement $\text{BM}: \Sha^1(H) \times \overline{\Brusse(H)} \rightarrow \mathbb{Q}/\mathbb{Z}$ induit une bijection $\Sha^1(H) \cong  \overline{\Brusse(H)}^D$.
\end{theorem}

\begin{proof}
Cela découle du lemme précédent, du corollaire \ref{cor}, et du fait que $\text{ab}^1$ est une bijection et que $B$ est un isomorphisme.
\end{proof}

\begin{remarque}
Comme la cohomologie d'un groupe unipotent sur un corps de caractéristique 0 est triviale, en quotientant par le radical unipotent, on montre que le théorème précédent reste valable pour un groupe algébrique linéaire connexe quelconque.
\end{remarque}

\begin{corollary}
On rappelle que l'on a supposé que $X$ est une courbe sur $k=\mathbb{C}((t))$. Soit $H$ un groupe linéaire connexe quelconque sur $K$. La seule obstruction au principe local-global pour les $K$-espaces homogènes sous $H$ est l'obstruction de Brauer-Manin associée à $\Brusse(H)$.
\end{corollary}

\begin{remarque}
Ce corollaire découle uniquement de la non-dégénérescence à gauche de l'accouplement $\text{BM}$: les résultats de l'article \cite{CTH} étaient donc déjà suffisants pour l'établir.
\end{remarque}

\subsubsection{Cas $d=1$}

Dans cette section, on suppose que $k= \mathbb{C}((t_1))((t_2))$ (le cas où $k$ est $p$-adique a été traité dans la partie 6 de \cite{HS1}). Soit $H$ un groupe réductif sur $K$ tel que $H^{sc}$ est quasi-déployé. On suppose de plus que:
\begin{itemize}
\item[$\bullet$] le groupe $\Sha^2(\mathbb{Z})$ est nul,
\item[$\bullet$] le groupe $H$ est déployé sur une extension finie galoisienne $L$ de $K$ (ie $H$ possède un tore maximal qui devient déployé sur $L$) telle que $\Sha^2(L,\mathbb{G}_m)=0$.
\end{itemize}

\begin{remarque}
L'hypothèse $\Sha^2(L,\mathbb{G}_m)=0$ implique que $\Sha^2(\mathbb{G}_m)=0$, mais la réciproque est fausse. En effet:
\begin{itemize}
\item[$\bullet$] si $\Sha^2(L,\mathbb{G}_m)=0$, un argument de restriction-corestriction montre que $\Sha^2(\mathbb{G}_m)$ est d'exposant fini; comme il est divisible, il est nul;
\item[$\bullet$] pour voir que la réciproque est fausse, il suffit de choisir $K=\mathbb{C}((t_1))((t_2))(x)$ et $L$ la clôture galoisienne d'une extension finie $L'$ de $K$ telle que $\Sha^2(L',\mathbb{G}_m)\neq 0$ (cela est possible grâce à l'exemple \ref{ex}).
\end{itemize}
De même, si $\Sha^2(L,\mathbb{Z})=0$, alors $\Sha^2(\mathbb{Z})=0$, mais la réciproque est fausse.
\end{remarque}

Soit $E$ un espace principal homogène sous $H$ tel que $E(\mathbb{A}_K) \neq \emptyset$. Comme dans la partie \ref{torseur tore}, on peut construire un morphisme $\rho_E: H^3_{\text{lc}}(E,\mathbb{Q}/\mathbb{Z}(2)) \rightarrow \mathbb{Q}/\mathbb{Z}$. Exactement de la même manière que dans la partie 6 de \cite{HS1}, on peut montrer le théorème suivant:

\begin{theorem}
On rappelle que l'on a supposé que $X$ est une courbe sur $k= \mathbb{C}((t_1))((t_2))$. Si $H^{sc}$ n'a pas de facteur $E_8$ et si $\rho_E$ est le morphisme trivial, alors $E(K) \neq \emptyset$. Si $H^{sc}$ est de type $E_8$ et si $\rho_E$ est le morphisme trivial, alors $E$ possède un zéro-cycle de degré 1.
\end{theorem}

\begin{remarque}
\begin{itemize}
\item[(i)] La preuve fait appel à l'invariant de Rost. En particulier, on utilise les deux résultats suivants:
\begin{itemize}
\item[$\bullet$] pour $H'$ un groupe semi-simple simplement connexe absolument presque simple quasi-déployé sur un corps $K'$ de dimension cohomologique au plus 3, si $H'$ n'est pas de type $E_8$, le noyau de l'invariant de Rost $H^1(K',H') \rightarrow H^3(K',\mathbb{Q}/\mathbb{Z}(2))$ est trivial (théorème 5.3 de \cite{CPS});
\item[$\bullet$] pour $H'$ un groupe semi-simple simplement connexe absolument presque simple quasi-déployé de type $E_8$ sur un corps $K'$ de dimension cohomologique au plus 4, tout torseur sous $H'$ représentant une classe du noyau de l'invariant de Rost $H^1(K',H') \rightarrow H^3(K',\mathbb{Q}/\mathbb{Z}(2))$ admet un zéro-cycle de degré 1 (\cite{Che1}, \cite{Che2}, \cite{Sem}).
\end{itemize}
\item[(ii)] La nullité de $\Sha^2(\mathbb{Z})$ est utilisée pour établir un résultat analogue à la proposition 6.2 de \cite{HS1}, ou plus précisément pour montrer l'injectivité de $H^3(K,\mathbb{Q}/\mathbb{Z}(2)) \rightarrow \prod_{v \in X^{(1)}} H^3(K_v,\mathbb{Q}/\mathbb{Z}(2))$: le noyau de $H^3(K,\mathbb{Q}/\mathbb{Z}(2)) \rightarrow \prod_{v \in X^{(1)}} H^3(K_v,\mathbb{Q}/\mathbb{Z}(2))$ est isomorphe à $\Sha^4(\mathbb{Z}(2))$, qui est nul si, et seulement si, $\Sha^2(\mathbb{Z})$ l'est d'après le lemme \ref{nul dual}.
\item[(iii)] La nullité du $\Sha^2(L,\mathbb{G}_m)$ permet de calculer la cohomologie des tores quasi-triviaux. Plus précisément, dans la partie 6 de \cite{HS1}, on a besoin de considérer une $z$-extension (suite exacte (37)) faisant intervenir un tore quasi-trivial $Q$. Il se trouve qu'avec le choix que nous avons fait du corps $L$, on peut supposer que le module des caractères de $Q$ est un $\mathbb{Z}[\text{Gal}(L/K)]$-module libre (proposition 3.1 de \cite{MS}). Comme $\Sha^2(L,\mathbb{G}_m)$ est nul, cela permet d'établir avec le lemme de Shapiro que $\Sha^2(Q)$ et $\Sha^3(\tilde{Q})$ sont nuls.
\end{itemize}
\end{remarque}

\subsubsection{Cas $d>1$}

Dans cette section, on suppose que $d>1$. Soit $H$ un groupe réductif sur $K$ tel que $H^{sc}$ est quasi-déployé. Soit $L$ une extension finie galoisienne de $K$ telle que $H$ contient un tore maximal déployé sur $L$. On suppose que $\Sha^4(\mathbb{Z}(2))=0$ et que $\Sha^2(L,\mathbb{G}_m)=0$. Soit $E$ un espace principal homogène sous $H$ tel que $E(\mathbb{A}_K) \neq \emptyset$. Comme dans la partie \ref{torseur tore}, on peut construire un morphisme $\rho_E: H^{d+2}_{\text{lc}}(E,\mathbb{Q}/\mathbb{Z}(d+1)) \rightarrow \mathbb{Q}/\mathbb{Z}$. On peut alors montrer le théorème suivant:

\begin{theorem}\label{zero}
On rappelle que l'on a supposé que $X$ est une courbe et que $d>1$. 
\begin{itemize}
\item[(i)] Si $\rho_E$ est le morphisme trivial et si $H^{sc}$ ne contient que des facteurs de type $A_n$ avec $n \leq 5$, $B_n$ avec $n \leq 6$, $C_n$ avec $n \leq 5$, $D_n$ avec $n \leq 6$, $^{1}D_7$, $E_6$, $E_7$, $F_4$, $G_2$, alors $E(K) \neq \emptyset$. 
\item[(ii)] Si $d=2$, $\rho_E$ est le morphisme trivial et $H^{sc}$ est de type $E_8$, alors $E$ possède un zéro-cycle de degré 1.
\end{itemize}
\end{theorem}

La preuve est très similaire à celle du théorème 6.1 de \cite{HS1} mais présente quelques différences que nous signalons dans la suite.

\begin{proof} \textit{(Esquisse)}
\begin{itemize}
\item[$\bullet$] Comme le noyau de l'invariant de Rost d'un groupe semi-simple simplement connexe absolument presque simple quasi-déployé de type $A_n$ avec $n \leq 5$, $B_n$ avec $n \leq 6$, $C_n$ avec $n \leq 5$, $D_n$ avec $n \leq 6$, $^{1}D_7$, $E_6$, $E_7$, $F_4$ ou $G_2$ est trivial (théorèmes 0.1 et 0.5 de \cite{GS}) et comme tout torseur dans le noyau de l'invariant de Rost d'un groupe semi-simple simplement connexe absolument presque simple quasi-déployé de type $E_8$ sur un corps de dimension cohomologique au plus 4 a un zéro-cycle de degré 1 (\cite{Che1}, \cite{Che2}, \cite{Sem}), on montre la propriété suivante exactement de la même manière que la proposition 6.2 de \cite{HS1}: sous les hypothèses de (i), le noyau de $H^1(K,H^{sc}) \rightarrow \prod_{v \in X^{(1)}} H^1(K_v,H^{sc}))$ est trivial; sous les hypothèses de (ii), tout torseur sous $H^{sc}$ représentant un élément de $\text{Ker}(H^1(K,H^{sc}) \rightarrow \prod_{v \in X^{(1)}} H^1(K_v,H^{sc}))$ possède un zéro-cycle de degré 1. Dans la preuve de ces résultats, l'injectivité de $H^3(K,\mathbb{Q}/\mathbb{Z}(2)) \rightarrow \prod_{v \in X^{(1)}} H^3(K_v,\mathbb{Q}/\mathbb{Z}(2))$ découle de la nullité de $\Sha^4(\mathbb{Z}(2))$.
\item[$\bullet$] On considère une $z$-extension de $H$:
$$1 \rightarrow Q \rightarrow H_z \rightarrow H \rightarrow 1.$$
C'est une extension centrale de $K$-groupes réductifs, $Q$ est un tore quasi-trivial dont le module des caractères est un module libre sur l'anneau $\mathbb{Z}[\text{Gal}(L/K)]$, et le sous-groupe dérivé $H_z^{ss}$ de $H_z$ est $H^{sc}$. Comme $H_z$ est réductif, on dispose aussi d'une suite exacte $1 \rightarrow H^{sc} \rightarrow H_z \rightarrow H_z/H^{sc} \rightarrow 1$ où $H_z/H^{sc}$ est un tore. On prouve alors de la même manière que le lemme 6.4 et la proposition 6.5 de \cite{HS1} le résultat suivant: le morphisme naturel d'ensembles pointés $\Sha^1(H_z)  \rightarrow \Sha^1(H)$ est un isomorphisme, et le noyau du morphisme naturel d'ensembles pointés $\Sha^1(H_z) \rightarrow \Sha^1(H_z/H^{sc})$ est trivial. Pour ce faire, il est nécessaire de montrer que $\Sha^2(Q)$ est nul: cela découle immédiatement du lemme de Shapiro, du fait que $\hat{Q}$ est un $\mathbb{Z}[\text{Gal}(L/K)]$-module libre et de la nullité de $\Sha^2(L,\mathbb{G}_m)$. On notera $E_z$ un torseur sous $H_z$ représentant l'image réciproque de $E$ par l'isomorphisme $\Sha^1(H_z)  \rightarrow \Sha^1(H)$, et $Y$ un torseur représentant l'image de $E_z$ par $\Sha^1(H_z) \rightarrow \Sha^1(H_z/H^{sc})$.
\item[$\bullet$] On note $T$ (resp. $T_z$) un tore maximal de $H$ (resp. $H_z$). On note $T^{(sc)}$ (resp. $T^{(sc)}_z$) l'image réciproque de $T$ (resp. $T_z$) dans $H^{sc}$ (resp. $H_z^{sc}$). On note finalement $G = [T^{(sc)} \rightarrow T]$ et $G_z = [T^{(sc)}_z \rightarrow T_z]$. Le lemme des cinq et le lemme 6.7 de \cite{San} fournissent un isomorphisme $H^1(\overline{E},\mathbb{Q}/\mathbb{Z}(1)) \cong H^1(\overline{H},\mathbb{Q}/\mathbb{Z}(1))$ et donc un isomorphisme $H^1(\overline{E},\mathbb{Q}/\mathbb{Z}(1)) \cong \text{\underline{Hom}}_K(\check{T}/\check{T^{(sc)}},\mathbb{Q}/\mathbb{Z})$ d'après la proposition 6.7 de \cite{CT}. Comme $\check{T^{(sc)}} \rightarrow \check{T}$ est injectif, on obtient ainsi un isomorphisme:
\begin{align*}
H^1(\overline{E},\mathbb{Q}/\mathbb{Z}(d+1)) &\cong \text{\underline{Hom}}_K(\check{T}/\check{T^{(sc)}},\mathbb{Q}/\mathbb{Z}(d)) \\ &\cong H^0\mathbb{R}\text{\underline{Hom}}_K([\check{T^{(sc)}} \rightarrow \check{T}],\mathbb{Q}/\mathbb{Z}(d)).
\end{align*}
On a alors un morphisme naturel $H^0(\tilde{G}_t[-1]) \rightarrow H^1(\overline{E},\mathbb{Q}/\mathbb{Z}(d+1))$ induit par l'accouplement $[\hat{T} \rightarrow \hat{T^{(sc)}}] \otimes^{\mathbf{L}} [\check{T^{(sc)}} \rightarrow \check{T}] \rightarrow \mathbb{Z}[1]$ que nous avons construit dans la preuve du lemme \ref{acc gm}. Comme $\check{T^{(sc)}} \rightarrow \check{T}$ est injectif et le conoyau de $\hat{T} \rightarrow \hat{T^{(sc)}}$ est fini, on a $H^0(\tilde{G}_t[-1]) = \tilde{G}_t [-1] $, et on a une suite exacte:
\begin{align*}
H^{d+1}(K, [\hat{T} \rightarrow \hat{T^{(sc)}}] \otimes^{\mathbf{L}} \mathbb{Q}(d)[-1]) &\rightarrow H^{d+1}(K, \tilde{G}_t[-1]) \rightarrow H^{d+2}(K, \tilde{G}[-1]) \\&\rightarrow H^{d+2}(K, [\hat{T} \rightarrow \hat{T^{(sc)}}] \otimes^{\mathbf{L}} \mathbb{Q}(d)[-1]) .
\end{align*}
Montrons que $H^{d+1}(K, [\hat{T} \rightarrow \hat{T^{(sc)}}] \otimes^{\mathbf{L}} \mathbb{Q}(d)[-1])$ est nul. Pour ce faire, on dispose de la suite exacte $H^{d+1}(K, \text{Ker}(\hat{T} \rightarrow \hat{T^{(sc)}}) \otimes^{\mathbf{L}} \mathbb{Q}(d))\rightarrow H^{d+1}(K, [\hat{T} \rightarrow \hat{T^{(sc)}}] \otimes^{\mathbf{L}} \mathbb{Q}(d)[-1]) \rightarrow H^{d+1}(K, \text{Coker}(\hat{T} \rightarrow \hat{T^{(sc)}}) \otimes^{\mathbf{L}} \mathbb{Q}(d)[-1])$. Le troisième terme est nul car $\text{Coker}(\hat{T} \rightarrow \hat{T^{(sc)}})$ est fini. Quant au premier, il est divisible et, le groupe $H^{d+1}(K',\mathbb{Q}(d))$ étant nul pour chaque corps $K'$, un argument de restriction-corestriction montre qu'il est d'exposant fini. Il est donc nul, et à fortiori le groupe $H^{d+1}(K, [\hat{T} \rightarrow \hat{T^{(sc)}}] \otimes^{\mathbf{L}} \mathbb{Q}(d)[-1])$ l'est aussi.\\
On montre de même que $H^{d+2}(K, [\hat{T} \rightarrow \hat{T^{(sc)}}] \otimes^{\mathbf{L}} \mathbb{Q}(d)[-1])=0$, et on obtient donc un isomorphisme: 
$$H^{d+1}(K, \tilde{G}_t[-1]) \cong H^{d+2}(K, \tilde{G}[-1]),$$
qui permet de construire par composition un morphisme: 
$$H^{d+2}(K, \tilde{G}[-1]) \rightarrow H^{d+1}(K,H^1(\overline{E},\mathbb{Q}/\mathbb{Z}(d+1))).$$
En composant avec le morphisme $H^{d+1}(K,H^1(\overline{E},\mathbb{Q}/\mathbb{Z}(d+1))) \rightarrow H^{d+2}(E,\mathbb{Q}/\mathbb{Z}(d+1))/ H^{d+2}(K,\mathbb{Q}/\mathbb{Z}(d+1))$, on obtient un morphisme:
$$H^{d+1}(K, \tilde{G}) \rightarrow H^{d+2}(E,\mathbb{Q}/\mathbb{Z}(d+1))/ H^{d+2}(K,\mathbb{Q}/\mathbb{Z}(d+1)).$$
En passant aux éléments localement triviaux, cela induit un morphisme:
$$\Sha^{d+1}(\tilde{G}) \rightarrow H^{d+2}_{\text{lc}}(E,\mathbb{Q}/\mathbb{Z}(d+1)).$$
De même, on a des morphismes:
$$\Sha^{d+1}(\tilde{G_z}) \rightarrow H^{d+2}_{\text{lc}}(E_z,\mathbb{Q}/\mathbb{Z}(d+1)),$$
$$\Sha^{d+2}((H_z/H^{sc})^{\sim}) \rightarrow H^{d+2}_{\text{lc}}(Y,\mathbb{Q}/\mathbb{Z}(d+1)).$$
En exploitant le diagramme commutatif:\\
\centerline{\xymatrix{
\Sha^{d+2}((H_z/H^{sc})^{\sim}) \ar[r] \ar[d] & \Sha^{d+1}(\tilde{G_z})\ar[d] & \Sha^{d+1}(\tilde{G})\ar[d]\ar[l]\\
H^{d+2}_{\text{lc}}(Y,\mathbb{Q}/\mathbb{Z}(d+1))\ar[r] &  H^{d+2}_{\text{lc}}(E_z,\mathbb{Q}/\mathbb{Z}(d+1)) & H^{d+2}_{\text{lc}}(E,\mathbb{Q}/\mathbb{Z}(d+1))\ar[l]
}}
et en utilisant le théorème \ref{local-global}, on voit qu'il suffit de montrer que le morphisme $\Sha^{d+1}(\tilde{G}) \rightarrow \Sha^{d+1}(\tilde{G_z})$ est un isomorphisme.\\
Pour ce faire, on écrit le triangle distingué $\tilde{G} \rightarrow \tilde{G_z} \rightarrow \tilde{Q}[1] \rightarrow \tilde{G}[1]$. Comme $Q$ est quasi-trivial, le lemme de Shapiro et la conjecture de Beilinson-Lichtenbaum imposent que $H^{d+1}(K,\tilde{Q})=H^{d+1}(K_v,\tilde{Q})=0$. De plus, en utilisant toujours le lemme de Shapiro et le fait que $\hat{Q}$ est un $\mathbb{Z}[\text{Gal}(L/K)]$-module libre, on a $\Sha^{d+2}(\tilde{Q}) = \Sha^{d+2}(L,\mathbb{Z}(d))^m$ pour un certain $m$, qui est nul d'après le lemme \ref{nul dual} puisque $\Sha^2(L,\mathbb{G}_m)=0$. Par conséquent, le morphisme $\Sha^{d+1}(\tilde{G}) \rightarrow \Sha^{d+1}(\tilde{G_z})$ est bien un isomorphisme, ce qui achève la preuve.
\end{itemize}
\end{proof}

\begin{remarque}
Dans le cas $d=2$, on n'a pas besoin de supposer que $\Sha^4(\mathbb{Z}(2))=0$. En effet, comme $\Sha^2(L,\mathbb{G}_m)=0$, un argument de restriction-corestriction montre que $\Sha^2(\mathbb{G}_m)=0$, et donc, en vertu du lemme \ref{nul dual}, $\Sha^4(\mathbb{Z}(2))$ est automatiquement nul. En particulier, dans ce cas, il suffit de supposer que le corps $L$ vérifie les hypothèses du corollaire \ref{nul 3} ou du corollaire \ref{nul 4}.
\end{remarque}

\nocite*

\end{document}